\documentclass[10pt]{article}
\usepackage{amsmath}
\usepackage{amsfonts}
\usepackage{amssymb}
\usepackage{graphicx}
\usepackage{amsthm}
\usepackage{hyperref}
\usepackage{xcolor}
\usepackage{accents}
\usepackage[numbers,sort&compress,square]{natbib}

\usepackage[english]{babel}
\usepackage [autostyle, english = american]{csquotes}
\MakeOuterQuote{"}

\numberwithin{equation}{section}
\theoremstyle{plain}
\newtheorem*{theorem*}{Theorem}
\newtheorem*{lemma*}{Lemma}
\newtheorem{theorem}{Theorem}
\newtheorem{lemma}{Lemma}[section]
\newtheorem{corollary}[lemma]{Corollary}

\newtheorem{proposition}[lemma]{Proposition}
\newtheorem{innercustomthm}{Theorem}
\newenvironment{customthm}[1]
  {\renewcommand\theinnercustomthm{#1}\innercustomthm}
  {\endinnercustomthm}

\theoremstyle{definition}
\newtheorem{definition}[lemma]{Definition}
\newtheorem{remark}[lemma]{Remark}
\newtheorem{example}[lemma]{Example}
\newtheorem*{remark*}{Remark}
\newtheorem*{example*}{Example}
\newtheorem*{er*}{Examples and Remarks}

\def\pr{\partial}
\def\rd{\partial}
  
\def\V{\Vert}
\def\teta{\tilde{\eta}} 

\usepackage{geometry}
\begin{document}

	
\title{On Singular Vortex Patches, I: Well-posedness Issues}
\author{Tarek M. Elgindi\textsuperscript{1} \and In-Jee Jeong\textsuperscript{2}}
\footnotetext[1]{Department of Mathematics, UC San Diego. E-mail: telgindi@ucsd.edu.}
\footnotetext[2]{School of Mathematics, Korea Institute for Advanced Study. E-mail: ijeong@kias.re.kr}
\date{\today}

\maketitle
\begin{abstract}
	The purpose of this work is to discuss the well-posedness theory of singular vortex patches. Our main results are of two types: well-posedness and ill-posedness. On the well-posedness side, we show that globally $m-$fold symmetric vortex patches with corners emanating from the origin are globally well-posed in natural regularity classes as long as $m\geq 3.$ In this case, all of the angles involved solve a \emph{closed} ODE system which dictates the global-in-time dynamics of the corners and only depends on the initial locations and sizes of the corners.  {Along the way we obtain a global well-posedness result for a class of symmetric patches with boundary singular at the origin, which includes logarithmic spirals.} On the ill-posedness side, we show that \emph{any} other type of corner singularity in a vortex patch cannot evolve continuously in time except possibly when all corners involved have precisely the angle $\frac{\pi}{2}$ for all time. Even in the case of vortex patches with corners of angle $\frac{\pi}{2}$ or with corners which are only locally $m-$fold symmetric, we prove that they are generically ill-posed. We expect that in these cases of ill-posedness, the vortex patches actually cusp immediately in a self-similar way and we derive some asymptotic models which may be useful in giving a more precise description of the dynamics. In a companion work \cite{SVP2}, we discuss the long-time behavior of symmetric vortex patches with corners and use them to construct patches on $\mathbb{R}^2$ with interesting dynamical behavior such as cusping and spiral formation in infinite time.
\end{abstract}
\tableofcontents
\section{Introduction}

\subsection{The notion of vortex patches}

In this paper, we investigate the dynamics of singular vortex patches, which are patch-like solutions to the 2D Euler equations with non-smooth boundaries. We first recall that the 2D Euler equations on $\mathbb{R}^2$, in vorticity form, are given by \begin{equation}\label{eq:2DEuler}
\begin{split}
&\pr_t \omega + (u\cdot\nabla)\omega = 0, 
\end{split}
\end{equation} where at each moment of time, $u$ is determined from $\omega$ by \begin{equation}\label{eq:BS}
\begin{split}
u(t,x) = \frac{1}{2\pi} \int_{ \mathbb{R}^2 } \frac{(x-y)^\perp}{|x-y|^2} \omega(t,y) dy. 
\end{split}
\end{equation} The transport nature of \eqref{eq:2DEuler} suggests that if the initial vorticity $\omega_0(x)$ is given by the characteristic function of a domain $\Omega_0 \subset \mathbb{R}^2$, the solution should take the form of the characteristic function of a domain that moves with time. We shall refer to such a solution as a vortex patch. Indeed, the theorem of Yudovich in \cite{Y1} gives that for any $\omega_0 \in L^1 \cap L^\infty(\mathbb{R}^2)$, there exists a unique solution to \eqref{eq:2DEuler} in the class $\omega(t,x) \in C^0_*(\mathbb{R}; L^1 \cap L^\infty(\mathbb{R}^2) )$ with $\omega(0,x) = \omega_0$, where $C^0_*$ denotes that $\omega(t,\cdot)$ is weak-star continuous in time. It turns out that this regularity is just sufficient to make sense of the flow maps $\Phi(t,\cdot) $ as homeomorphisms of $\mathbb{R}^2$ for all $t \in \mathbb{R}$: the velocity vector field satisfies the following log-Lipschitz estimate \begin{equation*}
\begin{split}
|u(t,x) - u(t,x')| \le C\V \omega_0\V_{L^\infty \cap L^1}|x-x'| \log\left(1 +  \frac{1}{|x-x'|} \right)
\end{split}
\end{equation*} which gives rise to a unique solution to the following ordinary differential equation \begin{equation*}
\begin{split}
\frac{d}{dt}\Phi(t,x) = u(t,\Phi(t,x)),\quad \Phi(0,x) = x.  
\end{split}
\end{equation*} As a particular case, if the initial data is given by $\omega_0(x) = \chi_{\Omega_0}$ for some bounded measurable set $\Omega_0$, the associated unique solution to \eqref{eq:2DEuler} takes the form \begin{equation*}
\begin{split}
\omega(t,x) = \chi_{\Omega(t)}, \quad \Omega(t) = \Phi_t^{-1}(\Omega_0)
\end{split}
\end{equation*} where $\Phi_t^{-1}$ is the inverse of $\Phi(t,\cdot)$. Therefore the following \textit{vortex patch problem} is well-defined: \begin{equation*}
\begin{split}
\text{Given a bounded measurable set } \Omega_0, \text{ what can be said about the sets } \Omega(t) \text{ for  } t \ne 0 ?  
\end{split}
\end{equation*} Before we proceed further, let us point out a simple consequence of the following Yudovich estimate: \begin{equation*}
\begin{split}
|x-x'|^{e^{ct\V\omega_0\V_{L^\infty \cap L^1}}} \le |\Phi(t,x) - \Phi(t,x')| \le |x-x'|^{e^{-ct\V\omega_0\V_{L^\infty \cap L^1}}}
\end{split}
\end{equation*} for all $x, x' \in \mathbb{R}^2$ with $|x-x'| < 1/2$ where $c > 0$ is an absolute constant. It guarantees that, if the boundary of $\Omega_0$ is given by a Jordan curve, this property holds for all of the domains $\Omega(t)$. However, since the estimate deteriorates with time, in general no uniform regularity can be obtained for all $\partial\Omega(t)$. 

Often, a vortex patch could mean the following more general object: a solution of the 2D Euler equations in the form \begin{equation*}
\begin{split}
\omega(t,x) = \sum_{i = 1}^N f_i(t,x) \chi_{\Omega_i(t)}
\end{split}
\end{equation*} where $N \ge 1$ is an integer, $\Omega_i(t)$ are mutually disjoint bounded measurable sets that move with time, and $f_i(t,x)$ are functions describing the profiles of vorticity. In this case, it is reasonable to require that $f_i(t,\cdot)$ is at least continuous on $\Omega_i(t)$. Moreover, the fluid domain could be a bounded domain in $\mathbb{R}^2$, the 2D torus, or some other surface. Unless otherwise stated, we shall restrict ourselves to simple $(N = 1)$ patches on $\mathbb{R}^2$, with the normalization $f_1 \equiv 1$. 

\subsection{Smooth versus singular patches}

Given Yudovich's theorem, it is natural to ask the smooth version of the above vortex problem: that is, if $\partial\Omega_0$ is given by a smooth curve, does this property hold for all $\partial\Omega(t)$? It turns out that the answer is positive: precisely, if  $\partial\Omega_0$ is a $C^{k,\alpha}$ H\"older continuous curve for some $k \ge 1$ and $0<\alpha<1$, then $\partial\Omega(t)$ is $C^{k,\alpha}$-regular for all $t$. In particular, the boundary remains a $C^\infty$-curve for all times if it is so initially.   {This was established first by Chemin \cite{C,C3}.} There are two separate issues for this smooth vortex patch problem, namely propagation of smoothness locally and globally in time. 

Note that even local propagation is non-trivial as $\omega(t) \in L^1\cap L^\infty(\mathbb{R}^2)$ does not give that the corresponding velocity $u(t)$ is Lipschitz in space, which is \textit{necessary} to keep the boundary smooth.\footnote{Actually, $u(t)$ is \textit{never} $C^1$-smooth across the boundary of the patch simply because $\omega = \pr_xu^2-\pr_yu^1$.} What saves us is the following special property of the double Riesz transforms (stated somewhat roughly): \begin{equation*}
\begin{split}
\text{If } \omega \text{ is } C^{k-1,\alpha}\text{-smooth along a } C^{k-1,\alpha}\text{ vector field, then } R_iR_j\omega \text{ has the same property.} 
\end{split}
\end{equation*} Here we need $k \ge 1$ and $0 < \alpha <1$, and $R_i$ denote the Riesz transform with $i, j \in \{ 1, 2\}$. Applying this fact to the case $\omega = \chi_{\Omega}$, we obtain that if the boundary is $C^{k,\alpha}$-smooth, then the velocity field belongs to $C^{k,\alpha}(\overline{\Omega})$ and also to $C^{k,\alpha}(\overline{\mathbb{R}^2\backslash\Omega})$. Here we have taken the closures to emphasize that the $C^{k,\alpha}$-regularity is valid uniformly up to the boundary. This ``frozen-time'' fact alone suffices to show local propagation of the boundary regularity. Note also that as long as the smooth solution exists, the flow maps $\Phi(t,\cdot) : \Omega_0 \rightarrow \Omega(t)$ are actually $C^{k,\alpha}$-regular diffeomorphisms in this case. 

The issue of global regularity,  {which was a subject of debate (\cite{But,DrMc}) and then resolved in \cite{C,BeCo}}, is much more subtle and really hinges on the vectorial nature of the velocity field defined by the 2D Biot-Savart kernel. Still, it is relatively straightforward to obtain the following statement on the propagation of regularity: \begin{equation*}
\begin{split}
\text{If }\partial\Omega_0\text{ is }C^{k,\alpha}\text{-smooth and somehow }\V \nabla u(t,x)\V_{L^\infty([0,T];L^\infty(\mathbb{R}^2))}<\infty,\text{ then }\Omega(t)\text{ is }C^{k,\alpha}\text{ up to time }T. 
\end{split}
\end{equation*} Of course, this is reminiscent of the classical estimate for smooth solutions to the Euler equations:
\begin{equation*}
\begin{split}
\frac{d}{dt}\V \omega(t) \V_{C^{k,\alpha}}  \lesssim_{k,\alpha} \V\nabla u(t)\V_{L^\infty} \V \omega(t) \V_{C^{k,\alpha}}  , 
\end{split}
\end{equation*}
which guarantees that the vorticity retains its initial H\"older regularity as long as the velocity remains Lipschitz. Indeed, in several respects, the regularity theory for smooth patches is parallel to the one for smooth vorticities.

At this point, it is worth emphasizing that the Yudovich theory is not relevant (probably even misleading) for the smooth vortex patch problem (both local and global); the latter is really about the anisotropic regularity statement for certain singular integral transforms. Hence it should not be surprising that even for systems such as the surface quasi-geostrophic equations and the 3D Euler equations, smooth patches can be solved locally in time. The Yudovich theorem only guarantees unique existence of a solution after the potential blow-up time (which does not happen for the 2D Euler equations, anyways). 

The story is completely different for patches without smooth boundaries. Let us even imagine an initial patch whose boundary is completely smooth except at a point where it is no better than $C^1$ (e.g. a slice of pizza). Then in general the corresponding initial velocity will fail to be Lipschitz (which is necessary to propagate regularity), and we are in the Yudovich regime, where the velocity is only log-Lipschitz. Here, let us clarify a theorem of Danchin  {(\cite{Da})}  which shows that for an initial patch with isolated singularities in the boundary (and otherwise smooth), the patch boundary remains smooth away from the trajectories of the singular points by the flow. However, it does not show propagation of piecewise smoothness \textit{uniform} up to each singularity, which may be valid for the initial patch as a slice of pizza does. Indeed, one of our results here shows that \textit{any} uniform regularity strictly better than $C^1$ is instantaneously lost for such a data. Then, of course, the right question is to ask what \textit{exactly} happens, and this is what this work makes progress on.


\subsection{Motivations for vortex patches}

Before we show some explicit computations on vortex patches, let us give a few motivations towards the vortex patch problem in general, with some emphasis on its singular version. The following items are indeed deeply related with each other. 

\begin{itemize}
	\item Vortex patches as idealized physical objects: It is reasonable to use vortex patches to model physical situations where a strong eddy-like motion is observed, e.g. a hurricane. In particular, a motivation for studying patches with corner singularities in aerodynamics is discussed in the introduction of \cite{CS}. For more information, one may consult classical textbooks on vortex dynamics (\cite{Lamb,Saff}). It in particular motivates the study of vortex patches on the 2-sphere (\cite{Dr,DP,PD,SSK}).
	\item Long-time behavior of smooth solutions: Regarding the 2D Euler equations, one of the most important problems is to understand the asymptotic behavior of smooth solutions as time goes to infinity. The strongest conservation law is the $L^\infty$-norm for the vorticity, and it is possible that any higher regularity blows up for $T = +\infty$ (this explicitly happens near the so-called Bahouri-Chemin solution; see \cite{KS,X} and Subsection \ref{subsec:computations} below). Hence $L^\infty$ is the natural space to study the long-time behavior.  
	\item Critical phenomena: The space $L^\infty$ in terms of the vorticity is a critical space, in the sense that the associated velocity field barely fails to be a Lipschitz function in space. This leads to interesting phenomena such as instantaneous cusp/spiral formation which is impossible with Lipschitz velocity fields. Moreover, recently there have been significant progress on understanding the Cauchy problem with critical initial data \cite{BL1,BL2,EJ,EM1,MY1,MY2}. For instance it has been shown that the incompressible Euler equations are \textit{ill-posed} in critical Sobolev and H\"older spaces. The corresponding problem for patches is a (folklore) open problem: what happens to the initial patch whose boundary is exactly $C^k$ or $C^{k-1,1}$ with some $k \ge 1$? Note that, as in \cite{EM1}, the case $k=1$ seems to be much more difficult than the case $k\geq 2$. This is because there is much better control on the velocity field in the latter case.
	\item Construction of special solutions: There has been a lot of interest in constructing solutions of the Euler equations with certain dynamical behavior. In this context, the class of vortex patches provides a whole variety of interesting solutions to the 2D Euler equations. Even in situations where one needs smooth solutions, a strategy that has proven useful is to consider patch solutions with the same dynamics and then try to ``smooth out'' the patch. (See a recent work \cite{CCG} where the authors constructed compactly supported and smooth rotating solutions to the 2D Euler equations.)
	\begin{itemize}
		\item $V$-states: Patches which simply rotate with some constant angular speed are called $V$-states (\cite{Bu,HM,HM2,HMV,DZ,dlHHMV,CCG,CCG2,SS,ST}). One may bifurcate from radial profiles to obtain $m$-fold symmetric $V$-states, and it is expected that in  {certain limiting regimes} one obtains $V$-states with either $90^\circ$ corners or cusps (\cite{WOZ,Ove,XJM}). See \cite{HMW} for recent rigorous progress on this problem. 
		\item Solutions with infinite norm growth: In two dimensions, Sobolev and H\"older norms of smooth Euler solutions can grow at most double exponentially in time. This sharp rate was achieved in the presence of a physical boundary in \cite{KS} by smoothing out the Bahouri-Chemin solution. In terms of vortex patches, the relevant question is whether two disjoint patches can approach each other double exponentially in time as $t \rightarrow +\infty$ (see \cite{Den3}). 
		\item Instantaneous instability: On the other hand, one may ask for initial vorticity configurations which maximize a certain functional (such as palenstrophy); see \cite{AP1,AP2,AP3} and references therein. It seems that in certain cases the maximizer takes the form of a (slightly regularized) vortex patch; the work \cite{EJ} shows this for the case of the $H^1$-norm in terms of the vorticity. 
	\end{itemize}
	In the opposite direction, one may consider patches as smoother alternatives for even more singular constructs, such as vortex sheets or point vortices. The study of singular vortex patches becomes relevant in this regard; for instance, one may take the vanishing angle limit of the patch supported on a sector, keeping the $L^1$-norm  {constant}. In the limit one obtains a sheet with linearly growing intensity from the corner which was numerically studied by Pullin \cite{Pull, Pull2, Pull3}.
\end{itemize}

\subsection{Main results and ideas of the proof}

As we have mentioned earlier, the primary goal in this paper is to understand the dynamics of patches initially supported on either a corner or a union of corners meeting at a point. In one sentence, our conclusion is that such a corner structure propagates continuously in time if and only if the initial patch satisfies an appropriate rotational symmetry condition at the origin, namely $m$-fold symmetry with some $m \ge 3$.\footnote{This is with the exception of special angles $0, \pi/2, \pi$, and $3\pi/2$, which we discuss separately.} We actually show that when such a symmetry condition is satisfied, then the propagation is global in time.

Our main well-posedness result concerns rotationally symmetric patches which have corners meeting at a point {. The main result shows that the uniform regularity of the patch boundary  (up to the corner) propagates for all time. For the economy of presentation, we give a somewhat rough statement here; detailed statements are given in Theorem \ref{thm:symmetric_corner} and Corollary \ref{cor:dynamics_angle}. }
\begin{customthm}{A}\label{mainthm:wellposedness}
	 {Fix some $0 < \alpha <1 $ and consider $\omega_0 = \chi_{\Omega_0}$, where $\Omega_0$ is $m$-fold rotationally symmetric around the origin for some $m \ge 3$, $\partial\Omega_0$ is $C^{1,\alpha}$-smooth away from the origin, and can be mapped by a $C^{1,\alpha}$-diffeomorphism $\Psi_0$ of $\mathbb{R}^2$ to a union of non-intersecting sectors. That is, we have \begin{equation}\label{eq:local_expression}
	\begin{split}
	\Psi_0(\Omega_0) \cap B_0(r_0) = \bigcup_{k = 0}^{m-1} \bigcup_{i = 1}^N \{ (r,\theta) : 0 < r < r_0,   a_{i,0} + 2\pi k/m < \theta < b_{i,0} + 2\pi k/m \}
	\end{split}
	\end{equation} for some $r_0 > 0$. 
	
	Then, the corresponding patch solution $\Omega(t)$ enjoys the same properties for all $t > 0$, with some $C^{1,\alpha}$-diffeomorphism $ \Psi(t)$ and $r(t) > 0$. To be more precise, $\Omega(t)$ is $m$-fold symmetric, $C^{1,\alpha}$-smooth away from the origin, and \begin{equation*}
	\begin{split}
	\Psi(t)(\Omega(t)) \cap B_0(r(t)) = \bigcup_{k = 0}^{m-1} \bigcup_{i = 1}^N \{ (r,\theta) : 0 < r < r(t),   a_{i}(t) + 2\pi k/m < \theta < b_{i}(t) + 2\pi k/m \}. 
	\end{split}
	\end{equation*} 
	
	Moreover, the corner angles of $\Omega(t)$ evolve according to a closed system of ordinary differential equations; in the simplest case of $N = 1$ in \eqref{eq:local_expression}, the corners rotate with a constant angular speed for all time, which is determined only by the initial angle and $m$. }
\end{customthm}

In the statement, ``$C^{1,\alpha}$'' can be replaced by ``$C^{k,\alpha}$'' throughout, for any integer $k \ge 1$. In particular if the initial boundary is uniformly $C^\infty$-smooth up to the corner, the boundary will remain so for all time. A prototypical example of a patch satisfying the assumption above is given by the region  \begin{equation*}
\begin{split}
\{ (r,\theta) : 0 < r < \sin(m\theta) \}
\end{split}
\end{equation*} with some $m \ge 3$; see Figure \ref{fig:petal} for the case $m = 3$. 
\begin{figure}
	\includegraphics[scale=0.4]{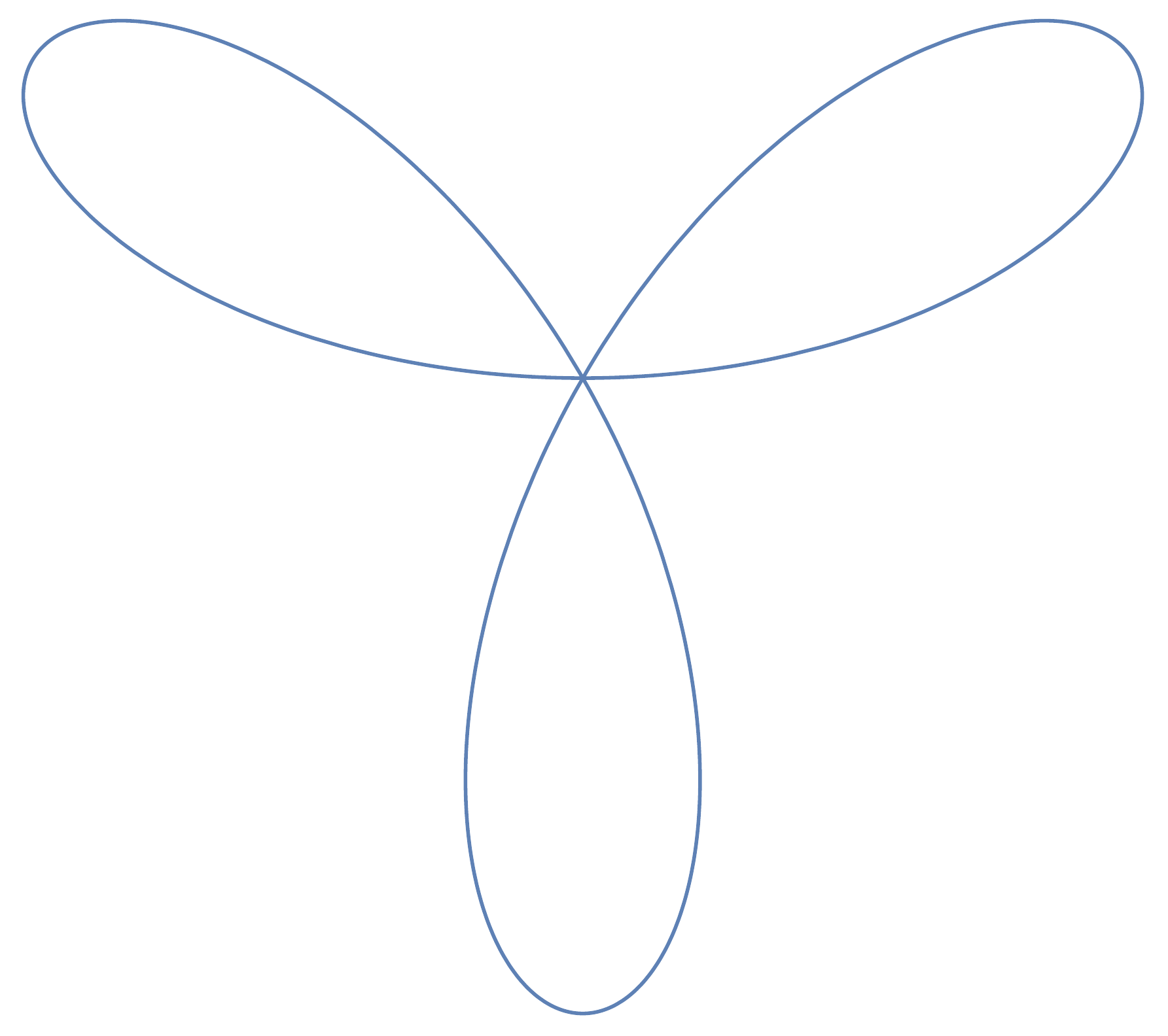} 
	\centering
	\caption{Symmetric corners with smooth boundary}
	\label{fig:petal}
\end{figure} Since $N =1 $ in this case, our result dictates that near the corner, the motion of the patch is given by a uniform rotation for all time. This is completely consistent with the existence of $V$-states which take a similar form as in Figure \ref{fig:petal} reported by numerical analysts (\cite{LF1,LF2}).\footnote{Interestingly such $V$-states can be found numerically by carefully bifurcating from a $V$-state consisting of three chunks of vorticity arranged symmetrically around the origin.} It turns out that the angular speed of rotation is a monotonic function of the initial angle. Therefore, if we perturb the circular patch in $L^1$ so that locally it looks as in Figure \ref{fig:petal}, there is a discrepancy between the speeds of rotation near the corner and at the bulk for all time, from a well-known stability result for the circular patch. Combining this with some topological and measure-theoretic arguments, we conclude infinite in time spiral formation in the companion work \cite{SVP2}. 

Our analysis is not limited to the case of $N = 1$, but also covers the case when there are multiple corners in a fundamental domain of the $m$-fold rotation. For an example, one can consider the domain obtained by the $3$-fold symmetrization of \begin{equation*}
\begin{split}
\ \{ (r,\theta) : 0 < r < \sin(6\theta), 0 < \theta < \pi/6 \} \cup \{ (r,\theta) : 0 < r < 2\sin(6\theta), \pi/3 < \theta < \pi/2 \}.
\end{split}
\end{equation*}   In such cases, the corner angles satisfy an interesting system of ODEs which we briefly study in Subsection \ref{subsec:cusp_formation}. We emphasize that this system is completely closed by itself, so that the local asymptotic shape of the patch for any $t > 0$ is determined from the initial corner angles. 

The statement regarding the angles might be counter-intuitive; after all, strong non-locality in the Biot-Savart kernel of the incompressible Euler equations is its main difficulty. However, consider for instance a radial vorticity $\omega = f(r)$ which is supported away from the origin. Then the velocity near the origin is identically zero; that is, symmetry introduces cancellations. For our purpose, which is to localize the dynamics of the angles, it suffices to guarantee that $u_{far}(x) = o(|x|)$ for $|x| \ll 1$ where $u_{far}$ is the non-local contribution to the velocity. As we will show in this work, it suffices to assume $m$-fold symmetry with $m \ge 3$. 


It turns out that the proof for the local in time statement is rather straightforward, and follows readily from the explicit computations that we shall demonstrate in the next section. Let us give the main points here: For local propagation of regularity, it suffices to establish that the velocity restricted onto the patch boundary is $C^{1,\alpha}$-smooth. However, for a patch given in the statement of Theorem \ref{mainthm:wellposedness}, the corresponding velocity can be considered as a sum of main part coming from exact sectors and remainder associated with cusp regions. The latter component of the velocity is smooth on the boundary. On the other hand, the velocity generated by a symmetric union of exact sectors takes the form $\nabla^\perp(r^2H(\theta))$, with $H \in W^{2,\infty}([0,2\pi))$. The log-Lipschitz part vanishes by symmetry, and it is not hard to see using this explicit expression that it is $C^{1,\alpha}$ along any $C^{1,\alpha}$-curve emanating from the origin. Essentially, this concludes the proof for local well-posedness. 

Unfortunately, the global well-posedness statement for such patches does not seem to follow from a simple adaptation of any of the existing arguments showing global well-posedness for smooth patches. For instance, let us explain the difficulty with respect to the ``geometric'' approach of Bertozzi and Constantin (see Subsection \ref{subsec:Bertozzi_Constantin} below for a brief review of their approach). In this framework, the patch boundary regularity is encoded by a level set function $\phi : \mathbb{R}^2 \rightarrow \mathbb{R}$, characterized by the property that $\phi > 0$ exactly in the interior of the patch. Then, the $C^{1,\alpha}$ norm of $\phi$ is (roughly) associated with the $C^{1,\alpha}$-regularity of the patch boundary, under the condition that $\nabla\phi$ is non-degenerate. Note that if we want such a level set function for the domain in Figure \ref{fig:petal}, $\phi$ certainly cannot be better than Lipschitz. To encode the information that the patch boundary is uniformly piecewise $C^{1,\alpha}$ up to the corner, we need to either give up that $\nabla\phi$ is non-degenerate, or use multiple level set functions to characterize the boundary. None of these variations seemed to work out well.\footnote{However, see a recent work Kiselev, Ryzhik, Yao, and Zlatos \cite{KRYZ} where the authors overcome a similar type of difficulty on the upper half-plane with brute force estimates.}

\begin{figure}
	\includegraphics[scale=0.8]{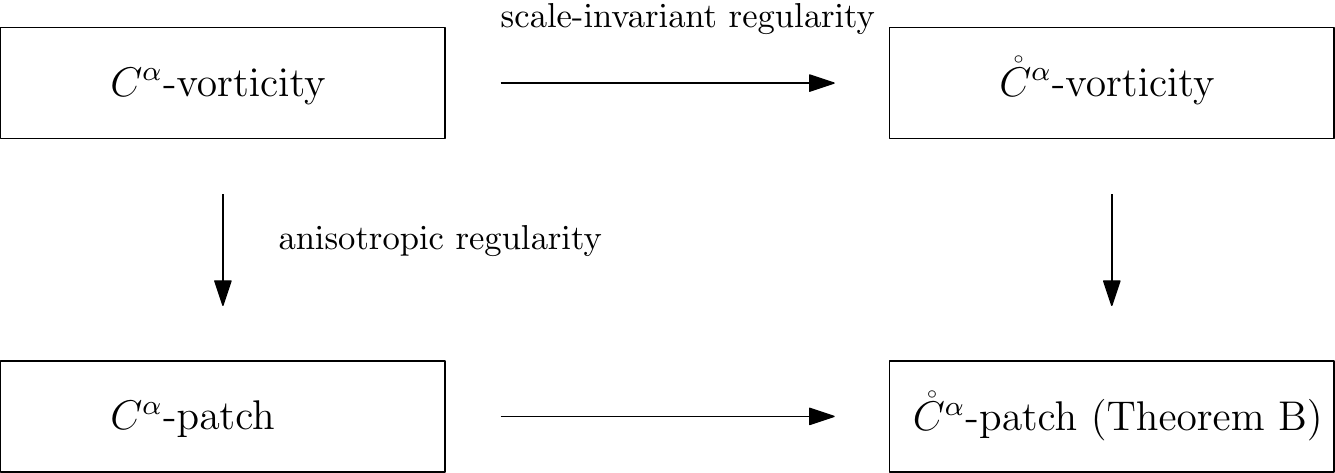} 
	\centering
	\caption{Global in time propagation of H\"older regularity of vorticity}
	\label{fig:square}
\end{figure}

Our approach was to go around this problem by first ``completing the square'' (see Figure \ref{fig:square}) and extract the global-in-time bound on the Lipschitz norm of the velocity from it. This piece of information combined with a Beale-Kato-Majda type argument was sufficient to conclude Theorem \ref{mainthm:wellposedness}. Let us now briefly explain Figure \ref{fig:square}; on the top left side, the classical result on the global well-posedness of $C^\alpha$ vorticity is placed. Then, the vertical and horizontal arrows correspond to the properties of the Euler equations which propagate anisotropic and scale-invariant H\"older regularity of the vorticity, respectively. The latter holds only in the presence of $m$-fold rotational symmetry with $m \ge 3$. The notation $\mathring{C}^\alpha$ was introduced in \cite{EJ1} and encodes scale-invariant $C^\alpha$-regularity; roughly, ``homogeneous'' derivatives $\rd_{\theta}^\alpha \omega$ and $r^\alpha \rd_r^\alpha \omega$ should be bounded, where $\rd_{\theta}^\alpha$ and $\rd_r^\alpha$ denote the $\alpha$-fractional derivative in the angle and radius, respectively. The global well-posedness of $\mathring{C}^\alpha$-vorticity under symmetry was established in \cite{EJ1}, and it is natural to consider the patch version of this result. On the other hand, one can equivalently consider the scale-invariant version of the $C^\alpha$-patch result. This is the content of the following result:
\begin{customthm}{B}\label{mainthm:wellposedness-critical}
	Consider a patch $\Omega_0$ which is $m$-fold symmetric for some $ m \ge 3$ and the piece of boundary at distance $O(r)$ from the origin is $C^{1,\alpha}$-smooth with Lipschitz norm bounded uniformly in $r$ and $C^{1,\alpha}$-norm bounded by $Cr^{-\alpha}$ for some $C > 0$ and $0 < \alpha < 1$. Then the patch solution $\Omega(t)$ retains this property for all $t > 0$.  
\end{customthm}
It is easy to see that the patches considered in Theorem \ref{mainthm:wellposedness} satisfy this condition. Note that the logarithmic spirals  {(e.g. functions of the form $\chi_{\{a<\theta + c\ln r<b\}}$ in polar coordinates for some constants $a,b,c$; see Section \ref{subsec:computations} and Figure \ref{fig:spiral})} satisfy this assumption as well, so that Theorem \ref{mainthm:wellposedness-critical} establishes global-in-time regularity propagation for them. The uniform Lipschitz assumption in Theorem \ref{mainthm:wellposedness-critical} in particular requires that the patch domain is \textit{weakly} Lipschitz, in the sense that near every point $p \in \partial\Omega_0$, there is a bi-Lipschitz map of $\mathbb{R}^2$ sending a neighborhood of $p$ intersected with $\Omega_0$ and $\partial\Omega_0$ to the upper half-plane and the boundary of the upper half-plane, respectively. Indeed, the logarithmic spirals are well-known examples of weakly Lipschitz domains which are not \textit{strongly} Lipschitz (near every point on the boundary, the boundary of the domain is given by the graph of a Lipschitz function); see \cite{AMc, Dac}. Hence, this result shows that even weakly Lipschitz domains propagate its regularity if we assume symmetry and scale-invariant H\"older condition. We shall give more details on the ideas of the proofs in the beginning of Sections \ref{sec:intermediate} and \ref{sec:main}.  

We now state our main ill-posedness result, which states roughly that when the symmetry condition in the above well-posedness statements are not satisfied, then the corner structure is lost immediately. 

\begin{customthm}{C}\label{mainthm:illposedness}
	Assume that $\omega(t)=\chi_{\Omega(t)}$ is a patch-type solution to the 2D Euler equation with a corner singularity whose initial angle is less than $180^\circ$ and propagates continuously in time on some interval $[0,\delta).$ Then, either the corner has angle $90^\circ$ for all $t\in [0,\delta)$ or the vortex patch is locally $m$-fold symmetric with respect to the corner for some $m\geq 3$ for all $t\in [0,\delta)$. Moreover, there exist initially locally $m-$fold symmetric patches and patches with a single $90^\circ$ corner which \emph{do not} propagate continuously in time. 
\end{customthm}

In addition to this, we shall show in Theorem  {\ref{thm:local-symmetry}} that the \textit{exact} $m$-fold symmetry condition is essential even for local well-posedness: for an initial vortex patch which is $m$-fold symmetric with $m \ge 3$ only locally at the origin, it is possible for the velocity to lose Lipschitz continuity immediately. 

Lastly, we discuss the important question of what is the actual dynamics of a corner without any symmetries. In Subsection \ref{subsec:computations} below, we shall carry out some computations for vortex patches supported on cusps and spirals, as possible candidates for describing the evolution of the corner. Let us explain here why we expect the corner to immediately cusp or spiral: To begin with, the passive transport by the initial velocity indicates that the corner rotates $45^\circ$ instantaneously and form a cusp there. However, as soon as this happens, if the vorticity near the point of singularity is ``thick'' enough in the angle, then the new velocity can make the patch rotate even further, up to another $45^\circ$. Then, either this process can go on indefinitely so that the resulting patch has formed an (infinite) spiral, or stop at some point that the patch is just a cusp. The difficulty is that this entire process is supposed to happen exactly at $t = 0$. Therefore, it makes sense to define a new variable incorporating both time and length scales, which rescales the instantaneous behavior of the patch to occur on a non-zero interval in this variable. It turns out that the natural change of variables is to introduce new time variable $\tau = t\ln \frac{1}{r}$. With this variable, we derive a formal evolution equation (a second order system of ordinary differential equations in terms of $\tau$) which is supposed to describe the boundary evolution near the corner at least for a short period of time. This procedure is comparable with introducing a self-similar variable in the study of vortex sheets supported on algebraic spirals. 


\subsection{Historical background} 

The celebrated 1963 theorem of Yudovich \cite{Y1} made it possible to pose the vortex problem, without any regularity assumptions on the patch boundary. Later this well-posedness result was extended in various directions (see e.g. \cite{Coz, CozK,  BH, BK, Y2, E2, Vi1, MR, BMP, S1, S2, AKLN, K1, Ta, TTY, EJ1}). We just note that when the  {patch} satisfies $m$-fold symmetry with some $m \ge 3$, global existence and uniqueness can be proved with just $L^\infty$ of vorticity, which makes it possible to treat patches with non-compact support in $\mathbb{R}^2$ (\cite{EJ1}). 

The dynamics of vortex patches, either numerically or theoretically, are usually considered using the contour dynamics equation (see \cite{ZHR,ZOWZ}), which reduce the 2D dynamics to a 1D evolution equation in terms of the boundary parametrization. It is required that the patch boundary is at least piecewise $C^1$. In the context of 2D Euler patches, the corresponding CDE seems to have first appeared in the work \cite{ZHR} published in 1979, in the context of providing reliable numerical scheme for the 2D Euler equations. In the thesis of Bertozzi \cite{B}, a local well-posedness theorem for smooth vortex patches was proved based on the CDE.  {At the time of that work, the problem of global regularity was open and patches with a corner were investigated therein as a possible candidate for the profile of the patch at the blow-up time.} In the late 80s and early 90s there have been a lot of numerical and theoretical works investigating the possibility of finite time singularity, which seemed highly likely (\cite{But,Maj,CT,A,DrMc,Dr}). 

This issue was settled by Chemin \cite{C,C3} in 1991 who showed global well-posedness using paradifferential calculus. Then several other proofs, based on different arguments, followed \cite{BC,S2}. See also more recent works \cite{BK2,Hu} as well as textbooks \cite{C2,BCD} which cover the proof of global well-posedness. The works of Danchin \cite{Da,Da2} cover global well-posedness for (regular) cusps as well as propagation of patch boundary regularity away from singularities. In the case when the physical domain has a boundary, it is more delicate to propagate regularity globally in time for smooth patches touching the boundary (see \cite{Dep, KRYZ} and references therein). In \cite{HP}, it was shown that a corner supported on the boundary of the Half-plane cusps immediately as $t > 0$. Here, the physical boundary significantly simplifies the analysis -- we revisit this result in the section on illposedness (Section \ref{sec:illposed}). The works \cite{CS, CD} numerically investigate the dynamics of a corner; the pictures suggest that initial angles smaller than $90^\circ$ shrink and those larger than $90^\circ$ expand for $t > 0$.

Many interesting dynamic problems regarding vortex patches are wide open. For patches in 3D and higher, it is certainly a challenging problem to prove whether smooth vortex patches can become singular in finite time. Regarding 2D patches, it is not known whether a (signed) patch can initially have finite diameter and the diameter grows without a uniform bound as $t \rightarrow \infty$. Non-trivial upper bounds on the diameter growth are known (and they are polynomial in time; see e.g. \cite{Mar1,ISG,ILN}). In the case when both signs are allowed, \cite{ISG} shows that the patch diameter can grow linearly in time  {(which is sharp)}. A similar question can be asked for the perimeter. In contrast with the diameter case, there is a possibility for a patch with rectifiable boundary to instantaneously lose this property for $t > 0$. However, the result \cite{Kim} which gives upper bounds on the growth of the Dirichlet eigenvalues for the Laplacian with little assumption on the boundary regularity suggests that such a behavior is unlikely. The study of patches with $90^\circ$ corners are left out in this work (but for an ill-posedness result, see Proposition \ref{prop:90}). As we have seen in the above, the difficulty is that the log-Lipschitz part of velocity only exists in the direction tangent to the patch boundary. It would be interesting to rigorously show existence of not only rotating patches with $90^\circ$ corners but also translating ones with an odd symmetry (see figures from \cite{LF3} and references therein).

\subsection{Outline of the paper}

The rest of this paper is organized as follows. The notations that we use throughout the paper are collected in the first subsection of Section \ref{sec:classical} which is followed by some useful explicit computations and a brief review of results that are necessary to the proof of our well-posedness results in Sections \ref{sec:intermediate} and \ref{sec:main}. Then in Section \ref{sec:intermediate}, we prove Theorem \ref{mainthm:wellposedness-critical}, which is global well-posedness for symmetric patches whose boundaries are $C^\alpha$-smooth in the angle. Using this result together with a tedious local calculation, we conclude Theorem \ref{mainthm:wellposedness} in Section \ref{sec:main}. Possible extensions to this main result are sketched at the end of that section. Ill-posedness results, including Theorem \ref{mainthm:illposedness}, are proved in Section \ref{sec:illposed}. Finally in Section \ref{effective}, we formally write down the effective system which describes the dynamics of the patch with a single corner. The necessary local well-posedness results for symmetric patches that we consider in Sections \ref{sec:intermediate} and \ref{sec:main} are proved in the Appendix for completeness. We emphasize that the work consists of two \emph{different} results whose proofs are independent of each other: well-posedness and ill-posedness. As such, a reader interested in the well-posedness results may focus solely on Sections \ref{sec:classical}, \ref{sec:intermediate}, and \ref{sec:main} while a reader interested mainly in the ill-posedness results may read Sections \ref{sec:classical}, \ref{sec:illposed}, and \ref{effective}.

\section*{Acknowledgments} 

I.-J. Jeong thanks T. Yoneda and N. Kim for their interest in this work and numerous conversations. T.M. Elgindi thanks N. Masmoudi and A. Zlato\v{s} for helpful discussions. T.M. Elgindi was partially supported by NSF DMS-1817134. I.-J. Jeong has been supported by the POSCO Science Fellowship of POSCO TJ Park Foundation and the National Research Foundation of Korea (NRF) grant (No. 2019R1F1A1058486). 

\section{Background Material}\label{sec:classical}

This section goes through some useful background material for the benefit of the reader. We begin by going through a few simple computations which give the reader a sense of the difficulties associated with vortex patches in general and singular vortex patches in particular. Then we discuss two prior works which are important to know: Chemin's global well-posedness result for smooth vortex patches, particularly the proof of Bertozzi-Constantin and our previous result on scale-invariant H\"older regularity for $m$-fold symmetric solutions to 2D Euler. 

\subsection{Notations and definitions}

Let us collect a few definitions and conventions that will be used throughout the paper. 

\begin{itemize}
	\item For $\theta \in [0,2\pi)$, we let $R_{\theta}$ be the matrix of counterclockwise rotation around the origin by the angle $\theta$. Using this notation, we say that a scalar-valued function $f : \mathbb{R}^2 \rightarrow \mathbb{R}$ (e.g. vorticity, level-set function, stream function) is $m$-fold symmetric if $f(x) = f(R_{2\pi/m}x)$ for any $x \in \mathbb{R}^2$. On the other hand, a vector field $v : \mathbb{R}^2 \rightarrow \mathbb{R}^2$ (e.g. velocity, flow maps) is $m$-fold symmetric if $v(R_{2\pi/m}x) = R_{2\pi/m} v(x)$.
	
	\item Given a vector $f = (f_1,f_2)$, we denote the counterclockwise $90^\circ$ rotation by $f^\perp = (-f_2,f_1)$. Similarly, $\nabla^\perp \phi = (-\pr_2\phi, \pr_1\phi)$ for a scalar function $\phi : \mathbb{R}^2 \rightarrow \mathbb{R}$. 
	
	\item The classical H\"older spaces are defined as follows: for $0 < \alpha \le 1$, \begin{equation*}
	\begin{split}
	\V f \V_{C^\alpha(\overline{U})} &= \V f\V_{L^\infty(U)} + \V f\V_{{C}_*^\alpha(\overline{U})} \\
	&= \sup_{x \in U} |f(x)| + \sup_{x \ne x'} \frac{|f(x) - f(x')|}{|x-x'|^\alpha}.
	\end{split}
	\end{equation*} We shall often use the ``inf'', defined by \begin{equation*}
	\begin{split}
	\V f \V_{\inf(F)} = \inf_{x \in F} |f(x)|. 
	\end{split}
	\end{equation*}
	\item We say that $\omega$ is a patch if it is a characteristic function on some (open) set $\Omega \in \mathbb{R}^2$. It will be assumed that the boundary $\partial\Omega$ is either a Jordan curve, or a union of a few Jordan curves intersecting only at the origin. We often identify the function $\omega$ with the set $\Omega$. 
	\item We denote the Biot-Savart kernel as \begin{equation*}
	\begin{split}
	K(x) = \frac{1}{2\pi} \frac{x^\perp}{|x|^2},
	\end{split}
	\end{equation*} and $\nabla K$ as its gradient. Convolution against $\nabla K$ is defined in the sense of principal value integration. 
	\item For functions depending on time and space, we write $f(t,\cdot) = f_t(\cdot)$. The latter notation is not to be confused with the partial derivative in time, which we always denote as $\pr_t$. 
	\item The flow $\Phi$ is defined as a map $[0,\infty) \times \mathbb{R}^2 \rightarrow \mathbb{R}^2$. For each fixed $t \ge 0$, $\Phi(t,\cdot) = \Phi_t$ is a homeomorphism of $\mathbb{R}^2$ whose inverse is denoted by $\Phi_t^{-1}$.
	\item A point in $\mathbb{R}^2$ is denoted by $x = (x_1,x_2)$ or by $y = (y_1,y_2)$. Often we slightly abuse notation and consider polar coordinates $(r,\theta)$, where $r = |x|$ and $\theta = \arctan(x_2/x_1)$. 
	\item Given $x \in \mathbb{R}^2$ and $r > 0$, we define $B_x(r) = \{ y  \in \mathbb{R}^2 : |x-y| < r \}$.
	\item Given two angles $0 \le \theta_1 < \theta_2 < 2\pi$, we define the sector  \begin{equation*}
	\begin{split}
	S_{\theta_1,\theta_2} = \{ (r,\theta) : \theta_1 < \theta < \theta_2 \}.
	\end{split}
	\end{equation*}
	
	 {  	\item Given $f:  \mathbb{R}^n \rightarrow \mathbb{R}^m$ and $g : \mathbb{R}^m \rightarrow \mathbb{R}^k$, we define the composition of $g$ and $f$ by $g \circ f(x) = g(f(x))$ as a map $\mathbb{R}^n\rightarrow\mathbb{R}^k$ . }
	
\end{itemize}

As it is usual, we use letters $C, c, \cdots$ to denote various positive absolute constants whose values may vary from a line to another (and even within a line). Moreover, we write $A \lesssim B$ if there is an absolute constant $C > 0$ such that $A \le CB$. We also use $A \approx B$ when we have $A \lesssim B$ and $A \gtrsim B$. We fix some value of $0 < \alpha <1$ throughout the paper, and the constants $C, c$ may depend on $\alpha$ as well.

\subsection{A few explicit computations}\label{subsec:computations}

In this subsection, we perform some simple computations which already illustrate key issues related to the vortex patch problem. 

\subsubsection*{Case of the disc} Consider the patch supported on the unit disc. Then, using the Biot-Savart law (it is much easier to use its radial version), one can explicitly compute that the corresponding velocity is given by \begin{equation*}
\begin{split}
u(x) = \begin{cases}
\frac{1}{2}x^\perp & \mbox{ if } |x| \le 1, \\
\frac{1}{2}\frac{x^\perp}{|x|^2} & \mbox{ if } |x| \ge 1 . 
\end{cases}
\end{split}
\end{equation*} Note that in the regions $\{ x : |x| \le 1 \}$ and $\{ x : |x| \ge 1 \}$, the velocity is $C^\infty$-smooth, respectively. 

This simple computation can be used as a basis for the following general result mentioned earlier in the introduction: for a patch $U$ bounded by a $C^{k,\alpha}$-curve, the velocity is a $C^{k,\alpha}$ function inside the patch. The point is that there exists a $C^{k,\alpha}$-diffeomorphism of the plane $\Psi$ which maps the unit disc onto $U$. Then after a change of variables, each component of $\nabla^k\nabla^\perp\Delta^{-1}\chi_U$ has an explicit integral representation involving derivatives of $\Psi$ in the unit disc. Working directly with this expression, the H\"older estimate can be achieved. We leave the details of this (tedious) computation to the interested reader. 

\subsubsection*{Bahouri-Chemin solution} On the torus $\mathbb{T}^2 = [-1,1)^2$, $\omega(x) = \mathrm{sgn}(x_1)\mathrm{sgn}(x_2)$ defines a stationary patch solution, which is often called the Bahouri-Chemin solution after the work \cite{BC}. Stationarity follows since the particle trajectories cannot cross the axes by odd symmetry. Here we consider a configuration in $\mathbb{R}^2$ which is odd with respect to both variables $x_1, x_2$ and  given by 
$\mathrm{sgn}(x_1)\mathrm{sgn}(x_2)$ near the origin, say on the unit disc for concreteness. It is well-known that the associated velocity field is only log-Lipschitz at the origin. For a computation based on Fourier series, see \cite{Den2}. Here we present a simpler way to see it by working in polar coordinates. To begin with, observe that the Bahouri-Chemin solution can be locally written as \begin{equation*}
\begin{split}
\omega(x) = \sum_{ k \ge 0} \frac{\sin(2(2k+1)\theta)}{2k+1} ,\qquad |x| \le 1 
\end{split}
\end{equation*} where $\tan(\theta) = x_2/x_1$ and therefore to compute $u = \nabla^\perp\Delta^{-1}\omega$, it suffices to invert the Laplace operator for functions $\sin(m\theta)$. Note that \begin{equation*}
\begin{split}
\Delta(r^2\sin(m\theta)) = \left(\rd_r^2 + \frac{1}{r}\rd_r + \frac{1}{r^2}\rd_\theta^2\right) (r^2\sin(m\theta)) = (4 - m^2) \sin(m\theta). 
\end{split}
\end{equation*} From this computation, it can be argued that for $m \ge 3$, \begin{equation*}
\begin{split}
\Delta^{-1}(\sin(m\theta)) = -\frac{1}{m^2 -4 } r^2\sin(m\theta). 
\end{split}
\end{equation*} (Strictly speaking $\sin(m\theta)$ on both sides needs to be appropriately truncated for $r \ge 1$.) Therefore, from straightforward estimates one can show that  \begin{equation*}
\begin{split}
\nabla^2\Delta^{-1} \left( \sum_{ k \ge 1} \frac{\sin(2(2k+1)\theta)}{2k+1}   \right) = -\nabla^2 \left( r^2 \sum_{ k \ge 1}\frac{1}{4(2k+1)^2 -4 } \frac{\sin(2(2k+1)\theta)}{2k+1} \right)
\end{split}
\end{equation*} is summable; that is, the corresponding velocity field is Lipschitz continuous. On the other hand, for $m = 2$ we have instead \begin{equation*}
\begin{split}
\Delta\left(r^2\ln\frac{1}{r}\sin(2\theta)\right) = -4\sin(2\theta), 
\end{split}
\end{equation*} so that \begin{equation*}
\begin{split}
\rd_{x_1}\rd_{x_2} \Delta^{-1} (\sin(2\theta)) = -\frac{1}{4} \ln\frac{1}{|x|} + \mbox{bounded}. 
\end{split}
\end{equation*} We conclude that $\rd_{x_1}u_1 = -\rd_{x_2}u_2$ are divergent logarithmically at the origin. In particular on the separatrices $\{ x_1 = 0 \}$ and $\{ x_2 = 0 \}$, this stationary velocity produces double exponential in time contraction and expansion, respectively. 


\subsubsection*{Patches supported on sectors} Generalizing the previous computation, we perform a similar calculation for patches that are locally supported on a union of sectors, which are the main object of study in this work. Explicit computations have appeared in several places (e.g. \cite{B, CS}) but again we provide a shortcut using polar coordinates. The arguments here which might seem formal can be justified either using directly the Biot-Savart kernel or arguments based on the uniqueness of $\Delta^{-1}$. 

We consider vorticity which takes the form $\omega(x) = h(\theta)$ for $|x| \le 1$, with some bounded function $h$ of the angle. Taking in particular $h$ to be the characteristic function on a union of intervals in $[0,2\pi)$, we obtain a vortex patch supported on a union of sectors meeting at the origin. The computations below go through for any function $h$. In view of the above, we know that the inverse Laplacian of a bounded function of $\theta$ may involve a logarithm. This suggests us to prepare an ansatz \begin{equation*}
\begin{split}
\Delta^{-1} h = r^2 H(\theta) + r^2 \ln \frac{1}{r} G(\theta),
\end{split}
\end{equation*} where $H$ and $G$ are functions to be determined. Here, we are neglecting possible constant and linear terms on the right hand side (this can be justified for instance when $h$ has some symmetries), which does not affect the velocity gradient in any essential way. Then, \begin{equation*}
\begin{split}
h &= \Delta \left( r^2 H(\theta) + r^2 \ln \frac{1}{r} G(\theta) \right) \\
& = 4H + H'' + (4G + G'')\ln\frac{1}{r} - 4G. 
\end{split}
\end{equation*} This forces $4G + G'' = 0$, or $G = c \cos(2\theta) + s\sin(2\theta)$ for some constants $c, s$, which are determined by multiplying both sides of the above equation by $\cos(2\theta)$ and $\sin(2\theta)$, respectively and integrating on $[0,2\pi)$: \begin{equation*}
\begin{split}
\begin{pmatrix}
c \\
s
\end{pmatrix} = -\frac{1}{4\pi}\int_0^{2\pi} h(\theta)\begin{pmatrix}
\cos(2\theta) \\
\sin(2\theta)
\end{pmatrix} d\theta .
\end{split}
\end{equation*}
Then $H$ is determined uniquely by \begin{equation*}
\begin{split}
(I + \rd_{\theta\theta})^{-1}(h + 4G) = H ,
\end{split}
\end{equation*} which is well-defined since $h + 4G$ is orthogonal to $\cos(2\theta)$ and $\sin(2\theta)$ (see \cite{EJ1} for a proof). An explicit kernel expression for $H$  {has} been derived in \cite{EJ1}; see also Section 2. Note that the velocity gradient coming from $r^2 H$ is bounded. Therefore, we conclude that \begin{equation}\label{eq:velgrad_sectors}
\begin{split}
\nabla u (x) = \begin{pmatrix}
\rd_{x_1} u_1 & \rd_{x_2} u_1 \\
\rd_{x_1} u_2 & \rd_{x_2} u_2 
\end{pmatrix} = \ln\frac{1}{|x|} \begin{pmatrix}
- 2s & 2c \\
2c & 2s
\end{pmatrix} + \mbox{bounded}. 
\end{split}
\end{equation} Interestingly, the non-Lipschitz part of the velocity is simply a constant multiple of the log-linear function, where the constant is determined only by the second-order Fourier coefficients of the vorticity profile. In particular, if $h$ has zero second-order coefficients, then the corresponding velocity gradient is bounded! This happens when $h$ is $m$-fold symmetric with some $m \ge 3$, but this is certainly not a necessary condition; for example one can take $h = \chi_{[-\theta_0,\theta_0]} + \chi_{[\pi/2-\theta_0,\pi/2+\theta_0]}$ (see \cite{E1} for a necessary and sufficient condition). 

We also compute the eigenvectors for the gradient matrix, which correspond to the separatrices generated by the flow: \begin{equation*}
\begin{split}
\begin{pmatrix}
-c \\
s + \sqrt{s^2 + c^2}
\end{pmatrix}, \qquad \begin{pmatrix}
s + \sqrt{s^2 + c^2} \\
c
\end{pmatrix},
\end{split}
\end{equation*} which are orthogonal to each other. 

Now, let us take the concrete case of $h(\theta) = \chi_{[-\theta_0,\theta_0]}$ for some $0 < \theta_0 < \pi/2$. Then, $s = 0$ and $c = \frac{1}{4\pi} \sin(2\theta_0)$. The separatrices are always given by the diagonals $\{ x_1 = x_2 \}$ and $\{ x_1 = -x_2\}$, independent of $\theta$. On the other hand, if we 2-fold symmetrize $h$, that is, take $\chi_{[-\theta_0,\theta_0]} + \chi_{[\pi-\theta_0,\pi + \theta_0]}$, then $s = 0$ again and $c$ is simply multiplied by 2. This shows that the effect of 2-fold symmetrization is just rescaling time by 2, modulo the effect of the bounded term, which is negligible for $|t|, |x| \ll 1$. 

Note that for $0 < \theta_0 < \pi/2$, the log-Lipschitz part of the velocity which is normal to the patch boundary vanishes only for $\theta_0 = \pi/4$, which corresponds to the $90^\circ$-corner. It gives some possibility for a patch with $90^\circ$-corners to retain its shape, which happens explicitly for the Bahouri-Chemin solution and conjecturally happens for certain $V$-states (\cite{WOZ,Ove}).

\subsubsection*{Cusps and logarithmic spirals}

Lastly, we consider vortex patches supported on cusps and logarithmic spirals. For us, a cusp (naively) refers to the region bounded by two $C^1$ curves which meet at a point with the same tangent vectors. By a logarithmic spiral, we mean a spiral where the distances between the turns are related by a geometric progression. These objects are not only of significant interest by themselves, but they are particularly relevant for our study as main candidates which would describe the evolution of a single corner. Indeed, instantaneous cusping or (logarithmic) spiraling of a corner is possible under the flow by a log-Lipschitz velocity. If one consider the passive transport of the patch supported on a corner (whose angle is between $0$ and $\pi/2$) by the flow associated with the initial velocity, then the corner immediately becomes a cusp tangent to one of the separatrices. On the other hand, one may transport the corner with the time-independent velocity $v(x) = x^\perp \log|x|$, and this would cause the corner to immediately become a logarithmic spiral.

An elementary but important fact regarding cusps is that the associated velocity is smooth as long as two pieces of the boundary curves meeting at the cusping point are smooth. To see this, for concreteness we take a patch $\Omega$ whose intersection with a small square $\Omega \cap (-\delta,\delta)^2$ is given by the region \begin{equation*}
\begin{split}
\{ (x_1,x_2) : 0 < x_1 < \delta,  g(x_1) < x_2 < f(x_1) \}
\end{split}
\end{equation*} with some $C^{1,\alpha}$ functions $g, f $ on $[0,\delta]$ satisfying $g'(0) = f'(0) = 0$. Then, locally near the origin, $\Omega \cap (-\delta,\delta)^2= A \backslash B$, where \begin{equation*}
\begin{split}
A = \{ (x_1, x_2) : -\delta < x_1 \le 0, x_2 < 0  \} \cup  \{ (x_1, x_2) :0 < x_1 \le  \delta , x_2 < f(x_1)  \}
\end{split}
\end{equation*} and \begin{equation*}
\begin{split}
B =  \{ (x_1, x_2) : -\delta < x_1 \le 0 , x_2 \le  0  \} \cup  \{ (x_1, x_2) : 0 < x_1 \le  \delta , x_2 \le  g(x_1)  \} .
\end{split}
\end{equation*} Then, the domains $A$ and $B$ have $C^{1,\alpha}$ boundaries on $[-\delta,\delta]^2$, so that the velocities associated with $\chi_A$ and $\chi_B$ are $C^{1,\alpha}$ near the origin in the interior of their respective domains. Hence the velocity of $\chi_{\Omega}$ is $C^{1,\alpha}$ in the interior of $\Omega \cap [-\delta/2,\delta/2]^2$, uniformly up to the boundary. 

Similarly, it follows that the velocity $\nabla^\perp\Delta^{-1}\chi_U$ is $C^{k,\alpha}(\bar{U})$ near the origin if $g, f \in C^{k,\alpha}$ for some $k \ge 1$ and $0 < \alpha <1$. If the preceding argument is somehow not convincing, one can compute explicitly the Biot-Savart kernel for $\nabla^\perp\Delta^{-1}\chi_U$ by first integrating out the second coordinate variable and check directly that the resulting function is $C^{1,\alpha}$. Indeed, we carry out such a computation (in a more complicated setting of a patch consisting of a corner with cusps attached to its sides) in Section \ref{sec:main} in the course of proving our main well-posedness result. 

In the meanwhile, this frozen-time argument already shows that the $C^{1,\alpha}$-cusps should be (at least) locally well-posed: in particular, there is no hope of starting from a corner and immediately becoming a $C^{1,\alpha}$-cusp. Actually such cusps are globally well-posed, and we sketch the argument below in Section \ref{sec:classical}. Therefore, while a cusp may be considered as a singularity, as long as the boundary curves are smooth, the corresponding vortex patch will not lose any regularity in time. 

We now turn to spirals. For simplicity we consider (locally) self-similar spirals with some bounded profile $h$: using polar coordinates, we can write \begin{equation*}
\begin{split}
\omega(r,\theta) = h(c\ln\frac{1}{r} + \theta),\qquad r \le 1
\end{split}
\end{equation*} for some $h\in L^\infty([0,2\pi))$ and $c > 0$.  {In the particular case when $h$ is a characteristic function of the interval, say $h = \chi_{[a,b]}$, then note that $\omega = \chi_\Omega $ where $\partial\Omega$ consists of two logarithmic curves \begin{equation*}
	\begin{split}
	\theta = a - c\ln \frac{1}{r}, \quad \theta = b - c\ln \frac{1}{r}.
	\end{split}
	\end{equation*} See Figure \ref{fig:spiral} for the case when $h$ is a characteristic function of three intervals.} Let us compute the velocity associated with it. We proceed formally by taking the following ansatz for the stream function: \begin{equation*}
\begin{split}
\Psi = r^2 H(c\ln\frac{1}{r}+\theta).  
\end{split}
\end{equation*} Then the relation $\Delta\Psi = \omega$ gives that \begin{equation*}
\begin{split}
4H - 4cH' + (1 + c^2) H'' = h. 
\end{split}
\end{equation*} It is not difficult to show that there exists a unique solution $H \in L^2([0,2\pi))$ (e.g. using Fourier series) and it actually belongs to $W^{2,\infty}([0,2\pi))$. Using $u_r := u \cdot e_r = -\frac{1}{r}\rd_\theta\Psi$ and $u_\theta = u \cdot e_\theta = -\rd_r\Psi$, we deduce that \begin{equation*}
\begin{split}
u_r(r,\theta) = rH'(c\ln\frac{1}{r} + \theta),\qquad u_\theta(r,\theta) = (2rH - crH')\circ(c\ln\frac{1}{r} + \theta)
\end{split}
\end{equation*} near $r = 0$. Taking another $\rd_r$ and $r^{-1}\rd_\theta$, it follows in particular that $u$ is indeed Lipschitz continuous. This fact could be somewhat surprising, especially since when one takes the limit $c \rightarrow 0^+$ in the above, then we are back to the radially homogeneous case where the velocity is explicitly log-Lipschitz. This transition can be seen in terms of the Biot-Savart kernel: the cancellations introduced by the spiral removes the logarithmic divergence of the gradient.  

Moreover, using the above ansatz for $\Psi$, we may write down a closed evolution equation in terms of $h$: \begin{equation*}
\begin{split}
\rd_{t} h + 2H \rd_{\theta} h = 0, \qquad 4H - 4cH' + (1 + c^2) H'' = h.
\end{split}
\end{equation*} The conservation of $ {\|h\|_{L^\infty}}$ ensures that this is globally well-posed. Note that in stark contrast to the radially homogeneous case, whose evolution equation for $h$ requires rotational symmetry, no such assumption is needed for the spiral case. While showing global well-posedness for the logarithmic spirals rigorously may take some work (we achieve this in the presence of rotational symmetry in Section \ref{sec:intermediate}), this is very plausible as the above ansatz for the stream function is correct modulo  {a} perturbation which is $C^\infty$ near the origin. In the end, it suggests that a corner cannot become a logarithmic spiral, and even if it spirals, the turns should be sparser than those of a logarithmic spiral. 

 { Note that instead of taking $h$ to be a bounded function, we may take it as a signed measure. Since $H$ is two orders more regular than $h$, we have that $H$ is Lipschitz continuous, which can be used to show that there is still a local-in-time unique measure-valued solution to \begin{equation*} 
	\begin{split}
	& \rd_t h + 2H\rd_\theta h = 0.
	\end{split}
	\end{equation*} In particular, one can take $h_0$ to be a finite sum of Dirac deltas, possibly with weights. Then this corresponds to a vortex sheet supported on logarithmic spirals which are famously known as Alexander spirals \cite{Alex}. Hence we have obtained, in a very simple manner, the evolution equation corresponding to multi-branched Alexander spirals (see Kaneda \cite{Kan} and also a very recent work of Elling-Gnann \cite{EG}). }

\subsection{Smooth vortex patches: approach by Bertozzi and Constantin}\label{subsec:Bertozzi_Constantin}

In this subsection, let us provide a brief outline of the elegant proof of Bertozzi and Constantin \cite{BeCo} on global regularity of smooth vortex patches. We restrict ourselves to domains $\Omega$ (bounded open set in $\mathbb{R}^2$) which has a level set $\phi : \mathbb{R}^2 \rightarrow \mathbb{R}$ such that: \begin{itemize}
	\item We have $\phi(x) > 0$ if and only if $x \in \Omega$ (hence $\phi$ vanishes precisely on $\partial\Omega$).
	\item The tangent vector field of $\phi$ satisfies $\nabla^\perp \phi \in C^{\alpha}(\mathbb{R}^2)$. 
	\item The function $\phi$ is non-degenerate near $\partial \Omega$, i.e., $ \V \nabla^\perp\phi\V_{\inf(\partial\Omega)}:= \inf_{x \in \partial \Omega} |\nabla^\perp\phi| \ge c > 0$. 
\end{itemize}
Then we say that the patch $\Omega$ is $C^{1,\alpha}$-regular, or a $C^{1,\alpha}$-patch. Given such a $\phi$, we associate the following characteristic quantity: \begin{equation*}
\begin{split}
\Gamma = \left( \frac{\V \nabla^\perp\phi\V_{C^\alpha_*(\mathbb{R}^2)}}{\V\nabla^\perp \phi\V_{\inf(\partial\Omega)}} \right)^{1/\alpha},
\end{split}
\end{equation*} which quantifies the $C^{1,\alpha}$-regularity of $\Omega$.  {Recall that $C^\alpha_*$ denotes the homogeneous $C^\alpha$-norm: \begin{equation*}
	\begin{split}
	\V f \V_{C^\alpha_*} = \sup_{x \ne x'} \frac{|f(x)-f(x')|}{|x-x'|^\alpha}. 
	\end{split}
	\end{equation*}} Note that it has units of inverse length, so that $\Gamma^{-1}$ provides a $C^{1,\alpha}$-characteristic length scale for $\Omega$. An alternative way of defining $C^{1,\alpha}$ patches is to require that, for any point $x \in \partial \Omega$, there exists a ball $B_x(r)$ with some radius $r > 0$ uniform over $x$ such that the intersection $B_x(r) \cap \partial U$ is given by the graph of a $C^{1,\alpha}$ function, after rotating the patch if necessary. Indeed, given $\Gamma$, one may take $r$ to be $1/(10\Gamma)$ and vice versa; given $r > 0$ for each $x \in \partial\Omega$, one may construct a level set function $\phi$.

Taking the initial vorticity to be the characteristic function $\omega_0 = \chi_{\Omega}$, we may denote its unique solution by $\chi_{\Omega_t}$. Since the vorticity is simply being transported by the flow, once we define the evolution of $\phi$ via \begin{equation}\label{eq:evolv_phi}
\begin{split}
\pr_t \phi + (u\cdot \nabla) \phi = 0, 
\end{split}
\end{equation} then it follows that \begin{equation*}
\begin{split}
\phi(t,x ) > 0 \quad \mbox{ if and only if }  \quad x \in  \Omega_t. 
\end{split}
\end{equation*}
To show that $\Omega_t$ stays as a $C^{1,\alpha}$-patch for all times, it suffices to establish an a priori bound on $\Gamma_t$. In Bertozzi-Constantin \cite{BeCo}, the authors have provided a proof that $\Gamma_t$ remains bounded for all time, based on the following two ``frozen-time'' lemmas:
\begin{lemma}[$L^\infty$-bound on $\nabla u$]
	Consider the velocity $u(x) = K*\chi_\Omega (x)$, where $\Omega$ is a $C^{1,\alpha}$-patch with a level set $\phi$. Then, we have a bound \begin{equation}\label{eq:geometric_lemma}
	\begin{split}
	\V \nabla u\V_{L^\infty(\mathbb{R}^2)} \le C \left( 1 + \log\left( 1+ \frac{\V \nabla^\perp\phi\V_{{C}^\alpha(\mathbb{R}^2)}}{\V \nabla^\perp\phi\V_{\inf(\partial\Omega)}} \right) \right).
	\end{split}
	\end{equation}
\end{lemma}

\begin{lemma}[Directional $C^\alpha$-bound on $\nabla u$]
	We have a pointwise identity \begin{equation}\label{eq:key_identity}
	\begin{split}
	\nabla u \nabla^\perp\phi (x) = \frac{1}{2\pi} \int_{\Omega}  \nabla K(x-y) \left( \nabla^\perp\phi(x) - \nabla^\perp\phi(y) \right) dy,
	\end{split}
	\end{equation}
	and in particular, this gives a bound \begin{equation}\label{eq:key_estimate}
	\begin{split}
	\V \nabla u \nabla^\perp\phi \V_{C^\alpha(\mathbb{R}^2)} \le C \V \nabla u\V_{L^\infty(\mathbb{R}^2)} \Vert\nabla^\perp\phi\V_{C^\alpha(\mathbb{R}^2)}.
	\end{split}
	\end{equation}
\end{lemma} \noindent The point of \eqref{eq:key_estimate} is that we do not need to take the $C^\alpha$-norm of the velocity gradient. 

Given these lemmas, one can finish the global well-posedness proof with a simple Gronwall estimate (details of this argument can be found in \cite{BeCo}). We differentiate \eqref{eq:evolv_phi} to obtain \begin{equation*}
\begin{split}
\pr_t \nabla^\perp\phi + (u\cdot\nabla)\nabla^\perp\phi = \nabla u \nabla^\perp\phi.
\end{split}
\end{equation*} Working on the Lagrangian coordinates, and using the bound \eqref{eq:key_estimate} and then the logarithmic estimate \eqref{eq:geometric_lemma} allows one to close the estimates in terms of $\V \nabla^\perp \phi\V_{C^\alpha}$ to show the bound \begin{equation*}
\begin{split}
\V \nabla^\perp \phi(t)\V_{C^\alpha(\mathbb{R}^2)} \le C\exp(C\exp(Ct))
\end{split}
\end{equation*} as well as \begin{equation*}
\begin{split}
\V \nabla^\perp\phi(t)\V_{\inf(\partial\Omega)} \ge c\exp(-ct)
\end{split}
\end{equation*} with positive constants depending only on the initial data $\nabla^\perp \phi_0$ (and $0 < \alpha <1$). 

We would like to point out that, although it was not necessary in the above global well-posedness argument, the velocity gradient is indeed uniformly $C^\alpha$ inside the patch, up to the boundary. There are a number of ways to obtain this piece of information. One approach, due to Serfati \cite{Ser1}, is that from the directional H\"older regularity $(\nabla^\perp\phi \cdot \nabla) u \in C^\alpha$ that we already have, one can ``invert'' this using $\nabla\cdot u = 0$ and $\nabla\times u =1$ (inside the patch) to recover $ {\nabla u \in C^\alpha}$. We exploited this idea in our proof of local well posedness (see Lemma \ref{lem:velocity_circle}). Alternatively, Friedman and Velazquez \cite{FV} have shown, directly working with the Biot-Savart kernel, the following estimate: \begin{lemma}[Friedman and Velazquez \cite{FV}]
	Assume that a $C^{1,\alpha}$-patch $\Omega$ is tangent to the horizontal axis at the origin, and that near the origin, $\partial\Omega$ is described as the graph of a $C^{1,\alpha}$-function: \begin{equation*}
	\begin{split}
	\partial\Omega \cap [-\delta,\delta]^2 = \{ (x_1,x_2) : x_2 = f(x_1) \}, \qquad f \in C^\alpha([-\delta,\delta]),  \qquad \sup_{[-\delta,\delta]}|f'| \le 1.
	\end{split}
	\end{equation*} Then, the velocity $u = K* \chi_{\Omega}$ is $C^{1,\alpha}$ along this portion of the boundary: \begin{equation*}
	\begin{split}
	\V \nabla u(x_1,f(x_1)) \V_{C^\alpha_{x_1}[-\delta/10,\delta/10]} \le C \V f \V_{C^{1,\alpha}[-\delta,\delta]} \log\left( 1 + \frac{1}{\delta} \right).
	\end{split}
	\end{equation*}
\end{lemma} With elliptic regularity, the above lemma immediately implies that for $C^{1,\alpha}$-patches, the velocity gradient is uniformly $C^\alpha$ up to the boundary. 

The above lemma of Friedman and Velazquez actually gives $C^{1,\alpha}$-regularity for  {the velocity field} coming from a $C^{1,\alpha}$-cusp: consider the domain $\Omega$ satisfying \begin{equation*}
\begin{split}
\Omega \cap [-\delta,\delta]^2 = \{  (x_1,x_2) \subset [0,\delta]\times[-\delta,\delta] : g(x_1) < x_2 < f(x_1)  \}
\end{split}
\end{equation*} where $g < f$ are $C^{1,\alpha}[0,\delta]$-functions with $g(0) = f(0) = 0$ and $g'(0) = f'(0) = 0$. Then, applying the lemma first with a $C^{1,\alpha}$ domain obtained by taking $(x_1,f(x_1))$ and the semi-axis $\{ (x_1,0) : x_1 \le 0 \}$ as a portion of its boundary, and then using the lemma another time with a domain using $(x_1,g(x_1))$ instead of $f$ establishes that $\nabla u$ is uniformly $C^\alpha$ in $[0,\delta/10]\times [-\delta,\delta] \cap \Omega$.  {We essentially re-prove this estimate in this work and use it in several places.}

\subsection{Euler equations in critical spaces under symmetry}\label{subsec:symmetries_and_critical}

In this subsection, let us provide a brief review of some of the results from \cite{EJ1}. The contents of Sections \ref{sec:intermediate} and \ref{sec:main} may be viewed as generalizations of the results below to the class of vortex patch solutions.

\subsubsection*{Well-posedness of the 2D Euler equations in critical spaces}

The following result shows that in the $L^1 \cap L^\infty(\mathbb{R}^2)$-theory of Yudovich, one can actually drop the $L^1$ assumption under $m$-fold rotational symmetry for some $m\ge 3$. 

\begin{theorem}[{{\cite[Theorem 4]{EJ1}}}]
	Assume that $\omega_0 \in L^\infty(\mathbb{R}^2)$ and $m$-fold symmetric for some $m \ge 3$. Then, there is a unique solution to the 2D Euler equation $\omega \in L^\infty([0,\infty); L^\infty(\mathbb{R}^2))$ and $m$-fold symmetric. Here, $u$ is the unique solution to the system \begin{equation*}
	\begin{split}
	\nabla \times  u = \omega , \qquad \nabla\cdot u = 0. 
	\end{split}
	\end{equation*} under the assumptions $|u(x)| \le C|x|$ and $m$-fold symmetric. It is well-defined pointwise by \begin{equation*}
	\begin{split}
	u(t,x) = \lim_{R \rightarrow \infty} \frac{1}{2\pi} \int_{|y| \le R} \frac{(x-y)^\perp}{|x-y|^2} \omega(t,y)dy.
	\end{split}
	\end{equation*}
\end{theorem}

Under the assumption of the above theorem, the velocity is only log-Lipschitz, just as in the case of Yudovich theory, but now one has the following scale-invariant log-Lipschitz estimate, which is a key step in the proof. 

\begin{lemma}[{{\cite[Lemma 2.7]{EJ1}}}]
	Under the $m$-fold symmetry assumption for $m \ge 3$, we have \begin{equation*}
	\begin{split}
	|u(x) - u(x') | \le C\V \omega \V_{L^\infty} |x - x'| \log\left( \frac{c \max(|x|,|x'|)}{|x-x'|} \right).
	\end{split}
	\end{equation*}
\end{lemma}

In particular, under symmetry, vortex patches can have infinite mass and the evolution is still well-defined. This allows us to treat infinite patches in the setup of Sections \ref{sec:intermediate} and \ref{sec:main} (assuming that the boundary regularity of the initial patch as $|x| \rightarrow +\infty$ satisfies suitable bounds), but we shall not pursue this generalization.

It turns out that under the symmetry assumption, one can prove higher regularity in the angular direction. A model situation is when the vorticity takes the form $\omega = h(\theta) + \tilde{\omega}$, where $h(\cdot):S^1 \rightarrow \mathbb{R}$ defines a radially homogeneous function on $\mathbb{R}^2$, and $\tilde{\omega}$ is smooth on $\mathbb{R}^2$. Then, one sees that while $\omega$ cannot be better than $L^\infty(\mathbb{R}^2)$ in the $C^{k,\alpha}$-scale (unless $h$ is trivial), but one can take as many angular derivatives $\pr_{\theta}$ as $h$ allows. In this setup, one would like to say that the Euler dynamics propagates this regularity. To this end, we have introduced the scale-invariant spaces $\mathring{C}^\alpha(\mathbb{R}^2)$: for any $0 < \alpha \le 1$, consider the  norm \begin{equation*}
\begin{split}
\V f\V_{\mathring{C}^\alpha(\mathbb{R}^2)} &:= \V f \V_{L^\infty(\mathbb{R}^2)} + \V |x|^\alpha f(x)\V_{\dot{C}^\alpha(\mathbb{R}^2)} \\
& = \sup_x |f(x)| + \sup_{x \ne x'} \frac{\left||x|^\alpha f(x) - |x'|^\alpha f(x') \right|}{|x-x'|^\alpha}.
\end{split}
\end{equation*} Note that if $f$ is a function of the angle, $f(x) = h(\theta)$,   then \begin{equation*}
\begin{split}
\V f\V_{\mathring{C}^\alpha(\mathbb{R}^2)} \approx \V h(\theta)\V_{C^\alpha(S^1)}. 
\end{split}
\end{equation*}

\begin{theorem}[{{\cite[Theorem 11]{EJ1}}}]
	Assume that $\omega_0 \in \mathring{C}^\alpha(\mathbb{R}^2)$ is $m$-fold symmetric for some $m \ge 3$. Then, the unique solution in $L^\infty([0,\infty);L^\infty(\mathbb{R}^2))$ actually belongs to $L^\infty_{loc}\mathring{C}^\alpha$ with a bound \begin{equation*}
	\begin{split}
	\V \omega(t)\V_{\mathring{C}^\alpha} \le C \exp(c_1 \exp(c_2t)),
	\end{split}
	\end{equation*} with constants depending only on $0 < \alpha \le 1$ and the initial data. 
\end{theorem}

A key ingredient is the following scale-invariant bounds on the velocity gradient:

\begin{lemma}[{{\cite[Lemma 2.14]{EJ1}}}]
	The velocity gradient satisfies 
	\begin{equation*}
	\begin{split}
	\V \nabla u \V_{L^\infty} &\le C_\alpha \Vert\omega\Vert_{L^\infty} \left( 1 + \log\left(1 + c_\alpha \frac{\V \omega \V_{\mathring{C}^\alpha}}{\V \omega \V_{L^\infty}} \right) \right) 
	\end{split}
	\end{equation*} and \begin{equation*}
	\begin{split}
	\V \nabla u\V_{\mathring{C}^\alpha} &\le C_\alpha \V \omega \V_{\mathring{C}^\alpha}.
	\end{split}
	\end{equation*}
\end{lemma}

It is important to keep in mind the following Bahouri-Chemin \cite{BC} counterexample, which is only 2-fold rotationally symmetric. Take $\omega(x_1,x_2) = \mathrm{sign}(x_1)\mathrm{sign}(x_2)\chi_R$, where $\chi_R$ is some smooth radial cutoff. This belongs to $L^\infty$ but near the origin, it can be computed that $u(x_1,0) \approx Cx_1 \log x_1$, so that in particular the estimate $|u(x)|\le C|x|$ fails.\footnote{Here, we are using $\approx$ to say that both sides coincide up to a smooth function vanishing at the origin.} Moreover, even if we smooth it our in the angular direction, for instance by putting $\omega(x_1,x_2) = \cos(2\theta) \chi_R$, then $\omega$ belongs to $\mathring{C}^\alpha$ but still one has $u(x_1,0) \approx C'x_1 \log x_1$. 

\subsubsection*{The 1D system for radially homogeneous vorticity}

The $L^\infty$-theorem described above gives rise to a class of (infinite) vortex patch solutions to the 2D Euler equation, by taking vorticity which is radially homogeneous. 

Indeed, when the initial data is of the form $\omega_0 = h_0(\theta)$ with $h_0 \in L^\infty(S^1)$, then the unique solution must stay radially homogeneous for all time, and therefore the dynamics reduces to a one-dimensional equation on $h(t)$. We have derived this evolution equation in \cite[Section 3]{EJ1}: 
\begin{theorem}[{{\cite[Proposition 3.5]{EJ1}}}]
	Consider the following transport equation on $S^1 = [-\pi,\pi)$ \begin{equation*}
	\begin{split}
	\pr_t h + 2H\pr_{\theta} h = 0, 
	\end{split}
	\end{equation*} where the initial data $h_0$ is $m$-fold rotationally symmetric on $S^1$ for some $m \ge 3$. Here, $H$ is the unique solution of \begin{equation*}
	\begin{split}
	h = 4H + H'', \qquad \frac{1}{2\pi} \int_{-\pi}^{\pi} H(\theta)\exp(\pm 2i\theta)d\theta = 0. 
	\end{split}
	\end{equation*} Alternatively, \begin{equation*}
	\begin{split}
	H(\theta) = \frac{1}{2\pi}\int_{-\pi}^{\pi} K_{S^1}(\theta - \theta') h(\theta')d\theta',
	\end{split}
	\end{equation*} with \begin{equation*}
	\begin{split}
	K_{S^1}(\theta) := \frac{\pi}{2} \sin(2\theta) \frac{\theta}{|\theta|} - \frac{1}{2} \sin(2\theta)\theta - \frac{1}{8} \cos(2\theta). 
	\end{split}
	\end{equation*} The system is globally well-posed for either $h_0 \in L^\infty$ or $h_0 \in C^{\alpha}$ for $0 < \alpha \le 1$. 
	
	By taking \begin{equation*}
	\begin{split}
	\omega(t,x) &= h(t,\theta),\\
	u(t,x) &= 2H(t,\theta) \begin{pmatrix}
	-x_2 \\ x_1
	\end{pmatrix} - \pr_{\theta}H(t,\theta) \begin{pmatrix}
	x_1 \\ x_2
	\end{pmatrix},
	\end{split}
	\end{equation*} we obtain the unique solution to the 2D Euler equation with initial data $\omega_0(x) = h_0(\theta)$. 
\end{theorem}

Indeed, one may check with direct computations that the velocity defined in the above formula satisfies $\nabla\times u = 4H + H'' = \omega$ and $\nabla\cdot u = 0$, which characterizes the velocity. 

The kernel $K_{S^1}$ is simply the Biot-Savart kernel, restricted to the case of radially homogeneous vorticity. Since the vorticity has $m$-fold symmetry, it is more efficient to symmetrize the kernel as well: we have \begin{equation*}
\begin{split}
K^{(4)}_1(\theta):= \frac{1}{4} \sum_{j=0}^{3} K_{S^1}(\theta + j\pi/2) = \frac{\pi}{8} |\sin(2\theta)|.
\end{split}
\end{equation*} In general, \begin{equation*}
\begin{split}
K^{(m)}_1(\theta):= \frac{1}{m} \sum_{j=0}^{m-1} K_{S^1}(\theta + 2j\pi/m) = c_1^{(m)} |\sin(m\theta/2)| + c_2^{(m)}
\end{split}
\end{equation*} for some constants $c_1 > 0$ and $c_2$. We shall use these expressions in Subsection \ref{subsec:cusp_formation}. 

In the special case when $h_0 $ is the ($m$-fold symmetric) characteristic function of a disjoint union of intervals in $S^1$, we obtain a vortex patch solution on the plane, which is a union of sectors and whose boundary is a union of straight lines passing through the origin. The dynamics of these lines determine the evolution of the patch, and it takes the form of a system of ODEs, which we derive and briefly study in Subsection \ref{subsec:cusp_formation}. 

Lastly, consider the situation where the initial vorticity is the sum of a radially homogeneous function and a smooth function vanishing at the origin. Then, the next result says that near the origin, the dynamics is determined by the 1D evolution of the radially homogeneous part. 

\begin{theorem}[{{cf. \cite[Theorem 23]{EJ1}}}]
	Assume that the initial vorticity $\omega_0 \in \mathring{C}^\alpha(\mathbb{R}^2)$ is $m$-fold symmetric for some $m \ge 3$ and satisfies \begin{equation*}
	\begin{split}
	\omega_0(x) = h_0(\theta) + \tilde{\omega}_0(x),
	\end{split}
	\end{equation*} where $h_0 \in \mathring{C}^\alpha(S^1)$ and $\tilde{\omega}_0 \in C^{1,\alpha}(\mathbb{R}^2)$ with $\tilde{\omega}_0(0) = 0$. Then, the solution satisfies \begin{equation*}
	\begin{split}
	\omega(t,x) = h(t,\theta) + \tilde{\omega}(t,x)
	\end{split}
	\end{equation*} where $h(t,\cdot)$ is the unique solution to the 1D equation with initial data $h_0$, and $\tilde{\omega}(t,\cdot) \in C^{1,\alpha}(\mathbb{R}^2)$ with $\tilde{\omega}(t,0) = 0$. 
\end{theorem}  {This result was stated and used (implicitly) in the work \cite{EJ1} without a proof. For the proof, one can easily adapt the arguments given in the proof of \cite[Theorem 23]{EJ1}, which establishes the corresponding statement for the SQG (surface quasi-geostrophic) equation.} In particular, $\sup_{t \in [0,T]} |\tilde{\omega}(t,x)| \le C(T) |x|^{1+\alpha}$ for some constant $C(T)$ depending on $T$ and initial data, and therefore, it is negligible relative to $h(t,\theta)$ in the regime $|x| \ll 1$ (unless $h_0$ were trivial to begin with).

\section{Global well-posedness for symmetric patches in an intermediate space}\label{sec:intermediate}

In this section, we show that if a vortex patch admits a level set whose gradient is, roughly speaking, $C^\alpha$ in the angle and non-degenerate, then the corresponding Yudovich solution retains this property for all time. As a consequence, we shall have that the velocity, and hence the flow map and its inverse, are Lipschitz functions in space  {for all finite time.} In this setup, it is necessary to impose that the patch is $m$-fold rotationally symmetric for some $m \ge 3$. 

\begin{definition}\label{def:level_set_circle}
	Let us say that a domain $\Omega$ is a $\mathring{C}^{1,\alpha}$-patch, if it admits a level set $\phi : \mathbb{R}^2 \rightarrow \mathbb{R}$ such that: \begin{itemize}
		\item We have $\phi(x) > 0$ if and only if $x \in \Omega$. 
		\item The tangent vector field of $\phi$ satisfies $\nabla^\perp \phi \in \mathring{C}^{\alpha}(\mathbb{R}^2)$ (In particular $\phi$ is Lipschitz). 
		\item The function $\phi$ is non-degenerate near $\partial \Omega$, i.e., $ \V \nabla^\perp\phi\V_{\inf(\partial\Omega)}:= \inf_{x \in \partial \Omega} |\nabla^\perp\phi| \ge c > 0$. 
	\end{itemize}
\end{definition}

Let us present a practical sufficient condition for a domain $\Omega$ to satisfy Definition \ref{def:level_set_circle}. Observe first that the definition is invariant under composition with bi-Lipschitz $\mathring{C}^{1,\alpha}$ maps. Indeed, assume that $\Omega_0$ is a $\mathring{C}^{1,\alpha}$ patch with corresponding level set function $\phi_0$ and that $\Psi$ is bi-Lipchitz and $\nabla\Psi\in\mathring{C}^{0,\alpha}$. Now consider $\Omega:=\Psi(\Omega_0)$. Consider the new level set function $\phi=\phi_0\circ\Psi^{-1}$. Observe that \[\{x: \phi(x)>0\}=\{x: \phi_0\circ\Psi^{-1}(x)=0\}=\Psi(\{x:\phi_0(x)>0\})=\Omega.\] Also observe that $\nabla\phi\in\mathring{C}^{0,\alpha}$. Moreover, since $\Phi$ is Bi-Lipschitz we know that:
\[\|\nabla\phi\|_{\inf(\partial\Omega)}=\|\nabla\Psi^{-1}\circ\Psi(x)\nabla\phi_0(x)\|_{\inf(\partial\Omega_0)}\geq kc_0>0,\] since we know that $\|\nabla\Psi^{-1}v\|\geq k\|v\|$ for all vectors $v\in\mathbb{R}^2$ for some constant $k>0$ while $\phi_0$ is non-degenerate near $\partial\Omega_0$. This concludes the proof that the definition is invariant under composition with certain bi-Lipschitz maps. A non-smooth example of such an $\Omega_0$ is just the domain $\Omega_0=\{x_1x_2>0\}$ where we take $\phi_0(x_1,x_2)=\frac{x_1x_2}{|x|}$. In this case we see that $\nabla\phi_0\in\mathring{C}^{0,\alpha}$ easily while $|\nabla\phi|=1$ along $\partial\Omega_0$. The above then tells us that any sufficiently nice bi-Lipschitz deformation of this domain is also a $\mathring{C}^{1,\alpha}$ patch.

We are ready to state our main result of this section. \begin{theorem}\label{thm:intermediate}
	Assume that the initial patch $\Omega_0$ is $m$-fold symmetric for some $m \ge 3$ and admits a level set $\phi_0$ described in Definition \ref{def:level_set_circle}. Then, the Yudovich solution $\Omega_t$ continues to have this property; more specifically, by defining $\phi(t)$ as the solution of \eqref{eq:evolv_phi}, we have a global-in-time bounds \begin{equation}\label{eq:Ccirclealpha_phi_expexp}
	\begin{split}
	\V \nabla^\perp\phi(t)\V_{\mathring{C}^\alpha(\mathbb{R}^2)} \le C\exp(C\exp(Ct)),
	\end{split}
	\end{equation} \begin{equation}\label{eq:inf_bound}
	\begin{split}
	\V \nabla^\perp\phi(t)\V_{\inf(\partial\Omega_t)} \ge c\exp(-ct),
	\end{split}
	\end{equation} and \begin{equation}\label{eq:Linfty_vel_gradient}
	\begin{split}
	\V \nabla u(t)\V_{L^\infty(\mathbb{R}^2)} \le C \exp(Ct),
	\end{split}
	\end{equation} with constants $C, c > 0$ depending only on $\nabla^\perp\phi_0$ and $0 < \alpha <1$.
\end{theorem}

\begin{remark}
	Note that, in the above theorem, we do not require the initial patch $\Omega_0$ to have compact support. However, we do require that the gradient $\nabla^\perp\phi_0$ to have  {uniformly bounded $\mathring{C}^\alpha$-norm on all of $\mathbb{R}^2$.}
\end{remark}

Recall from Subsection \ref{subsec:symmetries_and_critical} that the 2D Euler equation is globally well-posed with $\omega_0 \in \mathring{C}^\alpha$ under symmetry. Therefore, the global well-posedness of the patch admitting a level set (under the same symmetry assumption) with $\nabla^\perp\phi_0 \in \mathring{C}^\alpha$ is a natural analogue of the classical global well-posedness result of $C^{1,\alpha}$-patches. As an immediate consequence of the above theorem, we have that, 
\begin{corollary}
	Under the assumptions of Theorem \ref{thm:intermediate}, the flow map $\Phi_t$ is a Lipschitz bijection of the plane with a Lipschitz inverse for all times $t \ge 0$.
\end{corollary}

Before we proceed to the proof, let us describe a few classes of vortex patches satisfying the requirements of Definition \ref{def:level_set_circle}. 

\begin{er*} Theorem \ref{thm:intermediate} establishes global well-posedness for each of the following  {classes of examples}, under the assumption of $m$-fold rotational symmetry with some $m \ge 3$. 
	
	\begin{itemize}
		\item[(i)] Sectors: Assume that for some ball $B_0(r)$, the intersection $\Omega_0 \cap B_0(r)$ is a union of sectors meeting at the origin (see Figures \ref{fig:petal}, \ref{fig:multiple_corners} for symmetric examples). In addition, assume that  $\partial\Omega_0$ is $C^{1,\alpha}$-smooth in the complement of $B_0(r)$. Then, one may take a level set locally by $\phi_0(x) = r h_0(\theta)$ in polar coordinates with some $h_0(\cdot) \in C^{1,\alpha}(S^1)$, where $h_0$ can be appropriately chosen that $\phi_0$ satisfies Definition \ref{def:level_set_circle}. Moreover, the same holds for the image $\Psi(\Omega_0)$ of such a patch $\Omega_0$ under a global $C^{1,\alpha}$-diffeomorphism of the plane $\Psi$ satisfying $|\Psi(x)|\le C|x|^{1+\alpha}$ for some $C > 0$. These facts are proved in Lemma \ref{lem:admit_level_set} of the next section, where we study in detail the evolution of such vortex patches, under the assumption of $m$-fold symmetry.  
		
		This class of vortex patches (which are locally the $C^{1,\alpha}$-diffeomorphic image of a union of sectors meeting at the origin) are studied in great detail in Section \ref{sec:main}. Unfortunately, the fact that $\nabla^\perp\phi$ stays in $\mathring{C}^\alpha$ for all time is not sufficient to conclude that the evolved patch is still given by the image of some $C^{1,\alpha}$-diffeomorphism. Therefore, a careful local analysis should be supplemented to recover this information (see Subsection \ref{subsec:local_estimate}).  
		
		\item[(ii)] Logarithmic spirals: Take some periodic indicator function $\chi_I$ where $I$ is some interval of $S^1 = [0,2\pi)$ and consider a patch $\Omega_0$ which is locally given by \begin{equation*}
		\begin{split}
		\omega_0(r,\theta) = \chi_I \circ \left( -c \log(r) + \theta \right), \qquad r < 1/2
		\end{split}
		\end{equation*} where $c > 0$ is some constant. Taking $h_0 \in C^{1,\alpha}(S^1)$ vanishing precisely on the endpoints of the interval $I$ with non-zero derivatives, and then by setting $\phi_0 = rh_0(-c\log(r)+\theta)$, one may check that this function satisfies the requirements of Definition \ref{def:level_set_circle} (assuming for instance $\Omega_0$ is a $C^{1,\alpha}$ patch in  {the region $\{r  \ge 1/2\}$}). This boils down to checking that, for a given function $\zeta \in C^\alpha(S^1)$ with $0 < \alpha \le 1$, $\zeta \circ (- c\ln r + \theta) \in \mathring{C}^\alpha(\mathbb{R}^2)$. For simplicity, take the case $\alpha =1 $, and then \begin{equation*}
		\begin{split}
		\frac{1}{r}\pr_{\theta} \zeta = \frac{1}{r}\zeta', \qquad \pr_r \zeta = - \frac{c}{r}\zeta', 
		\end{split}
		\end{equation*} so that switching to rectangular coordinates, $|x||\nabla \zeta(x)| \in L^\infty(\mathbb{R}^2)$, or equivalently $\zeta(x) \in \mathring{C}^1(\mathbb{R}^2)$. Similarly as in the case of (i), one can treat patches which are given as the image of an exact spiral by a $C^{1,\alpha}$-diffeomorphism of the plane fixing the origin. 
		
		In the special case when the initial vorticity is given exactly by $\omega_0 = h_0(-c\log r + \theta)$, then as we have seen in the introduction, a 1D evolution equation satisfied by $h(t,\cdot)$ can be derived, so that $\omega(t,x) := h(-c\log r + \theta, t)$ solves the 2D Euler equations. This remark is due to Julien Guillod (private communication).
		
		It is interesting question to see if one can start with a patch which locally looks like a union of sectors (as in the case (i)) and converges to a logarithmic spiral when $t \rightarrow + \infty$. 
		
		The patch corresponding to the case $c = 5$ and $I = [0,5\pi/24]$, with 3-fold symmetrization, is given in Figure \ref{fig:spiral}. 
		
		\item[(iii)] Cusps: Consider the (infinite) region bounded by two tangent $C^{1,\alpha}$-functions $f_0, g_0 : [0,\infty) \rightarrow \mathbb{R}$: \begin{equation*}
		\begin{split}
		\Omega_0 =  \{ (x_1,x_2) : g_0(x_1) < x_2 < f_0(x_1) \}, \qquad f'_0(0) = g_0'(0) = 0, \qquad g_0 < f_0 \mbox{ on } (0,\infty).
		\end{split}
		\end{equation*} Here, we require that $f_0$ and $g_0$ are uniformly $C^{1,\alpha}$ in all of $\mathbb{R}$. A model case is provided by taking $f_0(x_1) = x_1^{1+\alpha}$ and $g_0(x_1) = -x_1^{1+\alpha}$ (locally for $x_1$ near 0). One may take a number of such cusps (possibly with different boundary profiles for each of them) and rotate each of them around the origin to make them disjoint. In particular, the resulting union of cusps can be $m$-fold symmetric for any $m \ge 3$. In this setting, it is convenient to consider the complement $\mathbb{R}^2 \backslash \Omega$, which is more-or-less a union of corners. Then one may take some $\phi_0$ with $\nabla^\perp\phi_0 \in \mathring{C}^\alpha$ defined on  $\mathbb{R}^2 \backslash \Omega_0$. It can be taken to be $C^{1,\alpha}$ smooth when one ``crosses'' each of the cusps (see Figure \ref{fig:cusp}). We discuss them in some detail in Subsection \ref{subsec:extensions}.  
		
		Danchin has shown in \cite{Da} that the cusp-like singularities in a smooth vortex patch propagates globally in time. We are also aware of works of Serfati in this direction. It is likely that the following alternative argument for the global well-posedness would go through: first apply Theorem \ref{thm:intermediate} to obtain global propagation in the intermediate class $\mathring{C}^\alpha$, and supply an additional local argument to recover $C^{1,\alpha}$-regularity up to the point of singularity. 
		
		\item[(iv)] Bubbles accumulating at the origin: Take a sequence of smooth $C^{1,\alpha}$-patches $\{ U_n \}_{n \ge 0}$, which for simplicity are assumed to have comparable diameters (say less than 1/2) and $C^{1,\alpha}$-characteristic scales. Now rescale the $n$-th patch $U_n$ by a factor of $2^{-n}$, denote it by $\tilde{U}_n$, and place it inside the annulus $A_n = \{ x : 2^{-n} < |x| < 2^{-n+1} \}$. Then define $\Omega_0$ as the union of rescaled patches $\cup_{n \ge 0} \tilde{U}_n$. It can be easily arranged that, by placing several disjoint patches in each annulus region, the entire set $\Omega_0$ is $m$-fold symmetric for some $m \ge 3$. 
		
		Assuming $m$-fold symmetry, Theorem \ref{thm:intermediate} applies to show that the evolution of the (rescaled) $n$-th patch $\tilde{U}_n$ has boundary in $C^{1,\alpha}$ with its characteristic satisfying \begin{equation*}
		\begin{split}
		c(T) 2^{ n} \le \Gamma_n(t) \le C(T) 2^{  n}
		\end{split}
		\end{equation*} for any $T > 0$ and $t \in [0,T]$. In particular, by rescaling each of $\tilde{U}_n$ back to a patch of diameter $O(1)$, we have that their $C^{1,\alpha}$-characteristics are uniformly bounded from above and below. Even without the symmetry, it can be shown that the boundary of each $\tilde{U}_n$ stays in $C^{1,\alpha}$ for all time. However, a uniform bound (after rescaling) cannot hold in general. Indeed, such a non-uniform growth was utilized in the work of Bourgain-Li \cite{BL1,BL2} (see also \cite{EJ,JY}), after smoothing out the patches appropriately, to produce examples of $\omega_0 \in H^1(\mathbb{R}^2)$ which escapes $H^1(\mathbb{R}^2)$ instantaneously for $t > 0$. 
	\end{itemize}
\end{er*}

\begin{figure}
	\includegraphics[scale=0.25]{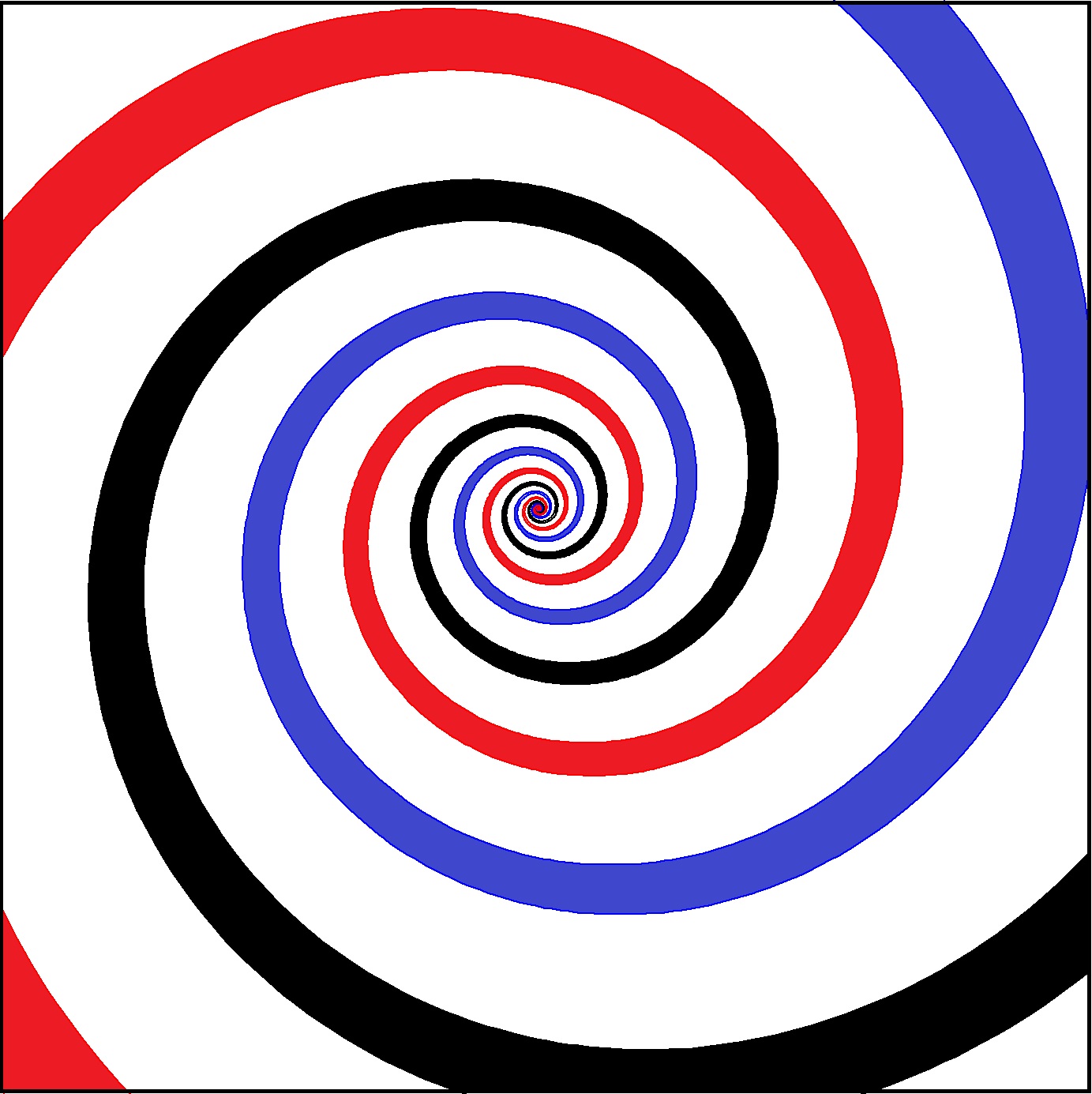} 
	\centering
	\caption{A 3-fold symmetric logarithmic spiral.}
	\label{fig:spiral}
\end{figure}

The proof of Theorem \ref{thm:intermediate} is parallel to the one given in \cite{BeCo} and based on two ``frozen time'' estimates, except that the $m$-fold rotational symmetry gets involved in the current setup.  We first observe that in this setting, an identity of the form \eqref{eq:key_identity} still holds: \begin{equation}\label{eq:key_identity_singular}
\begin{split}
\nabla u \nabla^\perp\phi (x) =  \int_{\Omega} \nabla K(x-y) \left( \nabla^\perp\phi(x) - \nabla^\perp\phi(y) \right) dy,
\end{split}
\end{equation} since all that was necessary to establish the above formula is to have the vector field $\nabla^\perp\phi$ divergence free and tangent to the boundary of the patch. Given the identity \eqref{eq:key_identity_singular}, we can prove the following estimate: \begin{lemma}\label{lem:key_identity}
	Assume that a domain $\Omega$ admits a level set $\phi$ satisfying Definition \ref{def:level_set_circle}.	Then, we have a bound \begin{equation*}
	\begin{split}
	\V \nabla u \nabla^\perp\phi\V_{\mathring{C}^\alpha(\mathbb{R}^2)} \le C \left( 1 + \V \nabla u\V_{L^\infty(\mathbb{R}^2)} \right) \V \nabla^\perp\phi\V_{\mathring{C}^\alpha(\mathbb{R}^2)}.
	\end{split}
	\end{equation*}
\end{lemma} This is just a particular case of a general estimate about the space $\mathring{C}^\alpha$, which works in the setting of convolution against classical Calderon-Zygmund kernels. It is worth noting that the symmetry is not necessary for this particular lemma.  Next, \begin{lemma}\label{lem:key_estimate}
	Under the assumptions of Lemma \ref{lem:key_identity}, we have the following logarithmic bound: \begin{equation}\label{eq:key_estimate_singular}
	\begin{split}
	\V \nabla u\V_{L^\infty(\mathbb{R}^2)} \le C_\alpha \left( 1 + \log\left( 1+ \frac{\V \nabla^\perp\phi\V_{\mathring{C}^\alpha(\mathbb{R}^2)}}{\V \nabla^\perp\phi\V_{\inf(\partial\Omega)}} \right) \right)
	\end{split}
	\end{equation}
\end{lemma}
The symmetry assumption is essential here; basically, the information that $\nabla^\perp\phi$ belongs to $\mathring{C}^\alpha$ gives an effective $C^{1,\alpha}$ bound on $\partial\Omega$ only in a region of $O(|x|)$ at a given point $x$, and the procedure of ``zooming out'' it to a region of size $O(1)$ will in general bring the logarithmic loss, unless the $m$-fold rotational symmetry for some $m \ge 3$ is imposed on the set $\Omega$.

Given these lemmas, let us give a sketch of the proof. 

\begin{proof}[Proof of Theorem \ref{thm:intermediate}]
	We assume that the local-in-time existence in the desired class is given, so that as long as the $\mathring{C}^\alpha$-characteristic for $\nabla^\perp\phi_t$ remains finite, the solution can be extended further. (This part is deferred to the Appendix.) 
	
	It suffices to obtain a global-in-time a priori estimate for the characteristic quantity\footnote{Note that, unlike the $C^{1,\alpha}$-characteristic quantity that appeared earlier in the case of smooth patches, this quantity is non-dimensional. We use the notation $\mathring{\Gamma}_t$ to emphasize this fact from now on.} \begin{equation*}
	\begin{split}
	\mathring{\Gamma}_t = \left( \frac{\V \nabla^\perp\phi_t \V_{\mathring{C}^\alpha(\mathbb{R}^2)}}{\V \nabla^\perp\phi_t\V_{\inf(\partial\Omega_t)}} \right)^{1/\alpha}.
	\end{split}
	\end{equation*}
	
	As we have mentioned earlier, this proof is completely parallel to the arguments of Bertozzi and Constantin \cite{BeCo}. We start with $W := \nabla^\perp\phi$, which satisfies \begin{equation*}
	\begin{split}
	\pr_t W + (u\cdot\nabla) W = \nabla u W. 
	\end{split}
	\end{equation*} Then, solving this equation along the flow, \begin{equation*}
	\begin{split}
	\frac{d}{dt} W(t,\Phi(t,x)) = \nabla u(t,\Phi(t,x)) W(t,\Phi(t,x)). 
	\end{split}
	\end{equation*} Integrating in time and then changing variables gives \begin{equation*}
	\begin{split}
	W(t,x) = W_0(\Phi_t^{-1}(x)) + \int_0^t (\nabla u W)(\Phi_{t-s}^{-1}(x),s) ds. 
	\end{split}
	\end{equation*} Using the bound on $\nabla\Phi_t^{-1}$ in terms of the velocity gradient, this implies, for points $x \ne x'$ satisfying $|x'| \le |x|$ and $|x-x'| \le |x|/2$, \begin{equation*}
	\begin{split}
	|W(t,x) - W(x',t)| &\le \V W_0\V_{\mathring{C}^\alpha} \exp\left( c\int_0^t \V \nabla u_s\V_{L^\infty} ds \right) \cdot \frac{|x-x'|^\alpha}{|x|^\alpha} \\
	&\qquad + \int_0^t \V \nabla u_s W_s \V_{\mathring{C}^\alpha} \exp\left( c \int_{s}^{t} \V \nabla u_{s'}\V_{L^\infty} ds' \right) ds \cdot \frac{|x-x'|^\alpha}{|x|^\alpha}. 
	\end{split}
	\end{equation*} Introducing $Q(s) =\V \nabla u_s \V_{L^\infty}$ and using Lemma \ref{lem:key_identity}, \begin{equation*}
	\begin{split}
	\V W_t \V_{\mathring{C}^\alpha} \le \V W_0 \V_{\mathring{C}^\alpha} \exp\left( c \int_0^t Q(s) ds \right) + C \int_0^t Q(s) \V W_s\V_{\mathring{C}^\alpha} \exp\left( c\int_s^t Q(s')ds' \right) ds. 
	\end{split}
	\end{equation*} (For a pair of points $x \ne x'$ and $|x'| \le |x|$ \textit{not} satisfying $|x - x'| \le |x|/2$, we can simply use the $L^\infty$-bound $\pr_t \V W_t \V_{L^\infty} \le Q_t \V W_t \V_{L^\infty}$.) Then, writing \begin{equation*}
	\begin{split}
	G(t) := \V W_t \V_{\mathring{C}^\alpha} \exp\left( -c \int_0^t Q(s)ds \right), 
	\end{split}
	\end{equation*} we have, after a little bit of manipulation, \begin{equation*}
	\begin{split}
	G(t) \le \V W_0 \V_{\mathring{C}^\alpha} + C \int_0^t Q(s) G(s) ds, 
	\end{split}
	\end{equation*} so that by Gronwall's Lemma, \begin{equation*}
	\begin{split}
	\V W_t \V_{\mathring{C}^\alpha} \le \V W_0 \V_{\mathring{C}^\alpha} \exp\left( (C + c) \int_0^t \V \nabla u_s\V_{L^\infty} ds \right). 
	\end{split}
	\end{equation*} On the other hand, we have trivially \begin{equation*}
	\begin{split}
	\V W_t \V_{\inf(\partial\Omega)} \ge \V W_0\V_{\inf(\partial\Omega)}  \exp \left( -\int_0^t \V \nabla u_s\V_{L^\infty} ds \right). 
	\end{split}
	\end{equation*} Combining these estimates, and then applying Lemma \ref{lem:key_estimate} finishes the proof. 
\end{proof}

\begin{proof}[Proof of Lemma \ref{lem:key_identity}]
	Let us set \begin{equation*}
	\begin{split}
	G(x) = \nabla u\nabla^\perp\phi(x).
	\end{split}
	\end{equation*} Then, we have trivially an $L^\infty$ bound: $|G(x)| \le \V\nabla u\V_{L^\infty} \V \nabla^\perp\phi\V_{L^\infty}$. Now the proof of the $C^\alpha$-estimate for $|x|^\alpha G(x)$ is strictly analogous to the proof of \eqref{eq:key_estimate} given in \cite[Proof of Corollary 1]{BeCo}. To see this, fix some $x, h$ and consider the difference \begin{equation*}
	\begin{split}
	|x|^\alpha G(x) - |x+h|^\alpha G(x+h).
	\end{split}
	\end{equation*} First, in the case $|h| > |x|/2$, after a rewriting the above expression is bounded in absolute value by \begin{equation*}
	\begin{split}
	\left| |x|^\alpha (G(x) - G(x+h)) + G(x+h)(|x|^\alpha - |x+h|^\alpha) \right| \le C|h|^\alpha \V G \V_{L^\infty} + |h|^\alpha \V G\V_{L^\infty}. 
	\end{split}
	\end{equation*} Therefore, we may assume that $|h| \le |x|/2$. Then, we write with $f := \nabla^\perp\phi$ \begin{equation*}
	\begin{split}
	&|x|^\alpha G(x) - |x+h|^\alpha G(x+h)  \\
	&\qquad=|x|^\alpha \int_{\Omega} \nabla K(x-y)(f(x) - f(y)) dy - |x+h|^\alpha \int_{\Omega} \nabla K(x+h-y)(f(x) - f(y)) dy \\
	&\qquad= |x|^\alpha \int_{\{ |x-y| < 2|h| \} \cap \Omega} \nabla K(x-y)(f(x)-f(y))dy \\
	&\qquad\quad -|x+h|^\alpha \int_{ \{|x-y| < 2|h|\} \cap \Omega } \nabla K(x+h-y) (f(x+h)-f(y)) dy \\
	&\qquad\quad +\int_{ \{|x-y| \ge 2|h|\} \cap \Omega }\nabla  K(x-y) (|x|^\alpha f(x) - |x+h|^\alpha f(x+h)) dy \\
	&\qquad\quad + \int_{ \{|x-y| \ge 2|h|\} \cap \Omega } \left(\nabla K(x-y)- \nabla K(x+h-y) \right) (|x+h|^\alpha f(x+h) - |x|^\alpha f(y)) dy \\
	&\qquad = I + II + III + IV.
	\end{split}
	\end{equation*} Then, \begin{equation*}
	\begin{split}
	|I| \le C |x|^\alpha \int_0^{2|h|} \V f\V_{\mathring{C}^{\alpha}} \frac{r^{\alpha-1}}{|x|^\alpha} \le C \V f\V_{\mathring{C}^{\alpha}} h^\alpha
	\end{split}
	\end{equation*} and similarly $|II| \le C \V f\V_{\mathring{C}^{\alpha}} h^\alpha$. For $III$, we note that \begin{equation*}
	\begin{split}
	|III| \le \left||x|^\alpha f(x) - |x+h|^\alpha f(x+h)\right| \cdot \left|\int_{ \{|x-y| \ge 2|h|\} \cap \Omega }\nabla  K(x-y)dy\right| \le C\V f\V_{\mathring{C}^{\alpha}}h^\alpha  \left( 1 + \V\nabla u\V_{L^\infty} \right),
	\end{split}
	\end{equation*} and finally for $IV$, simply rewrite \begin{equation*}
	\begin{split}
	|x+h|^\alpha f(x+h) - |x|^\alpha f(y) = (|x+h|^\alpha f(x+h) - |y|^\alpha f(y)) - f(y) (|x|^\alpha - |y|^\alpha)
	\end{split}
	\end{equation*} we use the decay of $\nabla\nabla K$ to bound both of them by \begin{equation*}
	\begin{split}
	|IV| \le C \left(\V f\V_{\mathring{C}^{\alpha}} + \V f \V_{L^\infty} \right) \int_{ \{|x-y| \ge 2|h|\} \cap \Omega } h \frac{1}{|x-y|^{3-\alpha}} dy \le C\V f\V_{\mathring{C}^{\alpha}} h^\alpha.
	\end{split}
	\end{equation*} This concludes the proof. 
\end{proof}

\begin{proof}[Proof of Lemma \ref{lem:key_estimate}]
	Let us take a (non-dimensional) parameter \begin{equation*}
	\begin{split}
	\delta = \frac{1}{10} \min\left\{ 1, \left( \frac{\V\nabla^\perp\phi\V_{\inf}}{ \V \nabla^\perp\phi\V_{\mathring{C}^\alpha} } \right)^{1/\alpha} \right\}.
	\end{split}
	\end{equation*} We then split the integral \begin{equation*}
	\begin{split}
	I(x) = \int_{\Omega } \nabla K(x-y) dy,
	\end{split}
	\end{equation*} (assuming that $x \ne 0$) as follows: \begin{equation*}
	\begin{split}
	\left[\int_{\Omega \cap \{ |x-y| < \delta |x| \} } + \int_{\Omega \cap \{ \delta|x| \le |x-y | < 10|x| \} } + \int_{\Omega \cap \{ 10|x| \le |x-y| \} } \right] \nabla K(x-y) dy =: I_1 + I_2 + I_3. 
	\end{split}
	\end{equation*} The bound for $I_1$  {follows from the ``geometric lemma''} of Bertozzi and Constantin \cite{BeCo}. To see this, first note that in the region $|x-y| < \delta |x|$, $\nabla^\perp\phi$ is a uniformly $C^\alpha$-function with norm $\approx |x|^{-\alpha}$. Then, we have 
	
	\medskip
	
	\textbf{Geometric Lemma.} For each $\rho > 0$, consider the total angle $R_\rho(x)$ of deviation of $\Omega \cap \partial B_{x}(\rho)$ from being a half-circle. Formally, \begin{equation*}
	\begin{split}
	R_\rho(x) :=S_\rho(x) \triangle \Sigma(x)
	\end{split}
	\end{equation*} (here $\triangle$ denotes the symmetric deference) with \begin{equation*}
	\begin{split}
	S_\rho(x)&:= \{ w : |w| = 1, x + \rho w \in \Omega \}, \qquad\quad
	\Sigma(x) := \{ w : |w| = 1, \nabla \rho(\tilde{x}) \cdot w \ge 0 \},
	\end{split}
	\end{equation*} where $\tilde{x}$ is a point on $\partial\Omega$ which achieves the minimum distance between $x$  and $\partial\Omega$\footnote{Indeed, one may imagine that $x \in \partial\Omega$ and hence $x = \tilde{x}$, as $\nabla u$ is potentially most singular at such points.}. Then, \begin{equation*}
	\begin{split}
	|R_\rho(x)| \le C \left( \frac{d(x,\partial\Omega)}{\rho} + \left(\frac{\rho}{\delta|x|}\right)^\alpha \right)
	\end{split}
	\end{equation*} as long as we take $\rho < \delta|x|$.  {This follows from the original geometric lemma in \cite{BeCo} since the $C^\alpha$ norm in $B_\rho(x)$ is comparable to $|x|^{-\alpha}$ if $\rho<\delta|x|$.}
	
	\medskip
	
	Given the above lemma, after using the fact that the averages of $\nabla K$ along half-circles vanish, we bound \begin{equation*}
	\begin{split}
	|I_1| = \left|\int_{\Omega \cap \{ |x-y| < \delta |x| \} } \nabla K(x-y) dy \right| \le C \int_{d(x,\partial\Omega)}^{\delta |x|} \frac{|R_\rho(x)|}{\rho} d\rho \le C. 
	\end{split}
	\end{equation*}
	
	Next, the bound on $I_2$ is straightforward; just taking absolute values, \begin{equation*}
	\begin{split}
	|I_2| \le \int_{\delta|x|}^{10|x|} \frac{C}{\rho}d\rho \le C \log(\delta^{-1}). 
	\end{split}
	\end{equation*}
	
	Finally, we use symmetry of the domain for $I_3$: note that  \begin{equation*}
	\begin{split}
	|I_3| \approx \frac{1}{m} \left| \int_{\Omega \cap \{ |y| \ge 10|x| \} } \sum_{j=0}^{m-1} \nabla K(x - R_{2\pi j/m}y) dy  \right|
	\end{split}
	\end{equation*} (Strictly speaking, the region $\Omega \cap \{ |x-y| \ge 10|x| \}$ is not really $m$-fold symmetric but the extra terms coming from the difference of the symmetrization of this set and $\Omega \cap \{ |y| \ge 10|x| \}$ can be bounded as in $I_2$) and then we use the fact that (see \cite[Lemma 2.17]{EJ1}) for $|y| \gtrsim |x|$, \begin{equation*}
	\begin{split}
	\frac{1}{m}\left|  \nabla K(x - R_{2\pi j/m}y) \right| \le C \frac{|x|}{|y|^3}.
	\end{split}
	\end{equation*} This gives $|I_3| \le C$, finishing the proof. 
\end{proof}

\section{Global well-posedness for symmetric $C^{1,\alpha}$-patches with corners}\label{sec:main}

\subsection{The geometric setup and the main statement}

In this section, we show global well-posedness of $C^{1,\alpha}$ vortex patches with corners meeting symmetrically at a point. Here we make this notion precise. For the convenience of the reader, let us recall some notations:  \begin{itemize}
	\item For a given angle $\theta$, we denote $R_\theta: \mathbb{R}^2 \rightarrow \mathbb{R}^2$ to be the counter-clockwise rotation of the plane by $\theta$ around the origin. 
	\item We define sectors using polar coordinates: \begin{equation*}
	\begin{split}
	S_{\beta,\beta+\zeta} := \{ (r,\theta) :  \beta < \theta < \beta + \zeta   \}.
	\end{split}
	\end{equation*}
\end{itemize}

\begin{definition}[$C^{1,\alpha}$-patches with symmetric corners]\label{def:C_1alpha_corner}
	We deal with patches $\Omega$ enjoying the following properties:
	\begin{itemize}
		\item (Symmetry) There exists an open domain $\Omega_1$, and some $m \ge 3$, such that \begin{equation*}
		\begin{split}
		\Omega = \cup_{j=0}^{m-1} R_{2\pi j/m}(\Omega_1)
		\end{split}
		\end{equation*} where the open sets $R_{2\pi j/m}(\Omega_1)$ are disjoint from each other and their closures intersect only at the origin. 
		\item ($C^{1,\alpha}$ away from the origin)  {For any $x\not=0$, there exists a small ball $B$ around $x$ such that $B \cap \partial\Omega$ is described by the graph of a $C^{1,\alpha}$-function (after rotating the patch if necessary).} 
		\item ($C^{1,\alpha}$ corner) There exists a $C^{1,\alpha}$-diffeomorphism $\Psi:\mathbb{R}^2 \rightarrow \mathbb{R}^2$ of the plane with $\Psi(0) = 0$ and $\nabla\Psi|_{x=0} = I$, such that for some $\delta > 0$, the image $\Psi(\Omega_1)$ is an exact sector of angle less than $2\pi/m$: \begin{equation}
		\Psi(\Omega_1) \cap B_0(\delta) = S_{\beta}(\zeta) \cap B_0(\delta)
		\end{equation} with some $0 < \zeta < 2\pi/m$ and $\beta \in [0,2\pi]$. 
	\end{itemize}
\end{definition} In the following, let us call such a patch by a ``symmetric $C^{1,\alpha}$-patch with corners'', or symmetric patch with corners for short. 

\begin{example}
	For each $m \ge 3$, the domain bounded by the set $\{ (r,\theta) : r = 1 + \cos(m\theta) \}$ (in polar coordinates)  {gives} an explicit example satisfying Definition \ref{def:C_1alpha_corner}. 
\end{example}

Quantifying $C^{1,\alpha}$-regularity of such a patch is a simple matter; one may use directly the $C^{1,\alpha}$-norm of the diffeomorphism $\Psi$, but we shall work with the following alternative description. By rotating the plane if necessary, we may assume that the boundary $\partial\Omega_1$ is locally described by the graph of two $C^{1,\alpha}$ functions $g < f$; that is, \begin{equation*}
\begin{split}
\partial\Omega_1 \cap [-\delta,\delta]^2 = \{ (x_1, f(x_1)) : 0 \le x_1 \le \delta \} \cup \{ (x_1, g(x_1)) : 0 \le x_1 \le \delta \}.
\end{split}
\end{equation*} Then, the last condition of Definition \ref{def:C_1alpha_corner} is equivalent to saying that $f$ and $g$ are $C^{1,\alpha}$-regular up to the boundary of the interval $[0,\delta]$. Then, the regularity of $\Omega$ may be quantified with the characteristic (note that it has the unit of inverse length) \begin{equation*}
\begin{split}
\Gamma(\Omega) := \V\nabla f\V_{C^\alpha[0,\delta]}^{1/\alpha} + \V\nabla g\V_{C^\alpha[0,\delta]}^{1/\alpha} + \Gamma(\Omega\backslash B_0(\delta/2))
\end{split}
\end{equation*} where $\Gamma(\Omega\backslash B_0(\delta))$ is the (usual) $C^{1,\alpha}$ characteristic of $\partial\Omega$ away from $B_0(\delta)$, where it is uniformly $C^{1,\alpha}$.

Our main statement of this section states that if $\Omega_0$ is a symmetric $C^{1,\alpha}$-patch with corners, then the unique Yudovich solution $\Omega(t,\cdot)$ remains so for all times $ t > 0$. We state it formally as follows: 

\begin{theorem}\label{thm:symmetric_corner}
	Let us assume that $\Omega_0$ is a vortex patch satisfying  Definition \ref{def:C_1alpha_corner}. Then the Yudovich solution $\Omega_t$ associated with $\Omega_0$ satisfy the same properties for all $t > 0$; that is, $ \Gamma(\Omega_t) < +\infty$. Moreover, the angles simply rotate with a constant angular speed for all time which is determined only by the size of the angle. In particular, the value of the angle does not change with time.  
\end{theorem}
To be clear, part of the statement is that for any $t > 0$, one can find $\delta = \delta(t) > 0 $ such that in the ball $B_0(\delta)$, the boundary of $\Omega_1(t)$ is given by two $C^{1,\alpha}$-curves $f_t$ and $g_t$, and in the complement of the ball $B_0(\delta/2)$, the boundary of $\Omega_t$ is uniformly $C^{1,\alpha}$.

At this point, let us note that all the hard work necessary in establishing the above result is to establish the last property in Definition \ref{def:C_1alpha_corner}: the first is trivial in view of the uniqueness and non-collision of particle trajectories. The second property is well-known; more generally, if a vortex patch is $C^{1,\alpha}$ away from some closed set, the solution remains smooth away from the image of the closed set under the flow \cite{Da2}. Indeed, this is an immediate consequence of Theorem \ref{thm:intermediate}, as soon as we prove the following
\begin{lemma}\label{lem:admit_level_set}
	Assume that $\Omega$ is a $C^{1,\alpha}$-patch with a symmetric corner. Then it  admits a level set $\phi : \mathbb{R}^2 \rightarrow \mathbb{R}$ satisfying conditions of Definition \ref{def:level_set_circle}. 
\end{lemma}
\begin{proof}
	
	We first check the statement in the case when $\Omega$ is given by an exact symmetric corner near the origin, that is, \begin{equation*}
	\begin{split}
	\Omega \cap B_0(\delta) = \cup_{j=0}^{m-1} R_{2\pi/m}^j (S_{\beta, \beta +\zeta}) \cap B_0(\delta)
	\end{split}
	\end{equation*} for some $\beta \in [0,2\pi)$ and $\zeta < 2\pi/m$. In this case, one may take a function $\phi$ which is $C^{1,\alpha}$ outside of the ball $B_0(\delta)$ and then locally \begin{equation*}
	\begin{split}
	\phi(x) = |x| \cdot g(\theta(x)),\quad \theta(x) = \tan^{-1}(x_2/x_1).
	\end{split}
	\end{equation*} Here, one may take a smooth function $g \in C^{1,\alpha}(S^1)$ so that $\phi(x)$ strictly positive inside the patch and negative outside, and also $g'$ is non-vanishing for angles which correspond to $\partial\Omega$. Then, taking the gradient one obtains for $|x| < \delta$ \begin{equation*}
	\begin{split}
	\nabla^\perp\phi(x) = \frac{x^\perp}{|x|} g(\theta) - \frac{x}{|x|}g'(\theta)
	\end{split}
	\end{equation*} and note that for $x \in \partial\Omega \cap B_0(\delta)$, \begin{equation*}
	\begin{split}
	|\nabla^\perp\phi(x)| = |g'(\theta)| > 0 
	\end{split}
	\end{equation*} by our choice of $g$ (here $\theta = \theta(x)$) and also \begin{equation*}
	\begin{split}
	\left||x|^\alpha \cdot \left( \frac{x}{|x|} g(\theta) + \frac{x^\perp}{|x|}g'(\theta) \right)\right|_{C^\alpha(\mathbb{R}^2)} \le C\V g\V_{C^{1,\alpha}(S^1)}.
	\end{split}
	\end{equation*}
	
	In the general case, recall that there is a map $\Psi : \mathbb{R}^2 \rightarrow \mathbb{R}^2$ that locally maps the path $\Omega$ to a union of exact symmetric corners. Then one just take the level set \begin{equation*}
	\begin{split}
	\phi(x) = \tilde{\phi} \circ \Psi(x),\quad \tilde{\phi}(z) = |z| \cdot g(\theta(z)).
	\end{split}
	\end{equation*} Then taking the gradient gives \begin{equation*}
	\begin{split}
	\nabla\phi = (\nabla\tilde{\phi}) \circ \Psi \cdot \nabla\Psi,
	\end{split}
	\end{equation*} and recalling that $\nabla\Psi = I + M$ with a matrix $M \in C^{1,\alpha}(\mathbb{R}^2)$ and $|M| \le C|x|^{1+\alpha}$, it is direct to show that, using bounds given in Lemma \ref{lem:calculus} in the Appendix, $|\nabla^\perp\phi|$ has a lower bound on $\partial\Omega$ as well as $\nabla^\perp\phi \in \mathring{C}^\alpha$. \end{proof}
 
\medskip

\textbf{A brief outline of the proof.} The proof of Theorem \ref{thm:symmetric_corner} will be completed in the following two subsections. In Subsection \ref{subsec:local_estimate}, we prove a frozen-time $C^{1,\alpha}$-estimate pertaining to the boundary of the patch near the corner. After that, we conclude the proof in Subsection \ref{subsec:proof_main} by combining the local estimate together with the $\mathring{C}^\alpha$-result. In the following subsections, we explore some consequences of Theorem \ref{thm:symmetric_corner} and several possible extensions. 
 
\subsection{Local $C^{1,\alpha}$-estimate near the corner}\label{subsec:local_estimate}

To complete the proof of the main statement, it needs to be argued that right at the origin, the $C^{1,\alpha}$ norms of the boundary curves does not blow up at any finite time. Near the corner, it does not seem appropriate to use a level set function which is uniformly $C^{1,\alpha}$. Instead, we will show via a direct computation that the velocity is uniformly $C^{1,\alpha}$ on the boundary, up to the origin. 

To get an idea of how such a statement could be true, one may first take the case of exact (either infinite or localized) sectors meeting symmetrically at the origin. While the velocity gradient associated with a single sector diverges logarithmically at the origin, with coefficient depending on the angle, it was established in the previous work of the first author \cite[Section 6]{E1}  that those logarithmic terms precisely cancel out when the sectors are arranged in an $m$-fold symmetric fashion. Even after cancellations of the divergent terms, the velocity gradient has a part which is a smooth function of the angle only (that is, a $C^\infty$-function of the variable $\tan^{-1}(x_2/x_1)$) and hence it only belongs to $L^\infty$ and not better. However, a key observation we make is that a smooth function of the angle on the plane is actually $C^{k,\alpha}$-smooth when restricted onto any $C^{k,\alpha}$-curve passing through the origin. 

Next, one can consider the case where the patch $\Omega_1$ is given (locally) by an exact sector with two $C^{1,\alpha}$-cusps attached at its sides.\footnote{Here, by a $C^{1,\alpha}$-cusp, we mean a region bounded between two $C^{1,\alpha}$-curves meeting at the origin with the same slope.} Then, from the above, we know that the velocity gradient coming from the sector is $C^{1,\alpha}$-smooth along the boundary curves of $\Omega_1$. Moreover, as a consequence of the work of Friedman and Vel\'azquez \cite{FV}, we also have that the velocity gradient coming from a $C^{1,\alpha}$-cusp is actually $C^{\alpha}$ along each piece of the boundary. 

In the general setting, though, it may happen that a boundary curve of $\Omega_1$ oscillates infinitely often around its tangent line at the origin. Therefore, we have simply chosen to estimate the $C^\alpha$-norm of $\nabla u$ with brute force by directly integrating the kernel. 

Before we begin the estimate, let us recall an explicit representation formula for the velocity gradient associated with vorticity $\chi_{\Omega}$ \cite{BeCo}: \begin{equation*}
\begin{split}
\nabla u(x) = \frac{1}{2\pi} p.v.\, \int_\Omega \frac{\sigma(x-y)}{|x-y|^2} dy + \frac{1}{2}\begin{pmatrix}
0 & -1 \\
1 & 0 
\end{pmatrix} \chi_{\Omega},
\end{split}
\end{equation*} where the characteristic function $\chi_{\Omega}$ is defined to be $1/2$ on $\partial\Omega$. The $2\times 2$ symmetric matrix $\sigma(z)$ is \begin{equation*}
\begin{split}
\frac{1}{|z|^2}\begin{pmatrix}
2z_1z_2 & z_2^2 - z_1^2 \\
z_2^2 - z_1^2 & -2z_1z_2
\end{pmatrix}.
\end{split}
\end{equation*} In particular, note that this formula provides a decomposition of $\nabla u$ into its symmetric and anti-symmetric parts. The anti-symmetric part is completely smooth on the patch, so it is only necessary to deal with the symmetric part, given by a principal value integration  {against a $-2$-homogeneous kernel.}

For the convenience of the reader, let us briefly recall the geometric setup for $\Omega$ and $\Omega_1$. We assume that a patch $\Omega_1$ is locally given by the region between two $C^{1,\alpha}$ curves meeting at the origin: to be precise, there exists some $\gamma > 0$, so that the boundary of $\Omega_1$ in the region $[-\gamma,\gamma]^2$ is given by \begin{equation*}
\begin{split}
\{ (y_1,y_2) : g(y_1) <  y_2 < f(y_1)  \}
\end{split}
\end{equation*} where $g, f$ belong to $C^{1,\alpha}[0,\gamma]$ with $f(0) = 0 = g(0)$ and $f'(0) > 0 > g'(0)$. 
Then, we consider the disjoint union $\Omega = \cup_{j=1}^4 \Omega_j$ where $\Omega_j  = R_{\pi(j-1)/2}(\Omega_1)$. Here, we are assuming that the patch is 4-fold symmetric ($m = 4$) just for the simplicity of notation. The other cases can be treated similarly, using the results of \cite{E1}. 

\begin{figure}
	\includegraphics[scale=1.0]{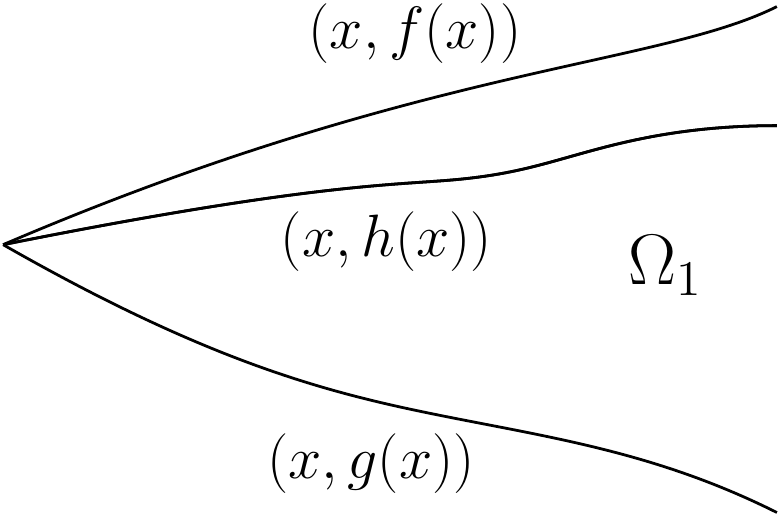} 
	\centering
	\caption{Description of the patch near the corner.}
	\label{fig:local_symm_patch}
\end{figure}

The value of $\delta > 0$  {is chosen as follows} \begin{equation*}
\begin{split}
\delta^\alpha = \frac{1}{10} \min\left\{ \left(\frac{\gamma}{2}\right)^\alpha, \V f\V^{-1}_{C^{1,\alpha}[0,\gamma]} ,\V g\V^{-1}_{C^{1,\alpha}[0,\gamma]} \right\}.
\end{split}
\end{equation*} Moreover, without loss of generality it can be assumed that on the interval $x \in [0,\delta]$, \begin{equation*}
\begin{split}
\frac{1}{2}|x| < |f(x)|, |g(x)| \le 2|x|, \qquad |f'(0)|, |g'(0)| \le 2
\end{split}
\end{equation*} (the specific values $1/2$ and $2$ will not play any essential role). 

\begin{lemma}[$C^{\alpha}$-estimate on the velocity gradient]\label{lem:corner_Holder}
	In the above setting, we have a bound \begin{equation}\label{eq:Holder_estimate}
	\begin{split}
	\left\V \frac{d}{dx}   u(x,f(x)) \right\V_{C^\alpha[0,\delta/10]} \le C \left( \V f\V_{C^{1,\alpha}} + \V g\V_{C^{1,\alpha}}\right) \left(1 + \V \nabla u\V_{L^\infty} \right) + C\delta^{-\alpha}.
	\end{split}
	\end{equation}
\end{lemma}
\begin{proof}
	Let us write down explicitly the expression for $\nabla u$ along a $C^{1,\alpha}$-curve $(x,h(x))$, which lies between two boundary curves of $\Omega_1$: \begin{equation*}
	\begin{split}
	g(x) < h(x) \le f(x),\qquad 0 < x \le \delta. 
	\end{split}
	\end{equation*} We begin with $\frac{d}{dx}u_2$: \begin{equation*}
	\begin{split}
	&\frac{d}{dx}u_2(x,h(x)) = \frac{1}{2\pi} \frac{d}{dx} \int_{\Omega} \frac{x - y_1}{(x-y_1)^2 + (h(x)-y_2)^2} dy \\
	&\qquad= \frac{1}{2\pi} \int_{\Omega} \frac{(x-y_1)^2 - (h(x)-y_2)^2 + 2  {h'(x)}(x-y_1)(h(x)-y_2)}{\left( (x-y_1)^2 + (h(x) - y_2)^2 \right)^2} dy \\
	&\qquad= \frac{1}{2\pi} \left[ \int_{\Omega \cap [-\delta,\delta]^2} + \int_{\Omega \backslash [-\delta,\delta]^2} \right] \frac{(x-y_1)^2 - (h(x)-y_2)^2 + 2  {h'(x)}(x-y_1)(h(x)-y_2)}{\left( (x-y_1)^2 + (h(x) - y_2)^2 \right)^2} dy
	\end{split}
	\end{equation*} where we have separated contribution from the bulk of the patch.\footnote{Strictly speaking the integrals are defined by the principal value. We suppress from writing it out in the computations below.} The contribution from the bulk can be trivially bounded in $C^\alpha$ using the decay of the kernel:  \begin{equation*}
	\begin{split}
	&\frac{1}{|x-x'|^\alpha} \left|\int_{|y| > \delta} \left(\nabla K(x)  - \nabla K(x') \right) \chi_{\Omega} dy \right| \\
	&\qquad \le C|x-x'|^{1-\alpha} \int_{|y| > \delta} \frac{1}{|x-y|^3} + \frac{1}{|x'-y|^3} dy \le C\delta^{-\alpha}. 
	\end{split}
	\end{equation*}
	Now, let us separately consider two integrals \begin{equation*}
	\begin{split}
	I_1(x) := \int_{\Omega \cap [-\delta,\delta]^2} \frac{(x-y_1)^2 - (h(x)-y_2)^2}{\left( (x-y_1)^2 + (h(x) - y_2)^2 \right)^2} dy
	\end{split}
	\end{equation*} and \begin{equation*}
	\begin{split}
	I_2(x) :=  {h'(x)} \int_{\Omega \cap [-\delta,\delta]^2} \frac{2 (x-y_1)(h(x)-y_2)}{\left( (x-y_1)^2 + (h(x) - y_2)^2 \right)^2} dy.
	\end{split}
	\end{equation*} We shall only consider $I_1$, and just briefly comment on the other term $I_2$ below.
	One can further write: \begin{equation*}
	\begin{split}
	I_1(x) = \sum_{j=1}^4 I_1^j(x), 
	\end{split}
	\end{equation*} where \begin{equation*}
	\begin{split}
	I_1^j(x) = \int_{\Omega_j \cap [-\delta,\delta]^2} \frac{(x-y_1)^2 - (h(x)-y_2)^2}{\left( (x-y_1)^2 + (h(x) - y_2)^2 \right)^2} dy
	\end{split}
	\end{equation*} (recall that $\Omega_j := R_{\pi (j-1)/2}(\Omega_1)$). We have, after integrating in $y_2$, \begin{equation*}
	\begin{split}
	I_1^1(x) &= \int_{0 \le y_1 \le \delta} \int_{g(y_1) \le y_2 \le f(y_1)} \frac{(x-y_1)^2 - (h(x)-y_2)^2}{\left( (x-y_1)^2 + (h(x) - y_2)^2 \right)^2} dy_2 dy_1 \\
	&= \int_{0}^{\delta} \left[ \frac{h(x)-f(z)}{(x-z)^2 + (h(x) - f(z))^2} - \frac{h(x)-g(z)}{(x-z)^2 + (h(x) - g(z))^2} \right] dz
	\end{split}
	\end{equation*} (we have renamed $y_1$ by $z$ for simplicity). Similarly, \begin{equation*}
	\begin{split}
	I_1^2(x) &= \int_{0 \le y_2 \le \delta} \int_{-f(y_2) \le y_1 \le -g(y_2)} \frac{(x-y_1)^2 - (h(x)-y_2)^2}{\left( (x-y_1)^2 + (h(x) - y_2)^2 \right)^2} dy_1 dy_2 \\
	&= \int_0^\delta \left[ \frac{x + f(z)}{(x+f(z))^2 + (h(x) - z)^2} - \frac{x+g(z)}{(x+g(z))^2 + (h(x)-z)^2}  \right] dz .
	\end{split}
	\end{equation*} 
	
	\medskip
	
	\noindent \textbf{Claim.} The integral \begin{equation}\label{eq:Holder_expression}
	\begin{split}
	\int_0^\delta \left[\frac{h(x)-f(z)}{(x-z)^2 + (h(x) - f(z))^2} + \frac{x + f(z)}{(x+f(z))^2 + (h(x) - z)^2} \right] dz 
	\end{split}
	\end{equation} defines a $C^\alpha$-function of $0 \le x \le \delta/10$ with $C^\alpha$-norm bounded by the right hand side of \eqref{eq:Holder_estimate}. 
	
	\medskip
	
	Once we show the \textbf{Claim} (together with the upper bound stated in \eqref{eq:Holder_estimate}) for $h = f$ and $h = g$, this concludes the proof that $I_1(x)$ belongs to $C^\alpha$, since each of $I_1^1 + I_1^2$ and $I_1^3 + I_1^4$ belongs to $C^\alpha$, by symmetry. A similar argument can be given for the other term $I_2(x)$: we write it as \begin{equation*}
	\begin{split}
	I_2(x) = f'(x) \sum_{j=1}^4 I_2^j(x),
	\end{split}
	\end{equation*} where \begin{equation*}
	\begin{split}
	I_2^j(x) = \int_{ \Omega_j \cap [-\delta,\delta]^2} \frac{2(x-y_1)(h(x) - y_2)}{((x-y_1)^2 + (h(x) - y_2)^2)^2} dy . 
	\end{split}
	\end{equation*} Then, we can integrate each of $I_2^j$ once with respect to either $y_1$ or $y_2$, resulting in similar expressions as above.

	Let us consider the case $h(z) \equiv f(z)$, which is actually the most difficult case. In this specific case, we rewrite the integrand in \eqref{eq:Holder_expression} as \begin{equation*}
	\begin{split}
	&\left[\frac{f(x)-f(z)}{(x-z)^2 + (f(x) - f(z))^2} - \frac{f'(x)}{1+(f'(x))^2} \cdot \frac{1}{x-z} \right] \\ &\qquad+ \left[\frac{x + f(z)}{(x+f(z))^2 + (f(x) - z)^2} - \frac{f'(x)}{1+(f'(x))^2} \cdot \frac{F}{x+f(z)} \right] + \frac{f'(x)}{1+(f'(x))^2} \left( \frac{1}{x-z} + \frac{F}{x+f(z)} \right)
	\end{split}
	\end{equation*} where $F := f'(0) > 0$. Let us first estimate in $C^\alpha$ the last term, which we further rewrite as: \begin{equation*}
	\begin{split}
	\frac{f'(x)}{1 + (f'(x))^2}  \left[ \left( \frac{1}{x-z} + \frac{F}{x+Fz}\right) + F \cdot \frac{Fz - f(z)}{(x+f(z))(x+Fz)} \right].
	\end{split} 
	\end{equation*} Since $f' \in C^\alpha$, it suffices to estimate in $C^\alpha$ the integrals of two terms in the large brackets. Regarding the first term, one just explicitly evaluate that  \begin{equation*}
	\begin{split}
	\int_0^\delta \frac{1}{x-z} + \frac{F}{x+Fz} dz = \log\left( \frac{F\delta + x}{\delta - x} \right),
	\end{split}
	\end{equation*} (defined by the principal value) which is clearly bounded in $C^\alpha$ by the right hand side of \eqref{eq:Holder_expression}  {for $x\leq \frac{\delta}{10}$.} Note that the logarithmically divergent terms (as $x \rightarrow 0^+$) present in each of the integrals cancel each other exactly. Regarding the second term, we first note that it is uniformly bounded: \begin{equation*}
	\begin{split}
	\left| \int_0^\delta \frac{Fz - f(z)}{(x+f(z))(x+Fz)} dz \right| \le CF \int_0^\delta \frac{z}{x^2 + (Fz)^2} dz \le C \frac{F}{1+F^2}. 
	\end{split}
	\end{equation*} To bound the $C^\alpha$-norm, we need to estimate for $0 \le x < x'$  \begin{equation*}
	\begin{split}
	\int_{0}^\delta\frac{1}{|x-x'|^\alpha} \cdot |Fz-f(z)| \cdot \left| \frac{1}{(x+f(z))(x+Fz)} - \frac{1}{(x'+f(z))(x'+Fz)}  \right|dz,
	\end{split}
	\end{equation*} and simply using that $|Fz - f(z)| \le \V f\V_{C^{1,\alpha}} |z|^{1+\alpha}$, we bound the above by \begin{equation*}
	\begin{split}
	&C\V f\V_{C^{1,\alpha}} \cdot \int_0^\delta \frac{|x-x'|^{1-\alpha}z^{1+\alpha}\left(x + x' + Fz\right)}{(x+f(z))(x+Fz)(x'+f(z))(x'+Fz)} dz \\
	&\qquad\le C\V f\V_{C^{1,\alpha}} \cdot \int_0^\infty \frac{|x-x'|^{1-\alpha}z^{1+\alpha}\left(x + x' + z\right)}{(x+z)^2(x'+z)^2} dz \\
	&\qquad\le C\V f\V_{C^{1,\alpha}} \cdot \int_0^\infty \frac{x^{1-\alpha}z^{1+\alpha}\left(x + z\right)}{(x+z)^2z^2} dz \le C \V f\V_{C^{1,\alpha}}.
	\end{split}
	\end{equation*} It remains to estimate \begin{equation*}
	\begin{split}
	T_1(x) = \int_0^\delta \left[\frac{f(x)-f(z)}{(x-z)^2 + (f(x) - f(z))^2} - \frac{f'(x)}{1+(f'(x))^2} \cdot \frac{1}{x-z} \right]dz
	\end{split}
	\end{equation*} and \begin{equation*}
	\begin{split}
	T_2(x) = \int_0^\delta   \left[\frac{x + f(z)}{(x+f(z))^2 + (f(x) - z)^2} - \frac{f'(x)}{1+(f'(x))^2} \cdot \frac{F}{x+f(z)} \right] dz.
	\end{split}
	\end{equation*}
	We begin with $T_1(x)$. After a bit of re-arranging, we have \begin{equation*}
	\begin{split}
	T_1(x) &= \int_0^\delta \frac{1}{x-z} \left[ \frac{\frac{f(x)-f(z)}{x-z}}{1 + \left(\frac{f(x)-f(z)}{x-z}\right)^2} - \frac{f'(x)}{1 + (f'(x))^2} \right] dz  \\
	& = \int_0^\delta \frac{1}{x-z}   \left[ \left(\frac{1}{1 + \left(\frac{f(x)-f(z)}{x-z}\right)^2} - \frac{1}{1 + (f'(x))^2} \right) \frac{f(x) - f(z)}{x-z} \right. \\
	&\qquad\qquad\qquad \quad \left.  + \frac{1}{1 + (f'(x))^2}\left( \frac{f(x) - f(z)}{x-z} - f'(x) \right) 
  \right]dz.
	\end{split}
	\end{equation*} To begin with, we state as a lemma the $C^\alpha$-estimate for the latter term (dropping the multiplicative factor which belongs to $C^\alpha$):
	\begin{lemma}\label{lem:m0} We have 
	 \begin{equation*}
	\begin{split}
	\left\Vert \int_0^\delta \frac{1}{x-z} \left(\frac{f(x)-f(z)}{x-z} - f'(x)\right) dz \right\Vert_{C^\alpha[0,\frac{\delta}{10}]} \le C \Vert f\Vert_{C^{1,\alpha}}. 
	\end{split}
	\end{equation*} 
	\end{lemma}
	\begin{proof}[Proof of Lemma \ref{lem:m0}]
	Take two points $0 \le x < x' < \delta/10$, and let us further assume that $|x| \ge |x'-x|$ (the other case is simpler). We need to take \begin{equation}\label{eq:4integrals}
	\begin{split}
	&\frac{1}{|x-x'|^\alpha}\int_0^\delta \frac{1}{x-z}\left( \frac{{f}(x) -  {f}(z)}{x-z} - f'(x) \right) - \frac{1}{x'-z}\left( \frac{{f}(x') - {f}(z)}{x'-z}  - f'(x')\right) dz \\
	&\quad = \frac{1}{|x-x'|^\alpha} \left[ \int_0^{x-|x-x'|/2} + \int_{x-|x-x'|/2}^{x+|x-x'|/2} + \int_{x'-|x-x'|/2}^{x'+|x-x'|/2} +  \int_{x'+|x-x'|/2}^\delta \right] \, \cdots \, dz =: I + II + III+IV. 
	\end{split}
	\end{equation} To begin with, we treat the second term \eqref{eq:4integrals}: we simply use the bound \begin{equation*}
	\begin{split}
	| {f}(x) -  {f}(z) - f'(x)(x-z)| \le \V f\V_{C^{1,\alpha}}|x-z|^{1+\alpha}
	\end{split}
	\end{equation*} (and similarly for $x$ replaced by $x'$) to bound   \begin{equation*}
	\begin{split}
	|II| \le \frac{\V f\V_{C^{1,\alpha}}}{|x-x'|^\alpha}\int_{x-|x-x'|/2}^{x+|x-x'|/2} |x-z|^{\alpha-1} + \left( |x-x'| + |x-z| \right)^{\alpha-1}dz \le C \V f\V_{C^{1,\alpha}}. 
	\end{split}
	\end{equation*} The term $III$ from \eqref{eq:4integrals} can be treated in a parallel way. Turning to the first integral, we rewrite as \begin{equation*}
	\begin{split}
	I = \frac{1}{|x-x'|^\alpha}&\int_0^{x - |x-x'|/2} \frac{({f}(x) - {f}(z) - f'(x)(x-z)) - ({f}(x') - {f}(z) - f'(x')(x'-z)}{(x-z)^2}  \\ 
	&\qquad\qquad\qquad\qquad+ \left(\frac{1}{(x-z)^2} - \frac{1}{(x'-z)^2}\right) ({f}(x') - {f}(z) - f'(x')(x'-z))dz
	\end{split}
	\end{equation*} Note that the numerator of the first term equals \begin{equation*}
	\begin{split}
	\left(f(x) - f(x') - f'(x)(x-x') \right) + \left( (x' - x) + (x - z) \right) (f'(x') - f'(x)) ,
	\end{split}
	\end{equation*} and simply using the bounds \begin{equation*}
	\begin{split}
	|f(x) - f(x') - f'(x)(x-x') | \le C \V f \V_{C^{1,\alpha}}|x-x'|^{1+\alpha},\qquad |f'(x) -f'(x')| \le C \V f \V_{C^{1,\alpha}}|x-x'|^\alpha,
	\end{split}
	\end{equation*} we bound the first term by $C\V f \V_{C^{1,\alpha}}$. The second one can be bounded by \begin{equation*}
	\begin{split}
	\V f \V_{C^{1,\alpha}}|x-x'|^{1-\alpha} \int_0^{x - |x-x'|/2}  \frac{|x-z| + |x-x'|}{(x-z)^2} |x'-z|^{\alpha-1} dz,
	\end{split}
	\end{equation*} and after a change of variable $v := (x-z)/|x'-x|$, \begin{equation*}
	\begin{split}
	\le C\V f \V_{C^{1,\alpha}} |x-x'|^{1-\alpha} \cdot |x-x'|^{\alpha -1} \int_{1/2}^{\infty} \frac{(1+v)^{\alpha}}{v^2}dv \le C\V f \V_{C^{1,\alpha}} .
	\end{split}
	\end{equation*} Now the term $IV$ from \eqref{eq:4integrals} can be treated in an analogous fashion. This gives the lemma.
	\end{proof}
	To finish the estimate of $T_1$, we still need to consider the expression \begin{equation*}
	\begin{split}
	&\int_0^\delta \frac{1}{x-z} \cdot \frac{f(x)-f(z)}{x-z} \cdot \left[ \frac{1}{1 + \left(\frac{f(x)-f(z)}{x-z}\right)^2} - \frac{1}{1+(f'(x))^2} \right]dz \\
	&\qquad = -\frac{1}{1+(f'(x))^2} \int_0^\delta \frac{1}{x-z}  \cdot \frac{f(x)-f(z)}{x-z}  \cdot \frac{\left(\frac{{f}(x)-{f}(z)}{x-z} - f'(x) \right) \cdot \left(\frac{{f}(x)-{f}(z)}{x-z} + f'(x)\right)}{1 + \left( F + \frac{\tilde{f}(x)-\tilde{f}(z)}{x-z} \right)^2}dz,
	\end{split}
	\end{equation*} where $F = f'(0)$ and $\tilde{f}(x) = f(x) - Fx$. Consider the expansion \begin{equation*}
	\begin{split}
	&\frac{1}{1 + \left( F + \frac{\tilde{f}(x)-\tilde{f}(z)}{x-z} \right)^2} = \frac{1}{1 + F^2 + \left( F + \frac{\tilde{f}(x)-\tilde{f}(z)}{x-z} \right)^2 - F^2} \\
	&\qquad= \frac{1}{1+F^2} \cdot \sum_{m\ge 0} (-1)^m\left( \frac{1}{1+F^2}\right)^m \cdot \left( \left(F + \frac{\tilde{f}(x)-\tilde{f}(z)}{x-z} \right)^2 - F^2\right)^m \\
	&\qquad = \frac{1}{1+F^2} \cdot \sum_{m\ge 0} (-1)^m\left( \frac{1}{1+F^2}\right)^m \cdot \left( 2F + \frac{\tilde{f}(x)-\tilde{f}(z)}{x-z}   \right)^m  \left(\frac{\tilde{f}(x)-\tilde{f}(z)}{x-z} \right)^m
	\end{split}
	\end{equation*} which is convergent simply because $\V \tilde{f}'\V_{L^\infty[0,\delta]} \le 1/10$ from our choice of $\delta$. Inspecting the terms, to estimate $T_1$, it suffices to obtain boundeness of \begin{equation*}
	\begin{split}
	\sum_{ m \ge 0} \left( \frac{1}{1+F^2}\right)^m \left\Vert   \int_0^\delta \frac{1}{x-z} \left( 2F + \frac{\tilde{f}(x)-\tilde{f}(z)}{x-z}   \right)^m \left( \frac{\tilde{f}(x)-\tilde{f}(z)}{x-z} \right)^m \left( \frac{f(x) -f(z)}{x-z}  - f'(x) \right) dz \right\Vert_{C^\alpha[0,\frac{\delta}{10}]}
	\end{split}
	\end{equation*} and its simple variants.  For any $m \ge 1$, we claim the bound (recall that $\V \tilde{f}'\V_{L^\infty[0,\delta]} \le 1/10$) \begin{equation}\label{eq:final}
	\begin{split}
	\left\V \int_0^\delta \frac{1}{x-z} \left( \frac{\tilde{f}(x)-\tilde{f}(z)}{x-z} \right)^m \left( \frac{f(x) -f(z)}{x-z}  - f'(x) \right) dz \right\V_{C^\alpha[0,\frac{\delta}{10}]} \le \frac{Cm  }{10^m} \V f\V_{C^{1,\alpha}[0,\delta]}.
	\end{split}
	\end{equation} We have already treated the case $m = 0$ in Lemma \ref{lem:m0}.  Let us sketch the proof of \eqref{eq:final}, which is completely parallel. Defining \begin{equation*}
	\begin{split}
	H(x,z) = \frac{\tilde{f}(x)-\tilde{f}(z)}{x-z}, 
	\end{split}
	\end{equation*} we need to treat \begin{equation*}
	\begin{split}
	&\frac{1}{|x-x'|^\alpha}\int_0^{\delta} dz \left[ \frac{1}{x-z}H^m(x,z)\left( \frac{f(x) -f(z)}{x-z}  - f'(x) \right) - \frac{1}{x'-z}H^m(x',z)\left( \frac{f(x') -f(z)}{x'-z}  - f'(x') \right) \right] \\
	&\qquad = I + II + III + IV, 
	\end{split}
	\end{equation*} where, exactly as in the proof of Lemma \ref{lem:m0}, the terms $I, II, III, IV$ correspond to integration over $[0,x-|x-x'|/2], [x-|x-x'|/2,x+|x-x'|/2], [x+|x-x'|/2,x'+|x-x'|/2], [x'+|x-x'|/2,\delta]$, assuming $0\le x<x'<\delta/10$ and $|x|\ge |x'-x|$ for simplicity. In regions $II$ and $III$, one can simply use  \begin{equation*}
	\begin{split}
	| {f}(x) -  {f}(z) - f'(x)(x-z)| \le \V f\V_{C^{1,\alpha}}|x-z|^{1+\alpha}
	\end{split}
	\end{equation*} (and  for $x$ replaced by $x'$)  since the length of the integration domain is of order $|x-x'|$. For the term $I$ ($IV$ can be treated similarly), we just write \begin{equation*}
	\begin{split}
	& \left[ \frac{1}{x-z}H^m(x,z)\left( \frac{f(x) -f(z)}{x-z}  - f'(x) \right) - \frac{1}{x'-z}H^m(x',z)\left( \frac{f(x') -f(z)}{x'-z}  - f'(x') \right) \right]  \\
	&\qquad = \left[ \frac{1}{x-z}H^m(x,z)\left( \frac{f(x) -f(z)}{x-z}  - f'(x) \right) - \frac{1}{x'-z}H^m(x,z)\left( \frac{f(x') -f(z)}{x'-z}  - f'(x') \right) \right] \\
	&\qquad\quad + \left[ \frac{1}{x'-z}H^m(x,z)\left( \frac{f(x') -f(z)}{x'-z}  - f'(x') \right) - \frac{1}{x'-z}H^m(x',z)\left( \frac{f(x') -f(z)}{x'-z}  - f'(x') \right) \right] 
	\end{split}
	\end{equation*} and then the first term on the right hand side is treated in the exact same way as in Lemma \ref{lem:m0}, resulting in the constant $10^{-m}$ thanks to the size of $H^m$ in $L^\infty$. For the last term, we simply rewrite \begin{equation*}
	\begin{split}
	\frac{1}{x'-z}\left( \frac{f(x') -f(z)}{x'-z}  - f'(x') \right)(H(x,z)-H(x',z))(H^{m-1}(x,z) + H^{m-2}(x,z)H(x',z)+\cdots + H^{m-1}(x',z))
	\end{split}
	\end{equation*} and then it can be estimates again in the same way, resulting in the constant $m10^{-m}$. With yet another parallel argument, it is not difficult to see \begin{equation*}
	\begin{split}
	& \left\Vert   \int_0^\delta \frac{1}{x-z} \left( 2F + \frac{\tilde{f}(x)-\tilde{f}(z)}{x-z}   \right)^m \left( \frac{\tilde{f}(x)-\tilde{f}(z)}{x-z} \right)^m \left( \frac{f(x) -f(z)}{x-z}  - f'(x) \right) dz \right\Vert_{C^\alpha[0,\frac{\delta}{10}]} \\
	&\qquad \le \frac{Cm(1+2F)^m  }{10^m} \V f\V_{C^{1,\alpha}[0,\delta]}.
	\end{split}
	\end{equation*}  This concludes the argument for $T_1(x)$. The other term $T_2(x)$ can be treated similarly, and it is simpler since the corresponding integral is less singular than that of $T_1$.

	We now sketch a proof that \textbf{Claim} holds in the case $h(x) \equiv g(x)$. In this case, the arguments are simpler since we have a gap $|h(x) - g(x)| \gtrsim |x|$. It suffices to show that the differences \begin{equation*}
	\begin{split}
	\int_0^\delta \left[ \frac{g(x) - f(z)}{(x-z)^2 + (g(x) - f(z))^2} - \frac{Gx - Fz}{(x-z)^2 + (Gx - Fz)^2}  \right] dz,
	\end{split}
	\end{equation*} \begin{equation*}
	\begin{split}
	\int_0^\delta   \left[\frac{x + f(z)}{(x+f(z))^2 + (g(x) - z)^2} - \frac{x + Fz}{(x+Fz)^2 + (Gx - z)^2} \right] dz,
	\end{split}
	\end{equation*} and \begin{equation*}
	\begin{split}
	\int_0^\delta \left[ \frac{Gx - Fz}{(x-z)^2 + (Gx - Fz)^2} + \frac{x + Fz}{(x+Fz)^2 + (Gx - z)^2} \right] dz
	\end{split}
	\end{equation*} belong to $C^\alpha$ with appropriate bounds, where $G := g'(0) < 0$. To begin with, the last integral can be evaluated directly: \begin{equation*}
	\begin{split}
	&- \frac{1}{1+ F^2} \left[ \tan^{-1}\left(\frac{(-F+G)x}{-(1+FG)x + (1+F^2)z}\right) + \frac{F}{2} \log\left( (1+G^2)x^2 + (1+F^2)z^2- 2xz(1+FG) \right) \right. \\
	& \qquad\qquad\qquad + \left.\left. \tan^{-1}\left(\frac{x(1+FG)}{(F-G)x + (1+F^2)z}\right) - \frac{F}{2} \log\left( (1+G^2)x^2 + 2(F-G)xz + (1+F^2)z^2 \right) \right] \right|_0^\delta.
	\end{split}
	\end{equation*}  {We claim that the above expression gives a $C^\alpha$-function of $x$. To see this, evaluating the above at $z = \delta$ and $z= 0$, and subtracting gives the following terms (up to multiplicative constants): 
	\begin{equation*}
	\begin{split}
	&\tan^{-1}\left(\frac{(-F+G)x}{-(1+FG)x + (1+F^2)\delta}\right) - \tan^{-1}\left(\frac{-F+G}{-(1+FG)}\right)\\
	&\quad  + \tan^{-1}\left(\frac{x(1+FG)}{(F-G)x + (1+F^2)\delta}\right) - \tan^{-1}\left(\frac{1+FG}{F-G}\right) 
	\end{split}
	\end{equation*} and \begin{equation*}
	\begin{split}
	&\log\left( (1+G^2)x^2 + (1+F^2)\delta^2- 2x\delta(1+FG) \right) - \log\left( (1+G^2)x^2 + 2(F-G)x\delta + (1+F^2)\delta^2 \right). 
	\end{split}
	\end{equation*} Here it is crucial that the logarithmic terms evaluated at $z = 0$ cancel each other. To treat the terms involving $\tan^{-1}$, one can directly compute that \begin{equation*}
	\begin{split}
	\left\V \tan^{-1}\left( \frac{Ax}{Bx + \delta} \right) \right\V_{C^\alpha[0,\delta/10]} \le C(A,B)\delta^{-\alpha}
	\end{split}
	\end{equation*} for nonzero constants $A$ and $B$. Another explicit computation gives that \begin{equation*}
	\begin{split}
	\left\V\log\left( \frac{(1+G^2)x^2 + (1+F^2)\delta^2- 2x\delta(1+FG)}{(1+G^2)x^2 + 2(F-G)x\delta + (1+F^2)\delta^2} \right) \right\V_{C^\alpha[0,\delta/10]} \le C(F,G)\delta^{-\alpha}.
	\end{split}
	\end{equation*}
} Now we return to the first integral, which equals \begin{equation}
	\begin{split}
	&\int_0^\delta \left[ \frac{(x-z)^2 \left( (Gx - g(x)) - (Fz - f(z)) \right)}{\left( (x-z)^2 + (g(x) - f(z))^2 \right) \left( (x-z)^2 + (Gx - Fz)^2 \right) } \right. \\
	&\qquad\qquad + \left.\frac{(g(x) - f(z))(Gx - Fz)\left( (g(x) - f(z)) - (Gx - Fz) \right)}{\left( (x-z)^2 + (g(x) - f(z))^2 \right) \left( (x-z)^2 + (Gx - Fz)^2 \right)} \right] dz
	\end{split}
	\end{equation} Here, the key points are: \begin{itemize}
		\item On the numerator, we gain an extra power of $|x|^{\alpha}$ or $|z|^{\alpha}$, from H\"{o}lder continuity of $f'$ and $g'$. 
		\item The denominator is uniformly bounded from above and below by constant multiples of $x^2 + z^2$. 
	\end{itemize} We sketch the proof of $C^\alpha$-continuity for the first term only, since the second one can be treated similarly. We need to estimate \begin{equation}\label{eq:2terms}
	\begin{split}
	&\frac{1}{|x-x'|^\alpha} \int_0^\delta \left[ \frac{(x-z)^2 \left( (Gx - g(x)) - (Fz - f(z)) \right)}{\left( (x-z)^2 + (g(x) - f(z))^2 \right) \left( (x-z)^2 + (Gx - Fz)^2 \right) } \right. \\
	&\qquad\qquad \qquad\qquad \left. -\frac{(x'-z)^2 \left( (Gx' - g(x')) - (Fz - f(z)) \right)}{\left( (x'-z)^2 + (g(x') - f(z))^2 \right) \left( (x'-z)^2 + (Gx' - Fz)^2 \right) }\right] dz
	\end{split}
	\end{equation} and we may assume $|x-x'| \le |x|$. Let us even further assume that the denominators in \eqref{eq:2terms} are the same, as they are roughly of the same size (and bounded uniformly from below by a constant multiple of $x^2 + z^2$ and $x'^2 + z^2$, respectively). Then, the resulting difference is bounded by: \begin{equation*}
	\begin{split}
	C( \V f \V_{C^{1,\alpha}} + \V g \V_{C^{1,\alpha}}) \frac{1}{|x-x'|^\alpha} \int_0^\delta \frac{|x-x'|\left( |x| + |x-x'| + |z| \right)\left( |x|^{1+\alpha} + |z|^{1+\alpha} \right) }{(x^2 + z^2)^2 }dz
	\end{split}
	\end{equation*} and at this point, the $C^\alpha$-bound simply follows from rescaling the variable $z = xv$. The actual proof can be done for instance by expanding one of the denominators in \eqref{eq:2terms} around the other denominator in a power series as we have done earlier. 
	
	The argument for the other component $\frac{d}{dx}u_1$ is completely analogous. We just note that along a curve $(x,h(x))$, it has the form: \begin{equation*}
	\begin{split}
	\frac{d}{dx} u_1(x,h(x)) &= \frac{1}{2\pi} \frac{d}{dx} \int_{\Omega} \frac{-(h(x) - y_2)}{(x-y_1)^2 + (h(x) - y_2)^2} dy \\
	&= \frac{1}{2\pi} \int_{ \Omega} \frac{-h'(x)\left( (x-y_1)^2 + (h(x)-y_2)^2 \right) + 2(h(x) - y_2)\left( (x-y_1) + h'(x) (h(x)-y_2)\right)}{\left((x-y_1)^2 + (h(x) - y_2)^2\right)^2}dy.
	\end{split}
	\end{equation*}

	\noindent This finishes the proof. 
\end{proof}

\subsection{Proof of the main result}\label{subsec:proof_main}

We are now in a position to complete the proof of Theorem \ref{thm:symmetric_corner}. Let us recall that as a consequence of Theorem \ref{thm:intermediate}, for any $ T > 0$, we have $L^\infty$-bounds \begin{equation}\label{eq:apriori_infty}
\begin{split}
\sup_{t \in [0,T]} \left(\V \nabla \Phi_t \V_{L^\infty} + \V \nabla \Phi_t^{-1} \V_{L^\infty} + \V \nabla u_t\V_{L^\infty} \right) \le C(T),
\end{split}
\end{equation} and moreover, for any $r > 0$, \begin{equation}\label{eq:apriori_alpha}
\begin{split}
\sup_{t \in [0,T]} \left\V \nabla u_t \right\V_{C^{\alpha}(\mathbb{R}^2\backslash B_0(r) )} \le C(T) r^{-\alpha}. 
\end{split}
\end{equation} As in the case of Theorem \ref{thm:intermediate}, the issue of local well-posedness is deferred to the Appendix (Proposition \ref{prop:LWP2}); hence, we shall assume that at least for some short time interval $[0,T_1]$, each piece of the boundary of $\Omega_1(t)$ remains uniformly $C^{1,\alpha}$ up to the origin. 
\begin{proof}[Proof of Theorem \ref{thm:symmetric_corner}]
	 {
	We shall fix some $T > 0$ and obtain a priori estimates which guarantee that for all $t \in [0,T]$, the boundary of $\Omega_1(t)$ is given by two $C^{1,\alpha}$-curves $f_t$ and $g_t$, after rotating the plane if necessary, on some interval of $x \in [0,\delta_t]$ for $\delta_t > 0$. We initially take $\delta_t = \delta_0$ for all $t \in [0,T]$ but need to shrink its value whenever necessary (but in a way depending only on $T$) in the following argument. This will be sufficient as $T > 0$ was arbitrary. Then the a priori estimates can be justified along the proof of the local well-posedness given in the Appendix. 
	
	Since the patch in general rotates around the origin, we need to work within subintervals of time of the form $[mT_0, (m+1)T_0]$ for some $0 < T_0 \le T$ (to be determined below) which depends only on $T$; at the end of each subinterval, we initialize the patch again.  This is allowed as the a priori estimates we obtain will not depend on $m$ but only on $T$. From now on, we shall assume that $t \in [0,T_0]$.  
	
	At the initial time, we may assume that $f_0$ and $g_0$ satisfy \begin{equation*} 
	\begin{split}
	& 0 < c_0 \le  |f_0'(x)|, |g_0'(x)| \le C_0 < +\infty 
	\end{split}
	\end{equation*} on $x \in [0,\delta_0]$ for some constants $c_0, C_0$. Unfortunately, $f_t$ and $g_t$ themselves do not obey a simple evolution equation. We instead work directly with the particle trajectories \begin{equation*}
	\begin{split}
	\eta^1(t,x):= \Phi^1(t,(x,f_0(x))), \qquad \eta^2(t,x) := \Phi^2(t,(x,f_0(x))) 
	\end{split}
	\end{equation*} and \begin{equation*} 
	\begin{split}
	&\teta^1(t,x):= \Phi^1(t,(x,g_0(x))), \qquad \teta^2(t,x) := \Phi^2(t,(x,g_0(x))) 
	\end{split}
	\end{equation*} (which are well-defined on $x \in [0,\delta_0]$), and apply the inverse function theorem to recover bounds on $f_t$ and $g_t$. Then, since $\rd_t {\eta}(t,x) = u(t,\eta(t,x))$, we have upon differentiating \begin{equation}\label{eq:diff}
	\begin{split}
	\frac{\partial}{\partial t} \left(\frac{\partial}{\partial x} \eta(t,x)\right) = \nabla u(t, \eta(t,x)) \left(\frac{\partial}{\partial x} \eta(t,x)\right).
	\end{split}
	\end{equation} First, from \eqref{eq:apriori_infty} we have \begin{equation}\label{eq:infty}
	\begin{split}
	\sup_{x \in [0,\delta_0]}| \pr_x \eta_t| \le C(T) < + \infty, \qquad \inf_{ x \in [0,\delta_0]} |\pr_x\eta_t| \ge c(T) > 0.
	\end{split}
	\end{equation} Moreover, \begin{equation*} 
	\begin{split}
	& \inf_{ x \in [0,\delta_0]} |\rd_x\eta^1_t| \ge c_0- C(T)t 
	\end{split}
	\end{equation*} so that by taking $T_0 = T_0(T)$ small, we may guarantee that $|\rd_x\eta^1_t| > \frac{c_0}{2}$. This guarantees that the function $\eta^1_t$ is invertible, and we denote the inverse by $(\eta^1_t)^{-1}$, which is well-defined on $[0,c\delta_0]$ for some $c = c(T) > 0$. We take $\delta_t = c\delta_0$ on $t \in (0,T_0]$. Using the chain rule and the bound \eqref{eq:infty}, it is easy to obtain \begin{equation}\label{eq:flowinverse-bound}
	\begin{split}
	\sup_{x \in [0,\delta_t]} | \rd_x (\eta_t^1)^{-1} | \le C(T) 
	\end{split}
	\end{equation} and \begin{equation}\label{eq:flowinverse-bound2} 
	\begin{split}
	\V \rd_x (\eta^1_t)^{-1} \V_{C^\alpha[0,\delta_t]} \le \V \rd_x\eta^1_t \V_{C^\alpha[0,\delta_0]} \frac{ \sup | \rd_x (\eta_t^1)^{-1} |  }{( \inf |\rd_x \eta^1_t| )^2} \le C(T) \V \rd_x\eta^1_t \V_{C^\alpha[0,\delta_0]} . 
	\end{split}
	\end{equation} We argue similarly for the other part of the boundary: denoting the inverse of $\teta^1_t$ by $(\teta^1_t)^{-1}$ (which is defined on the same interval of $x$), it can be shown that and $(\teta^1_t)^{-1}$ satisfies the same bounds as in \eqref{eq:flowinverse-bound}--\eqref{eq:flowinverse-bound2}. We now may define \begin{equation*} 
	\begin{split}
	& f_t = \eta_t^2 \circ (\eta_t^1)^{-1}, \quad g_t = \teta_t^2 \circ (\teta_t^1)^{-1} 
	\end{split}
	\end{equation*} on $[0,\delta_t]$. A straightforward computation (again using the chain rule) shows that \begin{equation*} 
	\begin{split}
	& \V f_t \V_{C^{1,\alpha}[0,\delta_t]} \le C(T) ( \V \rd_x(\eta_t^1)^{-1}\V_{C^\alpha} + \V \rd_x \eta^2_t \V_{C^\alpha} ) \le C(T) (\V \rd_x \eta^1_t \V_{C^\alpha} + \V \rd_x \eta^2_t \V_{C^\alpha}  ). 
	\end{split}
	\end{equation*} As the boundary of $\Omega_1(t)$ is given by the graph of two $C^{1,\alpha}$ curves $f_t$ and $g_t$ on $[0,\delta_t]$, we obtain from Lemma \ref{lem:corner_Holder} and \eqref{eq:apriori_alpha} that \begin{equation*}
	\begin{split}
	\V \nabla u_t \circ \eta \V_{C^\alpha[0,\delta_0]} \le C(T) \left( \left\V f_t \right\V_{C^{1,\alpha}[0,\delta_t]} + \left\V g_t \right\V_{C^{1,\alpha}[0,\delta_t]} + \delta_t^{-\alpha} \right).
	\end{split}
	\end{equation*} Now using the previous bounds on $f_t$, $g_t$, and $\delta_t$, \begin{equation*} 
	\begin{split}
	&\V \nabla u_t \circ \eta \V_{C^\alpha[0,\delta_0]} \le C(T, \delta_0)  \V  \pr_x \eta(t) \V_{C^\alpha[0,\delta_0]}. 
	\end{split}
	\end{equation*} Returning to \eqref{eq:diff}, and using the algebra property of the space $C^\alpha$, we deduce an a priori bound \begin{equation*}
	\begin{split}
	\frac{d}{dt}\V  \pr_x \eta(t) \V_{C^\alpha[0,\delta_0]} &\le  \V \nabla u_t\V_{L^\infty} \V \pr_x \eta(t)\V_{C^\alpha[0,\delta_0]} + \V \nabla u_t \circ \eta_t \V_{C^\alpha[0,\delta_0]} \V \pr_x\eta(t)\V_{L^\infty} \\
	& \le  C(T, \delta_0) \V  \pr_x \eta(t) \V_{C^\alpha[0,\delta_0]}. 
	\end{split}
	\end{equation*}
	This shows that  $\V  \pr_x \eta \V_{C^\alpha} $  remains finite with an upper bound depending only on $T$ and $\delta_0$. 
}

	It remains to show the statement regarding the dynamics of the angles. For this purpose, let us decompose \begin{equation*}
	\begin{split}
	\omega = \omega^{homog} + \omega^{cusp} + \omega^{far}, 
	\end{split}
	\end{equation*} where $\omega^{homog}$ is the 0-homogeneous vorticity which is the characteristic function of the $m$-fold symmetrization of the infinite sector \begin{equation*}
	\begin{split}
	\{ (x_1,x_2) : 0 < x_1,  G_t x_1 < x_2 < F_t x_1 \}, \qquad F_t = f_t'(0),\quad G_t = g_t'(0),
	\end{split}
	\end{equation*} and $\omega^{cusp}$ is simply $\chi_{B_0(\delta_t)} \cdot \left( \chi_{\Omega_t} - \omega^{homog} \right)$. To be concrete, modulo $m$-fold symmetry, \begin{equation*}
	\begin{split}
	\omega^{cusp}(x_1,x_2) = \begin{cases}
	+1  &  \mbox{ if }\quad F_t x_1 < x_2 < f_t(x_1) \quad\mbox{ or }\quad g_t(x_1) < x_2 < G_t x_1  \\
-1  &\mbox{ if }\quad F_t x_1 > x_2 > f_t(x_1) \quad\mbox{ or }\quad g_t(x_1) > x_2 > G_t x_1
	\end{cases},
	\end{split}
	\end{equation*} inside the ball $B_0(\delta_t)$. Then $\omega^{far}$ is defined as $\omega - \omega^{homog} - \omega^{cusp}$, and one note that it is supported outside the ball $B_0(\delta_t)$. Then, accordingly, we obtain a decomposition of the velocity \begin{equation*}
	\begin{split}
	u = u^{homog} + u^{cusp} + u^{far},
	\end{split}
	\end{equation*} and we claim that $u^{cusp}$ and $u^{far}$ does not effect the dynamics of the tangent lines to the boundary curves $(x_1,f_t(x_1))$ and $(x_1,g_t(x_1))$ at the origin for all times. This clearly follows once we establish that $|u^{cusp}(x)|, |u^{far}(x)| \lesssim |x|^{1+\alpha}$. 
	
	To begin with, the radially homogeneous component $u^{homog}$ induces the same rotation speed on the tangent lines $(x_1,F_tx_1)$ and $(x_1,G_tx_1)$. However, since $\nabla u$ is bounded for all time, the angle between $(x_1,F_tx_1)$ and $(x_1,f_t(x_1))$, and also between $(x_1,G_tx_1)$ and $(x_1,g_t(x_1))$ stays zero. Next, we know that $u^{far}$ is $C^{1,\alpha}$ (indeed, $C^\infty$) inside $B_0(\delta_t/2)$. Therefore, the associated stream function $\psi^{far}$ is $C^{2,\alpha}$ in the ball.\footnote{Here, although $\omega^{far}$ may have non-compact support,  the Poisson problem $ \Delta \psi^{far} = \omega^{far}$ has a unique solution with $\psi^{far} \in W^{2,\infty}(\mathbb{R}^2)$ under $m$-fold rotational symmetry with $m \ge 3$ and $\psi^{far}(0) = 0$; see \cite[Lemma 2.6]{EJ1}.} Taylor expansion gives \begin{equation*}
	\begin{split}
	\psi^{far}(x_1, x_2) = A + Bx_1 + Cx_2 + D(x_1^2 + x_2^2) + Ex_1x_2 + O(|x|^{2+\alpha}).
	\end{split}
	\end{equation*} However, $A = 0$ by assumption and $B = C = E = 0$ is forced under the $m$-fold rotational symmetry. Furthermore, $\Delta\psi^{far} \equiv 0$ near 0 so that $D = 0$. In particular, \begin{equation*}
	\begin{split}
	u^{far}(x_1,x_2) = \nabla^\perp\psi^{far}(x_1,x_2) =  O(|x|^{1+\alpha})
	\end{split}
	\end{equation*}   as $|x| \rightarrow 0$. Lastly, it is known that within each connected component of the complement of the (closure of the) support of $\omega^{cusp}$, the associated velocity $u^{cusp}$ is uniformly $C^{1,\alpha}$ up to the boundary. It follows from our computations in Subsection \ref{subsec:local_estimate} but also directly from the arguments of Friedman and Vel\'azquez \cite{FV} (see the statement of their Lemma in Subsection \ref{subsec:Bertozzi_Constantin}). Then, an identical argument as in the case of $u^{far}$ shows that, this time, $u^{cusp}$ is of order $|x|^{1+\alpha}$ in the complement of the support of $\omega^{cusp}$.  
	The proof is now complete. 
\end{proof}

\subsection{Multiple corners}\label{subsec:cusp_formation}

In the above main result, we have only dealt with the case when there is a single corner in a sector of angle $2\pi/m$ (which serves as a fundamental domain for rotations by multiples of $2\pi/m$). In this case, we have seen that the angle of the patch is preserved for all time. However, one may consider the case when there are several corners (separated from each other by some angle; see Figure \ref{fig:multiple_corners}) in each fundamental domain, and then some interesting dynamics for the angles can be observed. 

We just note that an essentially identical proof carries over to this case to establish global well-posedness of such patches, and also the fact that the angles evolve exactly as in the case of infinite sectors, up to a constant overall rotation.  {In fact, we just need to apply Lemma \ref{lem:corner_Holder} to each piece of the patch.}

We just modify the last item from the Definition \ref{def:C_1alpha_corner} to allow such patches: \begin{itemize}
	\item (multiple $C^{1,\alpha}$ corners) There exists a $C^{1,\alpha}$ diffeomorphism $\Psi:\mathbb{R}^2 \rightarrow \mathbb{R}^2$ of the plane  with $\Psi(0) = 0$ and $\nabla\Psi|_{x=0} = I$, such that for some $\delta > 0$, the image $\Psi(\Omega_1)$ is a union of exact sectors with total angle less than $2\pi/m$: \begin{equation}
	\Psi(\Omega_1) \cap B_0(\delta) = \bigcup_{j = 1}^{k} S_{\beta_k,\beta_k+\zeta_k}  \cap B_0(\delta)
	\end{equation} with some $0 < \zeta_j$ and $-\pi \le \beta_j  < \pi$ satisfying \begin{equation*}
	\begin{split}
	\beta_j + \zeta_j < \beta_{j+1} \quad \mbox{and} \quad \beta_{j+1} + \zeta_{j+1} - \beta_1 < 2\pi/m  \quad \mbox{for all} \quad j = 1, \cdots, k - 1,
	\end{split}
	\end{equation*} (the ordering is well-defined on the interval $[-\pi,\pi]$, assuming without loss of generality that $-\pi \le \beta_1 < 0$). 
\end{itemize} Then as before, we define $\Omega = \cup_{i=0}^{m-1} R_{2\pi i/m}(\Omega_1)$. 

Alternatively, we may describe the patch locally as a union of approximate sectors with angles $\zeta_1, \cdots, \zeta_k$, in counter-clockwise order, with gaps between them $\gamma_{1+1/2}, \cdots, \gamma_{k-1+1/2}$ where $\gamma_{j+1/2} := \beta_{j+1} - \beta_j - \zeta_j$. Note that given some value of $m \ge 3$, the values $\zeta_1, \cdots, \zeta_k$ together with $\gamma_{1+1/2}, \cdots, \gamma_{k-1/2}$  determine the local shape of the patch, up to a rotation of the plane. 

\begin{figure}
	\includegraphics[scale=1.0]{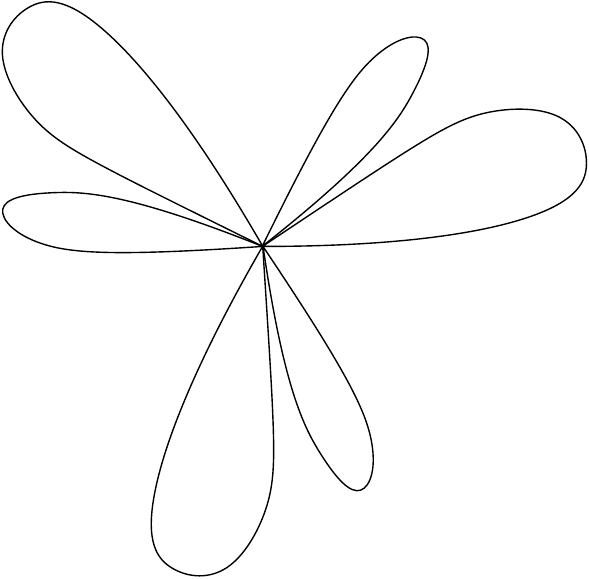} 
	\centering
	\caption{A 3-fold symmetric patch with multiple corners.}
	\label{fig:multiple_corners}
\end{figure}

\begin{corollary}[Dynamics of the angles]\label{cor:dynamics_angle}
	Assume that $\Omega_0$ is a symmetric $C^{1,\alpha}$-patch with multiple corners as defined in the above, with corner angles $\zeta_1(0), \cdots, \zeta_k(0)$ with separation angles $\gamma_{1+1/2}(0),\cdots,\gamma_{k-1/2}(0)$. Then, the angles evolve according to the following system of ordinary differential equations for all $t \in \mathbb{R}$: \begin{equation}\label{eq:ode1}
	\begin{split}
	\frac{d\zeta_j(t)}{dt}  =  C_m  \sin\left(\frac{m}{4}\zeta_j\right) &\left[   \sum_{l=1}^{j-1}  \sin\left(\frac{m}{4}\zeta_l\right) \cos\left( \frac{m}{4}\left( 2(\beta_j -\beta_l) +  (\zeta_j - \zeta_l) \right) \right) \right.\\ &\qquad \left. - \sum_{l=j+1}^k \sin\left(\frac{m}{4}\zeta_l\right) \cos\left( \frac{m}{4}\left( 2(\beta_j -\beta_l) +  (\zeta_j - \zeta_l) \right) \right) \right]
	\end{split}
	\end{equation} and \begin{equation}\label{eq:ode2}
	\begin{split}
	\frac{d\gamma_{j+ 1/2}(t) }{dt} = C_m \sin\left(\frac{m}{4}\gamma_{j+1/2}\right)  &\left[   \sum_{l=1}^{j}  \sin\left(\frac{m}{4}\zeta_l\right) \cos\left( \frac{m}{4}\left( (\beta_{j+1} -\beta_l) +  (\beta_j -\beta_l)  +  (\zeta_j - \zeta_l) \right) \right) \right.\\ &\qquad \left. - \sum_{l=j+1}^k \sin\left(\frac{m}{4}\zeta_j\right) \cos\left( \frac{m}{4}\left( (\beta_{j+1} -\beta_l) +  (\beta_j -\beta_l)  +  (\zeta_j - \zeta_l) \right) \right) \right]
	\end{split}
	\end{equation} with \begin{equation*}
	\begin{split}
	\beta_j - \beta_l &= \left(\gamma_{j-1/2} + \cdots + \gamma_{l+1/2}\right) + \left( \zeta_{j-1} + \cdots + \zeta_l \right), \qquad j > l 
	\end{split}
	\end{equation*} for some constant $C_m > 0$ depending only on $m$. 
\end{corollary}

\begin{proof}
	 {Repeating the arguments given in Section \ref{subsec:proof_main}, it can be shown that the dynamics of the angles in the multiple corners case is identical to the dynamics for the homogeneous case.} It therefore suffices to obtain the 1D system describing the evolution of 0-homogeneous vorticities. On the unit circle, we are given initial vorticity \begin{equation*}
	\begin{split}
	h_0(\theta) = \sum_{i=0}^{m-1} \sum_{j=1}^k S_{\beta_j + 2\pi i/m, \beta_j + \zeta_j + 2\pi i/m}.
	\end{split}
	\end{equation*} Moreover, given $h$, the corresponding angular velocity (counter-clockwise rotation) on the circle is defined explicitly by \begin{equation*}
	\begin{split}
	v(\theta) =  \int_{-\pi/m}^{\pi/m} \left( c^1_m \sin\left(\frac{m}{2}|\theta - \theta'| \right) - c^2_m \right) h(\theta') d\theta'
	\end{split}
	\end{equation*} for some constants $c^1_m > 0$ and $c^2_m$ depending only on $m \ge 3$. Since the integral of $h$ over the circle is conserved in time, one may redefine the angular velocity to be 
	\begin{equation*}
	\begin{split}
	\tilde{v}(\theta) = c^1_m  \int_{-\pi/m}^{\pi/m}\sin\left(\frac{m}{2}|\theta - \theta'| \right)h(\theta')d\theta'
	\end{split}
	\end{equation*} up to an overall rotation. Therefore, \begin{equation*}
	\begin{split}
	\frac{d}{dt} \zeta_j(t) &= \tilde{v}(\beta_j + \zeta_j) - \tilde{v}(\beta_j) \\
	&= \sum_{l=1}^{j-1} c^1_m \int_{\beta_l}^{\beta_l + \zeta_l} \left[ \sin\left( \frac{m}{2}(\beta_j+ \zeta_j - \theta) \right) -  \sin\left( \frac{m}{2}(\beta_j -\theta) \right)  \right] d\theta \\
	&\qquad + \sum_{l=j+1}^{k} c^1_m \int_{\beta_l}^{\beta_l + \zeta_l} \left[ \sin\left( \frac{m}{2}(\theta-\beta_j-\zeta_j) \right) -  \sin\left( \frac{m}{2}(\theta-\beta_j) \right)  \right] d\theta \\
	&= \sum_{l=1}^{j-1} c'_m \sin\left(\frac{m}{4}\zeta_j\right)\sin\left(\frac{m}{4}\zeta_l\right) \cos\left( \frac{m}{4}\left( 2(\beta_j -\beta_l) +  (\zeta_j - \zeta_l) \right) \right) \\
	&\qquad - \sum_{l=j+1}^{k} c'_m \sin\left(\frac{m}{4}\zeta_j\right)\sin\left(\frac{m}{4}\zeta_l\right) \cos\left( \frac{m}{4}\left( 2(\beta_j -\beta_l) +  (\zeta_j - \zeta_l) \right) \right)
	\end{split}
	\end{equation*} (note that the contribution from the $j$-th sector cancels out) and the relations \begin{equation*}
	\begin{split}
	\beta_j - \beta_l &= \left(\gamma_{j-1/2} + \cdots + \gamma_{l+1/2}\right) + \left( \zeta_{j-1} + \cdots + \zeta_l \right), \qquad j > l 
	\end{split}
	\end{equation*} enables us to express the right hand side in terms of $\gamma$'s and $\zeta$'s. Similarly, \begin{equation*}
	\begin{split}
	\frac{d}{dt} \gamma_{j+1/2} (t) &= \tilde{v}(\beta_{j+1}) - \tilde{v}(\beta_j + \zeta_j)\\
	&= \sum_{l=1}^{j} c'_m \sin\left(\frac{m}{4}\gamma_{j+1/2}\right)\sin\left(\frac{m}{4}\zeta_l\right) \cos\left( \frac{m}{4}\left( (\beta_{j+1} -\beta_l) +  (\beta_j -\beta_l)  +  (\zeta_j - \zeta_l) \right) \right) \\
	&\qquad - \sum_{l=j+1}^{k} c'_m \sin\left(\frac{m}{4}\gamma_{j+1/2}\right)\sin\left(\frac{m}{4}\zeta_l\right) \cos\left( \frac{m}{4}\left( (\beta_{j+1} -\beta_l) +  (\beta_j -\beta_l)  +  (\zeta_j - \zeta_l) \right) \right)
	\end{split}
	\end{equation*}
	This finishes the proof. 
\end{proof}

\subsection{Extensions}\label{subsec:extensions}

\subsubsection*{Generality of Serfati and Chemin}

The results of Serfati \cite{Ser1,S1} and Chemin \cite{C,C2,C3} demonstrates that propagation of boundary regularity for smooth patches is just a special instance -- the Euler equations indeed propagates ``striated'' regularity of vorticity. Here we present the version given by Bae and Kelliher \cite{BK1,BK2}. 

To formally state the general result, assume that a family of $C^\alpha(\mathbb{R}^2)$ vector fields $\{ Y_0^\lambda \}_{\lambda \in \Lambda}$ is given, and satisfies the following properties: \begin{equation*}
\begin{split}
\inf_{x \in \mathbb{R}^2} \left( \sup_{\lambda} \left| Y_0^\lambda(x) \right| \right) \ge c_0 > 0
\end{split}
\end{equation*} and \begin{equation*}
\begin{split}
\sup_{\lambda} \left( \V Y_0^\lambda\V_{C^\alpha} + \V \nabla\cdot Y_0^\lambda\V_{C^\alpha} \right) < + \infty.
\end{split}
\end{equation*} Moreover, assume that the initial vorticity satisfies \begin{equation*}
\begin{split}
\omega_0 \in L^1\cap L^\infty(\mathbb{R}^2), \qquad \sup_{\lambda} \V (Y_0^\lambda \cdot \nabla )\omega_0 \V_{C^{\alpha -1}} < + \infty.
\end{split}
\end{equation*} The latter condition says that $\omega_0$ is $C^\alpha$-regular in the direction of $Y_0^\lambda$. The negative index H\"older spaces may be defined in terms of the Littlewood-Paley decomposition, but it can be avoided as the above condition is equivalent to $K * ((Y_0^\lambda\cdot\nabla)\omega_0) \in C^\alpha$ (see \cite{BK1}), where $K$ is the usual Biot-Savart kernel. 

We evolve the family of vector fields by \begin{equation*}
\begin{split}
Y^\lambda_t(\Phi(t,x)) := ( Y_0^\lambda(x) \cdot \nabla )\Phi(t,x).
\end{split}
\end{equation*}

\begin{theorem*}[{\cite[Theorem 8.1]{BK1}}]
		In the above setting, the Yudovich solution $\omega_t$ and the vector fields $Y^\lambda_t$ satisfy the global-in-time bounds
		\begin{equation}\label{eq:Serfati_bound_vorticity}
		\begin{split}
		\sup_{\lambda} \V (Y_t^\lambda \cdot \nabla) \omega_t \V_{C^{\alpha -1}} \le C\exp(\exp(ct))
		\end{split}
		\end{equation} and \begin{equation}\label{eq:Serfati_bounds_vectorfield}
		\begin{split}
		\sup_{\lambda} \left(\V Y^\lambda_t \V_{C^\alpha} + \V \nabla\cdot Y^\lambda_t \V_{C^\alpha} \right)  \le C \exp(\exp(ct)).
		\end{split}
		\end{equation} The associated velocity is Lipschitz in space and indeed uniformly $C^{1,\alpha}$ after being corrected by a smooth multiple of the vorticity. That is, there is a matrix $A_t$ with $\V A_t\V_{C^\alpha} \le C\exp(\exp(ct))$ such that \begin{equation}\label{eq:Serfati_bounds_velocity}
		\begin{split}
		\V \nabla u_t \V_{L^\infty} &\le C \exp(ct), \\
		\V \nabla u_t - \omega_t A_t \V_{C^\alpha} & \le C\exp(\exp(ct)) 
		\end{split}
		\end{equation} holds. Here, the constants $C, c > 0$ depend only on $0 < \alpha <1$ and the initial data. 
\end{theorem*}

\begin{example*}
	Let us present two examples from \cite[Section 10]{BK1}. 
	
	\begin{itemize}
		\item[(i)] $C^{1,\alpha}$ Patches with $C^\alpha$ vortex profile: Take some $C^{1,\alpha}$-domain $\Omega_0$ and $C^\alpha$-function $f_0$, and then define $\omega_0 = \chi_{\Omega_0} f_0$. Then, we can take $Y_0^1 := \nabla^\perp\phi_0$ where $\phi_0$ is a $C^{1,\alpha}$ level set function for $\Omega_0$. In addition, we may take some vector field $Y_0^2$ so that $\{ Y_0^1, Y_0^2 \}$ satisfy all the requirements described in the above (most importantly, $Y_0^2$ should be non-vanishing whenever $Y_0^1$ vanishes). 
		
		We recover the usual vortex patch when the profile $f_0$ is a constant function. This results show that the vorticity can actually have a $C^\alpha$-profile on the patch. This particular statement also follows directly from the main result of Huang \cite{Hu}, which we discuss in the Appendix. 
		\item[(ii)] Vorticity smooth along leaves of a $C^{1,\alpha}$-foliation:  Consider $\phi_0 \in C^{1,\alpha}$ with $|\nabla^\perp\phi_0| \ge c > 0$ on $\mathbb{R}^2$, such that each level curve of $\phi_0$ crosses any vertical line exactly once. Under these assumptions, we define $\xi_{x_1}(x_2)$ so that $\phi_0(x_1,\xi_{x_1}(x_2)) = \phi_0(0,x_2)$. 
		
		Take some bounded measurable function $W : \mathbb{R} \rightarrow \mathbb{R}$ supported on some bounded interval $[c,d]$. Then, fix some $L > 0$ and define \begin{equation*}
		\begin{split}
		\omega_0(x_1,x_2) := \chi_{[-L,L]}(x_1) W(\xi_{x_1}(x_2)). 
		\end{split}
		\end{equation*} The above theorem applies to this case, simply with $Y_0 = \nabla^\perp\phi_0$. It follows that for all time, all the level curves of $\omega$ remain (uniformly in $\mathbb{R}^2$) $C^{1,\alpha}$. In the words of Bae and Kelliher, ``extreme lack of regularity of $\omega_0$ transversal to $Y_0$ does not disrupt the regularity of the flow lines.''
	\end{itemize}
\end{example*}

The generalization described in the above theorem can be easily adapted to our setting. Let us only described the necessary modifications in the assumptions. To begin with, we require that $\omega_0 \in L^1\cap L^\infty$ is $m$-fold symmetric for some $m \ge 3$, as usual. We need in addition that there is a distinguished vector field in the family, say $Y_0^{c}$, which is $m$-fold symmetric and satisfies \begin{equation*}
\begin{split}
\inf_{x \in B_0(r_0) } |Y_0^{c}(x) | \ge c_0 > 0, \qquad \mbox{ for some } r_0 > 0,
\end{split}
\end{equation*} and \begin{equation*}
\begin{split}
\V Y_0^c \V_{\mathring{C}^\alpha(\mathbb{R}^2)} + \V \nabla\cdot Y_0^c \V_{\mathring{C}^\alpha(\mathbb{R}^2)} + \V K * ((Y_0^c \cdot \nabla)\omega_0) \V_{\mathring{C}^\alpha(\mathbb{R}^2)}  < + \infty. 
\end{split}
\end{equation*} Then, we claim that the bounds \eqref{eq:Serfati_bound_vorticity},  and \eqref{eq:Serfati_bounds_vectorfield} hold, with $\mathring{C}^\alpha$ instead of $C^\alpha$ when $\lambda = c$. Moreover, the velocity will be Lipschitz in space for all time, and its gradient will belong to $\mathring{C}^\alpha$ after being corrected by a $\mathring{C}^\alpha$-matrix multiple of the vorticity. 

\subsubsection*{Symmetric Cusps}

Consider a $m$-fold symmetric set $\Omega_0$ which is a union of $C^{1,\alpha}$-cusps for some $m \ge 3$ in some ball $B_0(r_0)$ and has $C^{1,\alpha}$ boundary outside $B_0(r_0)$. 

It is possible to show that, using the methods of this paper (and the generalization described in the above), the boundaries of the cusp remain as $C^{1,\alpha}$ (uniformly up to the origin) curves for all time. This can be done as a two-step procedure -- the same strategy we have utilized to prove the propagation of $C^{1,\alpha}$-corners.

Note that the complementary region  $ B_0(r_0) \backslash \Omega_0$ is a disjoint union of regions, each of which can be given as the image of an exact sector under a $C^{1,\alpha}$-diffeomorphism of the plane fixing the origin. Therefore, in each of these regions, we can place a (divergence-free) vector field $Y_0^c$ just as in the case of $C^{1,\alpha}$-corners (see Figure \ref{fig:cusp}). This vector field can be extended to the interior of the cusp, so that $Y_0^c$ is non-vanishing in $B_0(r_0)$, $Y_0^c \in \mathring{C}^\alpha$, and finally $\nabla \cdot Y_0^c \in \mathring{C}^\alpha$.\footnote{To see this, consider the simple case of the $C^{1,1}$-cusp given by the region $\{ -x_1^2 \le x_2 \le x_1^2, x_1 \ge 0 \}$. Then, define $Y(x_1,x_2) = (1,2x_2/x_1)$ in the interior of the cusp. Then, $\partial_{x_2}Y = (0, 2/x_1)$ and $\partial_{x_1} Y = -2x_2/x_1^2$ so that $| |x|\nabla Y(x)| \in L^\infty$, which is equivalent to saying that $Y \in \mathring{C}^1$. Finally, $\nabla \cdot Y = 2/x_1$ and hence $\nabla (\nabla\cdot Y) = (-2/x_1^2,0)$, and since $|x_2| \le x_1^2$, $|x| |\nabla (\nabla\cdot Y)| \in L^\infty$ as well.  } After that, one takes a complementary vector field $Y_0^b$ which is $C^\alpha(\mathbb{R}^2)$, tangent to the boundary of the patch, with divergence in $C^\alpha(\mathbb{R}^2)$ and supported outside the ball $B_0(r_0/2)$. This construction of vector fields $\{ Y_0^c, Y_0^b \}$ gives global-in-time propagation of the $\mathring{C}^\alpha$-regularity of the patch. Moreover, the velocity is Lipschitz in space for all time.

After that, to recover the extra information that the boundary of the cusp stays in $C^{1,\alpha}$, one performs a local analysis which is parallel to the one given in \ref{subsec:local_estimate}. Indeed, the velocity generated by the cusps is uniformly $C^{1,\alpha}$ in the interior of the patch, up to the boundary. This finishes the argument.

\begin{figure}
	\includegraphics[scale=1.2]{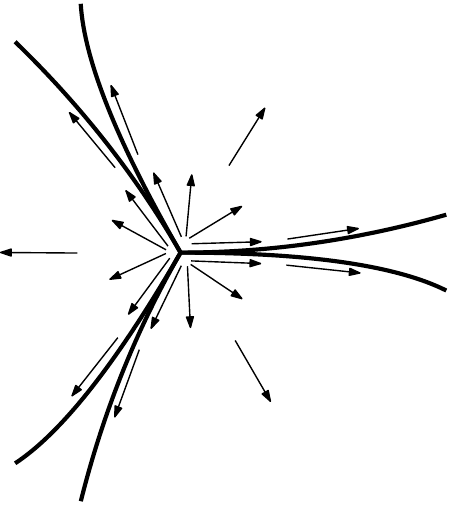} 
	\centering
	\caption{Vector field associated with a symmetric union of cusps}
	\label{fig:cusp}
\end{figure}


\section{Ill-posedness results for vortex patches with corners}\label{sec:illposed}

In this section we will give several results which show that vortex patches with corners which do not fall in the well-posedness results cannot retain a corner structure continuously in time. These are based on a general local expansion result for the velocity field associated to a bounded vorticity profile. As is well known, boundedness of the vorticity does not imply Lipschitz continuity of the velocity field; however, it turns out to be possible to give a first-order expansion of the velocity field near the origin (or any point) which isolates the non-Lipschitzian part in an explicit way. This expansion is reminiscent of the Key Lemma of Kiselev and \v{S}ver\'ak \cite{KS} but it is without any symmetry assumptions on the vorticity and it is valid for \emph{all} $x\in\mathbb{R}^2$. After giving this expansion, we use it to show ill-posedness for vortex patches with corners. In the case where the vortex patch satisfies an odd symmetry, we can actually prove immediate cusp formation. When there is just a single corner, we just show discontinuity though we believe that there is actually cusp formation and further investigation into this question is given in the next section. For corners which are only locally $m-$fold symmetric we also show ill-posedness by applying the results of the first author and Masmoudi \cite{EM1} to show that the Lipschitz bound on the velocity field of a locally $m$-fold symmetric patch may be lost immediately. This shows that the global symmetry assumptions which give the propagation of regularity proven in the previous two sections cannot be replaced by local symmetry assumptions. 

\subsection{An expansion for the velocity field associated to a bounded vorticity profile}

We now state the first and most important lemma toward the ill-posedness result which shows how one can expand the velocity field associated to a bounded vorticity profile near a point (the origin) up to "Lipschitz" error terms.  {In the following, $f \in L^\infty_c$ means that $f$ is a bounded function with compact support.  } 

\begin{lemma}\label{lem:key_radial}
	Assume that $\omega \in L^\infty_c(\mathbb{R}^2)$. Then, with polar coordinates, the corresponding velocity $u =  \nabla^\perp\Delta^{-1}\omega$ satisfies the estimate \begin{equation}\label{eq:key_radial}
	\begin{split}
	\left| u(r,\theta) - u(0) - \frac{1}{2\pi}\begin{pmatrix}
	\cos\theta \\ 
	-\sin\theta 
	\end{pmatrix} rI^s(r) + \frac{1}{2\pi}\begin{pmatrix}
	\sin\theta \\ 
	\cos\theta 
	\end{pmatrix} rI^c(r) \right| \le Cr \V \omega\V_{L^\infty}
	\end{split}
	\end{equation} with some absolute constant $C > 0$ independent on the size of the support of $\omega$. Here, \begin{equation*}
	\begin{split}
	 {u(0) = \left(
	-\frac{1}{2\pi} \int_0^\infty \int_0^{2\pi}\sin(\theta) \omega(r,\theta)d\theta dr, \,
	\frac{1}{2\pi} \int_0^\infty\int_0^{2\pi} \cos(\theta) \omega(r,\theta)d\theta dr  
	\right)^T,} 
	\end{split}
	\end{equation*} \begin{equation*}
	\begin{split}
	I^{s}(r) := \int_r^{\infty} \int_{0}^{2\pi} \sin(2\theta) \frac{\omega(s,\theta) }{s}d\theta ds, 
	\end{split}
	\end{equation*} and \begin{equation*}
	\begin{split}
	I^{c}(r) :=  \int_r^{\infty} \int_{0}^{2\pi} \cos(2\theta) \frac{\omega(s,\theta) }{s}d\theta ds. 
	\end{split}
	\end{equation*} 
\end{lemma}

\begin{remark}
	The idea of the proof is to simply decompose $\omega$ as \begin{equation*}
	\begin{split}
	\omega(r,\theta) = \sum_{ m \ge 0} \left( \sin(m\theta)f^{m,s}(r) + \cos(m\theta) f^{m,c}(r) \right),
	\end{split}
	\end{equation*} and compute, more or less explicitly, the velocity vector field corresponding to each term on the right hand side. 
\end{remark}

\begin{proof} Using polar coordinates, we write down the following decomposition of the vorticity: \begin{equation*}
	\begin{split}
	\omega = \omega^0 + \omega^1 + \omega^2 + \omega^\mathfrak{r},
	\end{split}
	\end{equation*} where \begin{equation*}
	\begin{split}
	\omega^0(r) := \frac{1}{2\pi}\int_0^{2\pi} \omega(r,\theta)d\theta 
	\end{split}
	\end{equation*} is the radial component, \begin{equation*}
	\begin{split}
	\omega^m(r,\theta) := \sin(m\theta) \frac{1}{\pi}\int_0^{2\pi} \sin(m\theta')\omega(r,\theta')d\theta' + \cos(m\theta) \frac{1}{\pi}\int_0^{2\pi} \cos(m\theta')\omega(r,\theta')d\theta' 
	\end{split}
	\end{equation*} is the $m$-fold symmetric component for $m = 1, 2$, and finally $\omega^\mathfrak{r}:= \omega - \omega^0 - \omega^1 - \omega^2$. Then, we can accordingly write \begin{equation*}
	\begin{split}
	u = u^0 + u^1 + u^2 + u^\mathfrak{r}, 
	\end{split}
	\end{equation*} where $u^m := \nabla^\perp \Delta^{-1}\omega^m$ for $m \in \{ 0, 1, 2, \mathfrak{r} \}$. We estimate each component of velocity separately. 
	
	\medskip
	
	\textbf{Radial part}
	
	\medskip
	
	We first solve $\Delta\Psi = \omega$ assuming that $\omega = \omega^0(r)$, i.e. when the vorticity is a radial function. In this case, it is well-known that the stream function is given by \begin{equation*}
	\begin{split}
	\Psi^0(r) = \int_0^r \frac{1}{s}\int_0^s \tau\omega^0(\tau)d\tau ds 
	\end{split}
	\end{equation*} and the velocity is then \begin{equation*}
	\begin{split}
	u^0(r,\theta) = \frac{1}{r}\int_0^r s\omega^0(s)ds \begin{pmatrix}
	-\sin\theta \\ \cos\theta 
	\end{pmatrix}.
	\end{split}
	\end{equation*} In particular, \begin{equation}\label{eq:radial}
	\begin{split}
	\left| \frac{u^0(r,\theta)}{r} \right| \le C \V \omega^0\V_{L^\infty} \le C \V \omega\V_{L^\infty}. 
	\end{split}
	\end{equation}

	\medskip
	
	\textbf{1-fold symmetric part}
	
	\medskip
	
	Next, we assume that $\omega = \omega^1(r,\theta) = \sin\theta f^{1,s}(r) + \cos\theta f^{1,c}(r)$. In this case, we write $\Psi^1(r,\theta) = \sin\theta \psi^{1,s}(r) + \cos\theta \psi^{1,c}(r)$, and consider the equations \begin{equation*}
	\begin{split}
	\pr_{rr} \psi^{1,i} + \frac{1}{r}\pr_r \psi^{1,i} - \frac{1}{r^2} \psi^{1,i} = f^{1,i}, \qquad i = s, c. 
	\end{split}
	\end{equation*} We then have \begin{equation*}
	\begin{split}
	\pr_r \left(\pr_r\psi^{1,i} + \frac{1}{r}\psi^{1,i} \right) = f^{1,i}
	\end{split}
	\end{equation*} and from $r^{-1}\psi^{1,i}(r), \pr_r\psi^{1,i}(r) \rightarrow 0$ as $r \rightarrow \infty$, \begin{equation*}
	\begin{split}
	\frac{1}{r}\pr_r (r\psi^{1,i}) =\pr_r\psi^{1,i} + \frac{1}{r}\psi^{1,i} = -\int_r^{\infty} f^{1,i}(s)ds 
	\end{split}
	\end{equation*} and hence \begin{equation*}
	\begin{split}
	\psi^{1,i}(r) = - \frac{1}{r}\int_0^r s\int_s^\infty f^{1,i}(\tau)d\tau ds. 
	\end{split}
	\end{equation*} From the formula \begin{equation*}
	\begin{split}
	u^1(r,\theta) = \pr_r \Psi^1(r,\theta) \begin{pmatrix}
	-\sin\theta \\ \cos\theta
	\end{pmatrix} - \frac{1}{r} \pr_{\theta} \Psi^1(r,\theta)\begin{pmatrix}
	\cos\theta \\ \sin\theta 
	\end{pmatrix},
	\end{split}
	\end{equation*} taking the first component, one obtains \begin{equation*}
	\begin{split}
	u^1_1(r,\theta) = \sin^2\theta \left( \frac{\psi^{1,s}}{r} + \int_r^\infty f^{1,s}(s)ds \right) - \cos^2\theta \frac{\psi^{1,s}(r)}{r} + \sin\theta\cos\theta \left( \frac{2\psi^{1,c}}{r} + \int_r^\infty f^{1,c}(s)ds \right).
	\end{split}
	\end{equation*} Observing the bound \begin{equation*}
	\begin{split}
	\left| \frac{2\psi^{1,i}}{r} + \int_r^\infty f^{1,i}(s)ds \right| \le Cr \V f^{1,i} \V_{L^\infty} \le Cr \V \omega \V_{L^\infty}
	\end{split}
	\end{equation*} we rewrite $u^1_1$ in the form \begin{equation*}
	\begin{split}
	u_1^1(r,\theta) &= \frac{1}{2}\int_r^\infty f^{1,s}(s)ds - \left(\frac{1}{2} - \sin^2\theta\right) \left( \frac{2\psi^{1,s}}{r} + \int_r^\infty f^{1,s}(s)ds \right) \\
	&\qquad\qquad +\sin\theta\cos\theta \left( \frac{2\psi^{1,c}}{r} + \int_r^\infty f^{1,c}(s)ds \right), 
	\end{split}
	\end{equation*} and finally arrive at the following bound: \begin{equation}\label{eq:first_1}
	\begin{split}
	\left| u_1^1(r,\theta) - \frac{1}{2}\int_0^\infty f^{1,s}(s)ds \right| \le Cr \V \omega\V_{L^\infty}.
	\end{split}
	\end{equation} On the other hand, for the second component of velocity we obtain \begin{equation}\label{eq:first_2}
	\begin{split}
	\left| u_1^2(r,\theta) + \frac{1}{2}\int_0^\infty f^{1,c}(s)ds \right| \le Cr \V \omega\V_{L^\infty}.
	\end{split}
	\end{equation}
	
	\medskip
	
	\textbf{2-fold symmetric part}
	
	\medskip
	
	We now need to solve $\Delta\Psi = \omega$ in the case when $\omega(r,\theta) = \sin(2\theta)f^{2,s}(r)$ and $\cos(2\theta)f^{2,c}(r)$, respectively. Setting $\Psi(r,\theta) = \sin(2\theta)\psi^{2,s}(r,\theta)$ and $\cos(2\theta)\psi^{2,c}(r,\theta)$ respectively gives the relations \begin{equation}\label{eq:psi_f}
	\begin{split}
	\pr_{rr}\psi^{2,i} + \frac{1}{r}\pr_r \psi^{2,i} - \frac{4}{r^2}\psi^{2,i} = f^{2,i}, \qquad i = s, c. 
	\end{split}
	\end{equation} Then, one may rewrite it as \begin{equation*}
	\begin{split}
	\pr_r \left( \frac{1}{r}\pr_r \psi^{2,i} + \frac{2}{r^2} \psi^{2,i} \right) = \frac{f^{2,i}}{r},
	\end{split}
	\end{equation*} and since we are looking for a solution with bounds $|\psi^{2,i}(r)| \le C\ln(1 + r)$ and $|\pr_r \psi^{2,i}(r)| \le Cr^{-1}$, we obtain \begin{equation*}
	\begin{split}
	\frac{1}{r}\pr_r \psi^{2,i} + \frac{2}{r^2} \psi^{2,i}  = - \int_r^\infty \frac{f^{2,i}(s)}{s}ds , 
	\end{split}
	\end{equation*} and integrating once more, \begin{equation*}
	\begin{split}
	\psi^{2,i}(r) = -\frac{1}{r^2} \int_0^r s^3 \int_s^\infty \frac{f^{2,i}(\tau)}{\tau} d\tau ds 
	\end{split}
	\end{equation*}
	Now, we compute $u^{2,s} = \nabla^\perp( \sin(2\theta)\psi^{2,s} )$ as well as $u^{2,c} = \nabla^\perp( \cos(2\theta) \psi^{2,c})$. In the case of $u^{2,s}$, we have \begin{equation*}
	\begin{split}
	u^{2,s}(r,\theta) = \pr_r (\sin(2\theta)\psi^{2,s}) \begin{pmatrix}
	-\sin\theta \\ \cos\theta
	\end{pmatrix} - \frac{1}{r}\pr_{\theta} (\sin(2\theta)\psi^{2,s}) \begin{pmatrix}
	\cos\theta \\ \sin\theta 
	\end{pmatrix}.
	\end{split}
	\end{equation*} Taking the first component, after a bit of rearranging we get: \begin{equation*}
	\begin{split}
	u^{2,s}_1(r,\theta) &= -\sin(2\theta)\sin\theta \pr_r \psi^{2,s} - 2\cos(2\theta)\cos\theta \frac{\psi^{2,s}}{r} \\
	& = \sin(2\theta)\sin\theta \left( \frac{4}{r} \psi^{2,s} + r \int_r^\infty \frac{f^{2,s}}{s}ds  \right) - \left( \frac{r}{2} \int_r^\infty \frac{f^{2,s}}{s}ds - \frac{r}{2} \int_r^\infty \frac{f^{2,s}}{s}ds + \frac{2}{r} \psi^{2,s} \right)\cos\theta 
	\end{split}
	\end{equation*} We note that \begin{equation*}
	\begin{split}
	\left| \frac{4}{r} \psi^{2,s} + r \int_r^\infty \frac{f^{2,s}}{s}ds  \right| \le Cr\V f^{2,s}\V_{L^\infty}
	\end{split}
	\end{equation*} uniformly in $r \ge 0$, with some absolute constant $C > 0$. Hence, we obtain \begin{equation*}
	\begin{split}
	u^{2,s}_1(r,\theta) = \frac{r\cos\theta}{2} \int_r^\infty \frac{f^{2,s}}{s}ds + \frac{2\sin(2\theta)\sin\theta - \cos\theta}{2} \left(\frac{4}{r} \psi^{2,s} + r \int_r^\infty \frac{f^{2,s}}{s}ds \right)
	\end{split}
	\end{equation*} and \begin{equation}\label{eq:second_sin_1}
	\begin{split}
	\left| \frac{u_1^{2,s}(r,\theta)}{r\cos\theta} - \frac{1}{2}\int_r^\infty \frac{f^{2,s}}{s} ds  \right| \le C \V f^{2,s}\V_{L^\infty} \le C \V \omega\V_{L^\infty}. 
	\end{split}
	\end{equation} Similarly, for the second component of velocity, we obtain \begin{equation}\label{eq:second_sin_2}
	\begin{split}
	\left| \frac{u_2^{2,s}(r,\theta)}{r\sin\theta} + \frac{1}{2}\int_r^\infty \frac{f^{2,s}}{s} ds  \right| \le C \V f^{2,s}\V_{L^\infty} \le C \V \omega\V_{L^\infty}. 
	\end{split}
	\end{equation} Next, in the case of $u^{2,c}$, \begin{equation*}
	\begin{split}
	u^{2,c}(r,\theta) = \pr_r (\cos(2\theta)\psi^{2,c}) \begin{pmatrix}
	-\sin\theta \\ \cos\theta
	\end{pmatrix} - \frac{1}{r}\pr_{\theta} (\cos(2\theta)\psi^{2,c}) \begin{pmatrix}
	\cos\theta \\ \sin\theta 
	\end{pmatrix}.
	\end{split}
	\end{equation*} Similarly as in the case of $u^{2,s}$, we rearrange it to obtain \begin{equation*}
	\begin{split}
	u^{2,c}(r,\theta) &= -\left(\frac{4}{r} \psi^{2,c} + r \int_r^\infty \frac{f^{2,c}}{s}ds\right) \begin{pmatrix}
	\cos(2\theta)\sin\theta \\ -\cos(2\theta)\cos\theta 
	\end{pmatrix} \\
	&\qquad + \left(\frac{r}{2} \int_r^\infty \frac{f^{2,c}}{s}ds -\frac{2}{r} \psi^{2,c} - \frac{r}{2} \int_r^\infty \frac{f^{2,c}}{s}ds \right)\begin{pmatrix}
	\sin\theta \\ \cos\theta 
	\end{pmatrix},
	\end{split}
	\end{equation*} and in particular, \begin{equation}\label{eq:second_cos}
	\begin{split}
	\left| \frac{u_1^{2,c}(r,\theta)}{r\sin\theta} + \frac{1}{2} \int_r^\infty \frac{f^{2,c}}{s}ds  \right| , \qquad \left| \frac{u_2^{2,c}(r,\theta)}{r\cos\theta} + \frac{1}{2} \int_r^\infty \frac{f^{2,c}}{s}ds  \right| \le C \V \omega\V_{L^\infty}. 
	\end{split}
	\end{equation}

	\medskip
	
	\textbf{Remainder}
	
	\medskip
	
	We shall assume that $\omega(r,\theta) = \omega^\mathfrak{r}(r,\theta)$, the point being that \begin{equation*}
	\begin{split}
	\int_0^{2\pi} \omega^\mathfrak{r}(r,\theta) \cos(m\theta) d\theta = 0 = \int_0^{2\pi} \omega^\mathfrak{r}(r,\theta) \sin(m\theta) d\theta
	\end{split}
	\end{equation*} for all $r  \ge 0$ and $m = 0, 1, 2$. In this case, we simply use the Biot-Savart kernel: \begin{equation*}
	\begin{split}
	u^\mathfrak{r}(x) = \frac{1}{2\pi}\int_{ \mathbb{R}^2} \frac{(x-y)^\perp}{|x-y|^2} \omega^\mathfrak{r}(y) dy . 
	\end{split}
	\end{equation*} Taking the first component, we write \begin{equation*}
	\begin{split}
	u^\mathfrak{r}_1(x) = \frac{1}{2\pi} \int_{ |y| \le 2|x|} \frac{-(x_2-y_2)}{|x-y|^2} \omega^\mathfrak{r}(y) dy +  \frac{1}{2\pi} \int_{ |y| > 2|x|} \frac{-(x_2-y_2)}{|x-y|^2} \omega^\mathfrak{r}(y) dy ,
	\end{split}
	\end{equation*} and the first piece is bounded by \begin{equation*}
	\begin{split}
	\left|\frac{1}{2\pi} \int_{ |y| \le 2|x|} \frac{-(x_2-y_2)}{|x-y|^2} \omega^\mathfrak{r}(y) dy \right| \le C\V \omega^\mathfrak{r}\V_{L^\infty} \int_{ |y| \le 2|x|} \frac{1}{|x-y|} dy \le C|x| \V \omega\V_{L^\infty}. 
	\end{split}
	\end{equation*} Regarding the second piece, we first use that \begin{equation*}
	\begin{split}
	\int_{|y| > 2|x|} \frac{y_2}{|y|^2}\omega^\mathfrak{r}(y) dy = 0
	\end{split}
	\end{equation*} to rewrite it as \begin{equation*}
	\begin{split}
	-\frac{1}{2\pi} \int_{ |y| > 2|x|} \left[ \frac{x_2-y_2}{|x-y|^2} + \frac{y_2}{|y|^2}  \right]\omega^\mathfrak{r}(y) dy,
	\end{split}
	\end{equation*} and note that \begin{equation*}
	\begin{split}
	\frac{x_2-y_2}{|x-y|^2} + \frac{y_2}{|y|^2}  = \frac{- x_1(2y_1y_2) + x_2(y_1^2 - y_2^2) + y_2|x|^2}{|x-y|^2|y|^2}. 
	\end{split}
	\end{equation*} Then, using that \begin{equation*}
	\begin{split}
	\int_{ |y| > 2|x|} \frac{2y_1y_2}{|y|^4} \omega^\mathfrak{r}(y) dy = 0  = \int_{ |y| > 2|x|} \frac{y_1^2-y_2^2}{|y|^4} \omega^\mathfrak{r}(y) dy ,
	\end{split}
	\end{equation*} we have that \begin{equation*}
	\begin{split}
	\frac{1}{2\pi} \int_{ |y| > 2|x|} \frac{-(x_2-y_2)}{|x-y|^2} \omega^\mathfrak{r}(y) dy  =- \frac{1}{2\pi} \int_{ |y| > 2|x|} \left[ \frac{x_2-y_2}{|x-y|^2} + \frac{y_2}{|y|^2}   +x_1 \frac{2y_1y_2}{|y|^4} - x_2 \frac{y_1^2 - y_2^2}{|y|^4}   \right]\omega^\mathfrak{r}(y) dy ,
	\end{split}
	\end{equation*} and the expression in the large brackets equals \begin{equation*}
	\begin{split}
	\frac{y_2|x|^2|y|^2 + ( x_2(y_1^2 - y_2^2) -2x_1y_1y_2)(2x\cdot y - |x|^2) }{|x-y|^2|y|^4} 
	\end{split}
	\end{equation*} which is bounded in absolute value by $C|x|^2|y|^{-3}$ for some uniform constant $C > 0$ in the region $|y| > 2|x|$. Therefore, \begin{equation*}
	\begin{split}
	\left| \frac{1}{2\pi} \int_{ |y| > 2|x|} \frac{-(x_2-y_2)}{|x-y|^2} \omega^\mathfrak{r}(y) dy \right| \le C |x|^2\V \omega\V_{L^\infty} \int_{ |y| > 2|x|} \frac{1}{|y|^3} dy \le C|x|\V\omega\V_{L^\infty}. 
	\end{split}
	\end{equation*} We have shown the desired bound \begin{equation}\label{eq:remainder_1}
	\begin{split}
	|u_1^\mathfrak{r}(x)| \le C|x| \V \omega\V_{L^\infty}. 
	\end{split}
	\end{equation} A completely parallel argument establishes that \begin{equation}\label{eq:remainder_2}
	\begin{split}
	|u_2^\mathfrak{r}(x)| \le C|x| \V \omega\V_{L^\infty}. 
	\end{split}
	\end{equation}
	
	\medskip
	
	Combining the estimates \eqref{eq:radial}, \eqref{eq:first_1}, \eqref{eq:first_2}, \eqref{eq:second_sin_1}, \eqref{eq:second_sin_2}, \eqref{eq:second_cos}, \eqref{eq:remainder_1}, and \eqref{eq:remainder_2} finishes the proof. \end{proof}

\subsection{Loss of boundary regularity for odd-odd patches}


In this subsection, we demonstrate that under the odd-odd symmetry, vortex patches with a corner may  {continuously lose regularity of the boundary with }time. By the odd-odd symmetry, we mean that the vorticity satisfies $\omega(x_1,x_2) = -\omega(-x_1,x_2) = -\omega(x_1,-x_2)$ on $\mathbb{R}^2$ (or on $\mathbb{T}^2 = [-1,1)^2$). Equivalently, one may consider vorticities which is odd in $x_1$ on the upper half-plane $\mathbb{R} \times \mathbb{R}_+$, with the slip boundary condition. 

In the result below, we consider an odd-odd patch with four corners meeting at the origin and also tangent to the $x_1$-axis (see Figure \ref{fig:oddodd}). It shows that the ``angle'' of each corner at the origin immediately becomes $\pi/2$ for $t > 0$, and 0 for $t < 0$. 

\begin{theorem}\label{thm:odd_odd_Hoelder}
	Consider an odd-odd vortex patch supported on $\Omega_0 \subset \mathbb{R}^2$ such that $$\overline{\Omega}_0 \cap \{ (x_1,x_2) : 0 \le x_1, x_2 \le 1/2  \} = \{ (x_1,x_2): 0 \le x_1 \le x_2 \le 1/2 \}~.$$ Consider the trajectories of points which initially lie on the diagonal $\Phi(t,(x,x)) =: z(t,x)$. Then, there exist constants $T^*,\delta > 0$, such that  in the ball $[0,\delta]^2$, we have bounds \begin{equation}\label{eq:bootstrap_first}
	\begin{split}
	2z_1(t,x)^{\beta(t)} \ge z_2(t,x) \ge \frac{1}{2} z_1(t,x)^{\alpha(t)},
	\end{split}
	\end{equation} for some strictly decreasing, positive, and continuous functions defined on $[0,T^*]$ with $\alpha(0) = \beta(0) = 1$. 
	In particular, the angle of the patch at the origin becomes immediately $\pi/2$. On the other hand, if one considers the backwards in time evolution, there exists some time interval $[-T',0)$ with $T'>0$, during which the angle of the patch at the origin is zero.	
\end{theorem}

One may formally define the (cosine of the) angle as follows: given a domain $U \subset [0,1]^2$ (which is assumed to intersect any small square $[0,\delta]^2$), \begin{equation}\label{def:angle}
\begin{split}
\cos\theta_U := \lim_{\delta \rightarrow 0^+} \sup_{x,x' \in U\cap [0,\delta]^2 } \frac{x\cdot x'}{|x||x'|}~.
\end{split}
\end{equation}

\begin{figure}
	\includegraphics[scale=0.8]{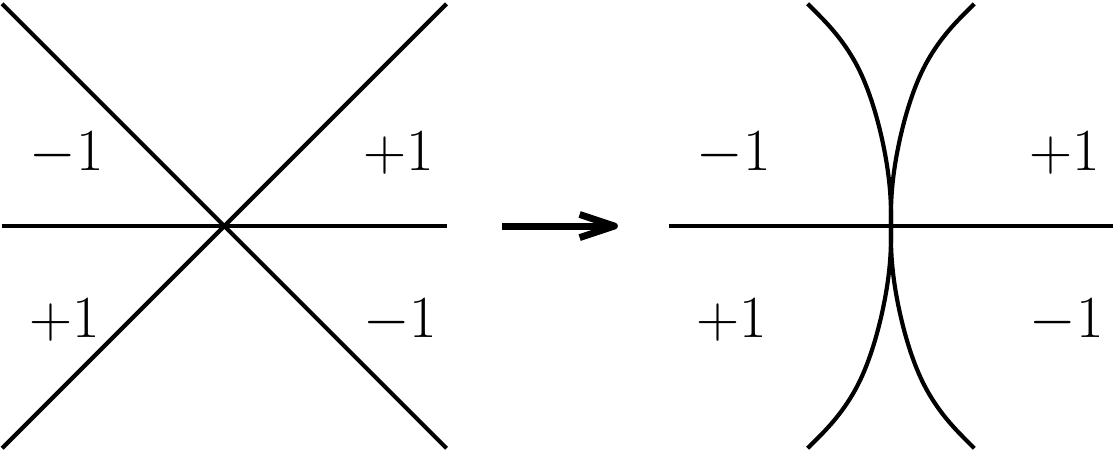} 
	\centering
	\caption{Evolution of a corner with odd-odd symmetry.}
	\label{fig:oddodd}
\end{figure}

We may consider these data on the upper half-plane, and in this case, the boundary of the initial patch $\partial\Omega_0$ is given as the graph of a $C^{0,1}$ and $C^{1,0}$-function, respectively, near the origin. This property is not maintained for any small time $t > 0$. On the other hand, if initially one considers odd-odd patch given by the region below the graph of a $C^{1,\alpha}$-function whose derivative vanish at the origin (see Figure \ref{fig:oddodd_cusp}), it can be shown using the ideas of previous sections that the solution continues to satisfy these properties. See also the recent work of Kiselev, Ryzhik, Yao, and Zlato{\v{s}} \cite{KRYZ} where they show (among other things) well-posedness for cusps touching the boundary. In this sense, these two results show that the vortex patch problem is ill-posed, when its boundary on the upper half-plane is only $C^{0,1}$ or $C^1$.

\begin{remark}
	The second part of the theorem was established in an earlier work by Hoff and Perepelitsa \cite{HP}, with a similar patch initial data but having just one odd symmetry with respect to the $x_1$-axis. It is expected that the dynamics in that case is equivalent to the odd-odd symmetry case, up to a translation of the corner point. Here we offer a simplified proof using the Key Lemma. 
\end{remark}



\begin{figure}
	\includegraphics[scale=0.8]{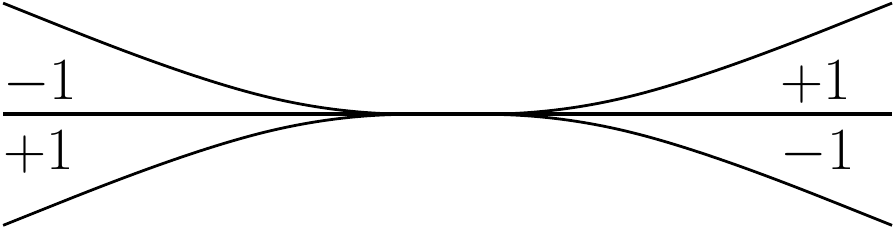} 
	\centering
	\caption{A $C^{1,\alpha}$-cusp with odd-odd symmetry.}
	\label{fig:oddodd_cusp}
\end{figure}

We now recall the key lemma of Kiselev-\v{S}ver\'{a}k \cite{KS} and Zlato\v{s} \cite{Z}: \begin{lemma}\label{lem:key}
	Let $\omega(t,\cdot) \in L^\infty( \mathbb{R}^2 )$ be odd-odd. For $x_1,x_2 \in (0,1/2]$, we have \begin{equation}\label{eq:key}
	\begin{split}
	(-1)^j\frac{ u_j(t,x)}{x_j} = \frac{4}{\pi} \int_{Q(2x)} \frac{y_1y_2}{|y|^4} \omega(t,y) dy + B_j(t,x) 
	\end{split}
	\end{equation} where $Q(2x):=[2x_1,1]\times[2x_2,1]$ and \begin{equation*}
	\begin{split}
	\left| B_j(t,x)  \right| \le C \V \omega(t,\cdot)\V_{L^\infty} \left( 1 + \ln\left( 1 + \frac{x_{3-j}}{x_j} \right)   \right)
	\end{split}
	\end{equation*} for $j \in \{ 1, 2\}$. 
\end{lemma}
\begin{remark}
We note that this actually follows from the more general expansion given in Lemma \ref{lem:key_radial}.
\end{remark}
For simplicity of notation, we will denote the integral in \eqref{eq:key} as \begin{equation*}
\begin{split}
I(t,x):= \frac{4}{\pi} \int_{Q(2x)} \frac{y_1y_2}{|y|^4} \omega(t,y) dy~.
\end{split}
\end{equation*}
On the other hand, we have the well-known log-Lipschitz bound for velocity: for any $x,x'$ with $|x-x'| < 1/2$, \begin{equation*}
\begin{split}
|u(t,x)- u(t,x')| \le C\V \omega(t,\cdot)\V_{L^\infty} |x-x'| \ln \frac{1}{|x-x'|}~.
\end{split}
\end{equation*}

\begin{proof}[Proof of Theorem \ref{thm:odd_odd_Hoelder}]
	
	
	
	
	We consider the case $t \ge 0$. We shall work within a short time interval $[0,T^*]$ for some $T^* > 0$, and in several places, the value of $T^*$ will be taken to be sufficiently small for the arguments to work. 
	
	We begin with a simple observation. Note that on the diagonal $x = (x',x')$ with $0 < x'$, we have  \begin{equation*}
	\begin{split}
	\frac{u_1(t,x)}{u_2(t,x)} = - \frac{I(t,x) + B_1(t,x) }{I(t,x) + B_2(t,x)}
	\end{split}
	\end{equation*} and $|B_j(t,x)| \le C$ for all time. Clearly, one can find some small $\delta_1 > 0$ and $T^* > 0$ such that $I(t,(\delta_1,\delta_1)) \ge 10C $  for all $ 0 \le t \le T^*$, simply because $I(t,(\delta_1,\delta_1))$ is continuous in $t, \delta_1$ and $I(0,(\delta_1,\delta_1)) \rightarrow +\infty$ as $\delta_1 \rightarrow 0^+$. Therefore, we take $\delta_1 \ge \delta_2 > 0$ such that the triangle $\{ 0 \le x_1 \le x_2 \le \delta_2 \}$ is contained in $\Phi(t,\Omega_0)$ for all $0<t<T^*$. 
	
	Consider a point on the diagonal $(x,x)$ with $0<x<\delta \ll \delta_2$ (the value of $\delta>0$ will be specified later) and denote its trajectory by $z(t,x) = (z_1(t,x),z_2(t,x)):=\Phi(t,(x,x))$. From the basic log-Lipschitz estimate on $u_2$,  \begin{equation*}
	\begin{split}
	\left| \frac{d}{dt} z_2(t,x) \right| \le C z_2(t,x)  \ln \frac{c}{z_2(t,x)}~, 
	\end{split}
	\end{equation*} (since $u_2(t,(z_1(t,x),0)) = 0$ by odd symmetry) and upon integration, we deduce that $z_2(t,x) \le c x^{\exp(-Ct)}$. Proceeding analogously for $z_1(t)$, we obtain $z_1(t,x) \ge c  x^{\exp(Ct)}$ this time. Inserting these crude bounds, \begin{equation}\label{eq:boundforerror}
	\begin{split}
	|B_1(t,z(t,x))| \le C\left( 1 + \exp(2CT^*) \ln \frac{1}{x} \right) 
	\end{split}
	\end{equation} for $0 \le t \le T^*$. Moreover, with $\beta(t):= \exp(-2Ct)$, we obtain \begin{equation*}
	\begin{split}
	z_1(t,x)^{\beta(t)} \ge \frac{1}{2} {z_2(t,x)},
	\end{split}
	\end{equation*}  for $ 0 \le t \le T^*$ by choosing $T^*$ smaller if necessary. 
	
	A lower bound for the integral $I(t,z(t,x))$ comes from the fact that $\Phi(t,\Omega_0)$ contains a triangle. We could have chosen $\delta > 0$ small so that for all $x < \delta$, its trajectory satisfies the bound $ z_2(t) \le \delta_2$. In particular, the region $Q(z(t,x))$ contains the triangle with vertices $(z_2(t,x),z_2(t,x)), (\delta_2,z_2(t,x))$, and $(\delta_2,\delta_2)$. Hence \begin{equation*}
	\begin{split}
	I(t,z(t,x)) \ge c \ln \frac{\delta_2}{z_2(t,x)} \ge c\exp(-Ct) \ln \frac{\delta_2'}{x}
	\end{split}
	\end{equation*} and comparing this with \eqref{eq:boundforerror}, we could have chosen $\delta, T^* > 0$ smaller so that for $0<t < T^*$ and $0<x<\delta$, \begin{equation*}
	\begin{split}
	I(t,z(t,x)) \ge \frac{1}{10} |B_1(t,z(t,x))|~. 
	\end{split}
	\end{equation*} Therefore, we may neglect the $B_1$-term in \eqref{eq:key} at the cost of changing the multiplicative constant, and deduce \begin{equation*}
	\begin{split}
	-\frac{u_1(t,z(t,x))}{z_1(t,x)} \ge c \ln \frac{\delta_2'}{x}~.
	\end{split}
	\end{equation*} In turn, this ensures that \begin{equation*}
	\begin{split}
	z_1(t,x) \le c' x^{1+ct}
	\end{split}
	\end{equation*} with $c' \rightarrow 1$ as $T^* \rightarrow 0^+$. From the trivial bound $z_2(t,x) \ge x$, we obtain  \begin{equation*}
	\begin{split}
	{z_2(t,x)} \ge \frac{1}{2} {z_1(t,x)^{\alpha(t)}},
	\end{split}
	\end{equation*} with $\alpha(t) := (1+ct)^{-1}$. This finishes the proof of the first part. 
	
	We now consider the backwards in time dynamics. Instead of reversing time, we revert the sign of vorticity, so that now initially the direction of velocity is southeast on the diagonal segment. As in the above, we set $z(t,x) := \Phi(t,(x,x))$, and restrict our attention to $0<x<\delta$ and $ 0 \le t \le T'$, for small $\delta, T' > 0$ to be chosen below. Similarly as before, using either the Key Lemma or the log-Lipschitz estimate on velocity gives \begin{equation*}
	\begin{split}
	z_2(t,x) \le z_1(t,x) \le 2 z_2(t,x)^{\gamma(t)} \le 2x^{\gamma(t)}
	\end{split}
	\end{equation*} for all $0 < x < \delta$ and $0 \le t \le T'$ with some sufficiently small $\delta,T' > 0$. Here $\gamma(t)> 0$ is some continuous monotonically decreasing function with $\gamma(0) = 1$ and $\gamma(t) < 1 - ct$. This gives a lower bound on the integral \begin{equation*}
	\begin{split}
	I(t,z(t,x)) \ge I(t,(z_1(t,x),z_1(t,x)) \ge c \frac{\delta^{\eta(t)} - z_1^{\eta(t)} }{\eta(t)} \ge c' \frac{\tilde{\delta}^{ct} - x^{ct} }{ct}
	\end{split}
	\end{equation*} where $\eta(t) = 2(1/\gamma(t) -1) \gtrsim t $ satisfies $\eta(0) = 0$ and is monotonically increasing with $t$. 
	Applying the Key Lemma to each of $z_1(t,x)$ and $z_2(t,x)$, we see that \begin{equation*}
	\begin{split}
	\frac{d}{dt} \left(  \frac{z_1(t,x)}{z_2(t,x)}  \right) \ge \frac{z_1(t,x)}{z_2(t,x)} \left( 2I(t,z(t,x)) - C \ln\left( 1 + \frac{z_1(t,x)}{z_2(t,x)} \right) \right),
	\end{split}
	\end{equation*} or equivalently, \begin{equation*}
	\begin{split}
	\frac{d}{dt} \ln\left( 1 + \frac{z_1(t,x)}{z_2(t,x)} \right)\ge c' \frac{\tilde{\delta}^{ct} - x^{ct} }{ct} - C \ln\left( 1 + \frac{z_1(t,x)}{z_2(t,x)} \right).
	\end{split}
	\end{equation*} Fix some small $x > 0$ and consider the ODE \begin{equation*}
	\begin{split}
	\frac{d}{dt} f^{(x)}(t) = \frac{\tilde{\delta}^{ct} - x^{ct} }{ct} - Cf^{(x)}(t), \qquad f(0) = \ln 2. 
	\end{split}
	\end{equation*} It is straightforward to show that, for all sufficiently small $0 < t \le T'$, we have \begin{equation*}
	\begin{split}
	\lim_{x \rightarrow 0^+} \frac{f^{(x)}(t)}{x} = +\infty. 
	\end{split}
	\end{equation*} Then, this implies that \begin{equation*}
	\begin{split}
	\lim_{x \rightarrow 0^+} \frac{z_1(t,x)}{z_2(t,x)} = +\infty
	\end{split}
	\end{equation*} for all $0 < t \le T'$, which shows that the angle of the patch is zero in the same time interval.
\end{proof}

\begin{remark} In the result above, the initial corner angle of the patch can be an arbitrary number strictly between 0 and $\pi/2$. Moreover, a straightforward modification of the proof shows that there is an initial patch with corner angle zero whose angle immediately becomes $\pi/2$ for $t > 0$. This can be done for instance using a patch of the form \begin{equation*}
	\begin{split}
	\mathrm{cl}(\Omega_0) \cap [0,1/2]^2 = \left\{ (x_1,x_2) : 0 \le x_2 \le x_1 \left(\ln \ln\frac{1}{x_1}\right)^{-1}  \right\} \cap [0,1/2]^2~.
	\end{split}
	\end{equation*}
\end{remark}

\begin{remark}
	We note that very recently, the key lemma was utilized to obtain double exponential rate of growth in time for the curvature of smooth vortex patches touching the horizontal axis -- see \cite{KiLi}. 
\end{remark}

\subsection{Non-continuity of the angle in a vortex patch with a corner}

In this subsection, we prove two non-continuity results. The first is descriptive in that it gives a lower bound on the angular movement of particles close to the corner. The second one asserts that if at any time a good portion of the mass of a vortex patch (or sequence of vortex patches) asymptotically is in a sector, then the location of that sector  {cannot vary continuously in time.} We now state and prove our first theorem in this direction. The reader should take note that the first theorem is easier to prove and contains two of the ideas contained in the proof of the second theorem.  
\begin{theorem}
Assume that $\Omega_0\subset B_{1/2}(0)$ is a 2-fold symmetric open set and  $\Omega_0\cap B_{1/4}(0)=S_{-\theta_1,\theta_1}\cup S_{\pi-\theta_1,\pi+\theta_1}\cap B_{1/4}(0)$ with $\theta_1 \in (0,\pi/2)\cup(\pi/2,\pi)$. Then, there exists a fixed constant $c>0$, a sequence of radii $\epsilon_n\rightarrow 0$, and a sequence of times $t_n\rightarrow 0$ so that for all $n\in\mathbb{N}$ we have \[ \theta_1+c\leq \arctan\Big(\frac{\Phi_2(t_n,x)}{\Phi_1(t_n,x)}\Big)\leq \theta_1+10c \] for all $x$ with $\epsilon_n\leq |x|\leq 2\epsilon_n$  {in $S_{-\theta_1,\theta_1}$}.
\end{theorem}

\begin{remark}
The theorem indicates that the patch "jumps" up at least by the angle $c$ once $t>0$. 
\end{remark}

\begin{proof}

First, it is not difficult to show that since $\Omega_0\subset B_\frac{1}{2}(0)$, we have that for all $t\geq 0$ \[|u(x)-u(y)|\leq 100|x-y||\log|x-y||\] for all $|x-y|<\frac{1}{2}$. To see this, we need only observe that $\V\omega\V_{L^\infty}\leq 1$ and $\V\omega\V_{L^1}\leq 1$ and run a standard potential theory argument. Moreover, it suffices to consider the case $\theta_1 \in (0,\pi/2)$, since otherwise one can argue with the complement of $\Omega_0$ instead, which has acute angles.  

Next, we define $\Phi(t,x_0)$ to be the position of a particle initially at $x_0$ at time $t$. Using the log-Lipschitz bound on the velocity field above and the (generalized) Gronwall lemma we have the following bounds on a small interval of time: \[|x|^{1+2Ct}\leq |x|^{e^{Ct}}\leq |\Phi(t,x)|\leq |x|^{e^{-Ct}}\leq |x|^{1-2Ct}.\] Now we take $\gamma_1=\min\{\frac{1}{40C},\frac{1}{40}\}$ so that if $|t|\leq t_0:=\frac{\gamma}{|\log|x_0||}$ we get that \[\frac{1}{2}\leq \frac{|\Phi(t,x)|}{|x|}\leq 2\] if $|x|\geq \frac{1}{10}|x_0|$ where we used that $|x_0|^{\frac{1}{2|\log|x_0||}}=\sqrt{e}$. Next, we claim that if $x\in \partial\Omega_0\cap B_{\frac{1}{4}}(0)\cap\{x_1,x_2\geq 0\}$ is so that $|x|\geq \frac{1}{10}|x_0|$ and if $t<\frac{\gamma}{|\log|x_0||}$ with $\gamma\leq \gamma_1$ small enough we have: \[\theta_1-\epsilon\leq \arctan\Big(\frac{\Phi_2(t_0,x)}{\Phi_1(t_0,x)}\Big)\leq \theta_1+\epsilon<\frac{\pi}{2}\] for given $\epsilon>0$ small. Indeed, for all such $x$ and $t$ we have that \[|u(t,x)|\leq 2C|x| |\log|x_0||.\] Consequently, \[|\Phi(t,x_0)-x_0|\leq C\gamma |x|\] from which the claim follows. The reader should notice that $\gamma$ depends linearly on $\epsilon$. In particular, if $|x|\geq |x_0|$ and $t<\frac{\gamma}{|\log|x_0||},$ the "bulk" of the vortex patch does not move too much.
 
Let $A_0=\Omega_0\cap \{\frac{1}{4}|x_0|\leq x\leq 4|x_0|\}.$ From Lemma \ref{lem:key_radial} we now see that we have the following bound for all $t\in [0,\frac{\gamma}{|\log|x_0||}]:$ \[|u(t,x)-u(0,x)|\leq 10\epsilon |x||\log|x||+C|x|\leq 20\epsilon|x||\log|x||\] if $|x|$ is sufficiently small. However, invoking Lemma \ref{lem:key_radial} again we see that $u_0(x)=\sin(2\theta_1)\log\frac{1}{|x|}(x_2,x_1)+C|x|$ for $C$ a fixed constant. Now we notice, putting together the preceding considerations, that \[\frac{d}{dt} \Big(\frac{\Phi_2(t,x)}{\Phi_1(t,x)}\Big)\geq \frac{\sin(2\theta_1)}{2}|\log|x_0||\Big(1-\frac{\Phi_2(t,x)^2}{\Phi_1(t,x)^2}-C\epsilon\Big).\] It is now easy to see that $\frac{\Phi_2(t,x)}{\Phi_1(t,x)}$ grows by a fixed constant depending only on $\gamma$ over the time interval $[0,\frac{\gamma}{|\log|x_0||}]$ by rescaling time by $|\log|x_0||$. This concludes the proof. 
\end{proof}

It is possible that this proof can be strengthened to actually give cusp formation for such an initial patch. Essentially, for points of size $\epsilon$, we can only track the evolution of the point for time $c|\log(\epsilon)|^{-1}$ where $c$ is a small universal constant. (Hence the variable $t\ln\frac{1}{r}$ is natural, which is used explicitly in Section \ref{effective}.) Any improvement in this time-scale would be a step forward. Our next result is slightly more general in that we make no assumption on the initial configuration of the patch except to say that the patch is asymptotically close to a sector as $r\rightarrow 0$. 

\begin{theorem}\label{thm:illposed-acute}
There exists an absolute constant $M>0$ such that there are no angles $\theta_1(t)$ and $\theta_2(t)$ which depend continuously on any time interval $[0,\delta]$ with the property that \begin{equation}\label{eq:angle_continuity}
\begin{split}
\limsup_{r\rightarrow 0}\sup_{t\in [0,\delta]}\frac{|(\Omega(t)\Delta S_{\theta_1(t),\theta_2(t)})\cap B_r(0)|}{|B_r(0)|}<\frac{\theta_2(0)-\theta_1(0)}{M},
\end{split}
\end{equation} if $0<\theta_2(0)-\theta_1(0)<\frac{\pi}{2}.$
\end{theorem}

\begin{remark}
This theorem says that the vortex patch $\Omega(t)$ near $x=0$ cannot asymptotically be approximated by a sector even with a small (but non-zero) error depending on the initial size of the angle. Of course, this implies that the corner could never remain a regular corner continuously in time since that would imply that the limit in the statement of the theorem vanishes. Note also that this theorem also applies to the case of several (or infinitely many) vortex patches. One way to interpret the result is to say that acute or obtuse corners cannot be formed dynamically in time. 
\end{remark}

\begin{remark}
We give the statement and proof for $2-$fold symmetric patches but the proof remains the same for single corners. The only difference is that we have to factor out translation with respect to the velocity at the corner. Otherwise, the expansion of the velocity field and all other arguments are identical. 
\end{remark}

\begin{proof}
Toward a contradiction, assume such $\theta_1(t), \theta_2(t)$ and $\delta>0$ exist. By rotation invariance and continuity, we may assume that $\theta_1(0)=-\theta_2(0)$ while $-\pi/4< \theta_1(t)<0< \theta_2(t)<\pi/4$ for all $t\in [0,\delta]$.  
Now let us expand the velocity field $u$.
It is easy to see that (see explicit computations in Subsection \ref{subsec:computations}) \[u_0(x)= \frac{1}{2\pi}\sin(2\theta_2(0))\log \frac{1}{|x|} \begin{pmatrix}
x_2 \\ x_1 
\end{pmatrix} +O(|x|)\] as $|x|\rightarrow 0$.  Now define $\alpha:= \frac{1}{2\pi}\sin(2\theta_2(0))>0.$
Then we have \[u(t, x)= \alpha\log\frac{1}{|x|}\begin{pmatrix}
x_2 \\ x_1 
\end{pmatrix} +O((\epsilon_1+\frac{\alpha}{M})|x||\log \frac{1}{|x|})\] as $t,|x|\rightarrow 0$ using \ref{eq:angle_continuity} and continuity of the angles. In particular, if $|x|$ and $t$ are small enough (depending only on $\alpha$), we can essentially neglect the second term in regions where the sizes of $x_1$ and $x_2$ are comparable. To see this, by continuity of $\theta_i$ there exists $t_1<\delta$ so that if $t\in [0,t_1]$, \[|\sin(2\theta_2(t))-\sin(2\theta_1(t))-2\sin(2\theta_2(0))|+|\cos(2\theta_2)-\cos(2\theta_1(t))|<\epsilon_1\] for given $\epsilon_1 > 0 $. 
Next, from the assumption \ref{eq:angle_continuity}, there exists $\delta_1>0$ so that if $x\in B_{\delta_1}(0)$ we have that \[|u(t,x)-u_{\theta_1(t),\theta_2(t)}(x)|\leq \frac{C\alpha}{M} \log\frac{1}{|x|}.\] Here $u_{\theta_1(t),\theta_2(t)}(x)$ is the velocity associated with $S_{\theta_1(t),\theta_2(t)} \cup S_{\theta_1(t)+\pi,\theta_2(t)+\pi}$. 

Next, we  claim that there exist $\delta_2>0$ and $t_2<t_1$ so that for all $x_0\in B_{\delta_2}(0)$ we have that the solution to the ODE: \[\dot{\Phi}(t)=u(t,\Phi(t))\] \[\Phi(0)=x_0\] satisfies that $\Phi(t)\in B_{\delta_1}(0)$ for all $t\in [0,t_2]$. This is due to the trivial estimate: \[|\Phi(t)|\leq |x_0|^{\exp(-Ct)}\] which we know from the Yudovich theory. In particular, we may take $\delta_2=\delta_1^2$ and let $t_2 < t_1$ small independent of $\delta_1$. 

Summing up the preceding considerations, given $\epsilon>0$, we can find a radius $\delta_2>0$ and a time interval $[0,t_2]$ so that for all $x_0\in B_{\delta_2}(0)$ the associated trajectory $\Phi(t,x)$ remains in the ball $B_{\delta_1}(0)$ for all $t\in [0,t_2]$ where we know: \[|u(t,x) - \alpha \log\frac{1}{|x|}\begin{pmatrix}
x_2 \\ x_1 
\end{pmatrix}|\leq \epsilon\alpha  |x| \log\frac{1}{|x|}.\]

Notice that the vector field $(x_2,x_1)$ is tangent to the lines $z_2=z_1$ and $z_2=-z_1$ and that the flow associated to this vector field is hyperbolic near $0$. We now want to observe that $u$ is "almost" hyperbolic: indeed, consider the region $V_1=\{ 0\leq (1-\eta)z_1\leq z_2\leq (1+\eta)z_1 \}$ for $\eta=3\epsilon<\frac{1}{2}$. We claim that a particle $X(0)$ starting in $B_{\delta_2}(0)\cap V_1$ never escapes $B_{\delta_1}(0)\cap V_1$ for $t\in [0,t_2]$. Indeed, by construction, we know that $X(t)\in B_{\delta_1}(0)$ for all $t\in [0,t_2]$. Thus, to conclude, it suffices to show that $u(z)\cdot n(z)<0$ for all $z\in \partial V_{1} \cap B_{\delta_1}(0)\setminus\{0\}$ where $n(z)$ is the unique outer normal to $V_{1}$ at such $z$. Now we compute: If $z_2=(1+\eta)z_1$ we have \[\sqrt{(1+\eta)^2+1}(u(z)\cdot n(z))=-\alpha(1+\eta) u_1(z)+u_2(z)\leq\alpha(-(1+\eta)^2+1)z_1\log\frac{1}{|z|}+\epsilon\alpha |z|\log\frac{1}{|z|}<0.\]  Similarly, if $z_2=(1-\eta)z_1$
\[\sqrt{(1-\eta)^2+1}(u(z)\cdot n(z))=(1-\eta)u_1(z)-u_2(z)\leq \alpha((1-\eta)^2-1)z_1\log\frac{1}{|z|}+\epsilon\alpha |z|\log\frac{1}{|z|}<0.\]  Thus, any particle starting in $B_{\delta_2}\cap V_1$ stays in $V_1$ for all $t\in [0,\delta_2]$. Since we are only concerned with $t\in [0,t_2]$ and $x\in B_{\delta_2}(0),$ $V_1$ is an invariant region. It can be shown similarly that \[V_{3}:=-V_1\] is an invariant region. 
Now consider $V_2=\{z^\perp: z\in V_1\}$ and $V_4=\{z^\perp: z\in V_3\}.$ A similar calculation shows that $u(z)\cdot n(z)<0$ for points on $\partial V_2\cap B_{\delta_1}(0)$ and $\partial V_4\cap B_{\delta_1}(0)$. 

Now we are ready to show that $\theta_i$ cannot be continuous. We will do this by showing that most of the vortex patch is immediately pushed up close to the line $z_1=z_2.$ This will contradict the continuity of $\theta_i$. Take $X(0)\in T:= \{|z_2|< (1-\eta)z_1\}\cap B_{\delta}(0)$ for some $\delta\leq \delta_2$. By our choice of $t_2$ and the Yudovich bound, we know that $X(t)\in B_{\sqrt{\delta}}(0)$ for all $t\in [0,t_2]$. If for some $t_*\in [0,t_2]$, $X(t)$ hits $\partial T$, it must be that it hits the line $z_2=(1-\eta)z_1$ since the velocity field is pushing into $T$ on the lower boundary. Thereafter, $X(t)$ does not exit $V_1$ until after time $t_2$. Now let us study what happens for $t\in [0,t_*]$. By using the expansion of the velocity again, we find for $t\in [0,t_*]:$ 
\[\frac{d}{dt} \Big(\frac{X_2}{X_1}\Big)\geq \alpha \log\frac{1}{|X(t)|}(1-\frac{X_2^2}{X_1^2})-\epsilon\frac{|X|}{X_1}\log\frac{1}{|X|}\geq \alpha\log\frac{1}{|X|}(1-\frac{2\epsilon}{\alpha}-\frac{X_2^2}{X_1^2}).\]
Thus, \[\frac{d}{dt}\Big(\frac{X_2}{X_1}\Big)\geq \alpha \log\frac{1}{\sqrt{\delta}}(1-\frac{2\epsilon}{\alpha}-\frac{X_2^2}{X_1^2}).\] 
Letting $\lambda=\frac{X_2}{X_1},$ we see: 
\[\lambda'(t)\geq \frac{1}{2}\alpha \log\frac{1}{|X(0)|}(1-\frac{2\epsilon}{\alpha}-\lambda^2).\] 
Consequently, $\lambda$ is increasing, so long as $-\sqrt{1-\frac{2\epsilon}{\alpha}}<\lambda<\sqrt{1-\frac{2\epsilon}{\alpha}}$. In fact, so long as $\lambda<\frac{1}{2}$ we see: 
\[\lambda(t)\geq -\sqrt{1-\frac{2\epsilon}{\alpha}}+(\sqrt{1-\frac{2\epsilon}{\alpha}}+\lambda(0))\exp(ct\log\frac{1}{|X(0)|})\] for $c=\frac{1}{4}\alpha$.
In particular, if $\lambda(0)<0$, we have that $\lambda(t)$ will hit $0$ before the time $\frac{\log\Big( \sqrt{1-\frac{2\epsilon}{\alpha}}+\lambda(0)\Big)}{c\log|X_0|}.$

Now let $\tilde T=\{|z_2|\leq \sin(\theta_1(0))z_1\}.$ It is easy to see that if $X(0)\in B_\delta(0)\cap \tilde T$, then there exists $t_\delta\leq \frac{C}{|\log\delta|},$ with $C$ independent of $\epsilon$ so that $X_2(t)\geq 0$ for all $t\in [t_\delta,t_2].$ In particular, all particles initially in $\tilde T\cap B_\delta(0)$ are transported to the region $z_2\geq 0$ by the time $t_\delta$ and they never leave this region for $t\in [0,t_2]$. By the Yudovich bound  \[|X(0)|^{\exp(Ct)}\leq X(t)\leq |X(0)|^{\exp(-Ct)}\] we have that at time $t_\delta$, there are no trajectories in $B_\delta(0)\cap \{z_2\leq 0\leq z_1\}$ which began in $\tilde T$. Thus, all particles which lie in $B_\delta(0)\cap \{z_2\leq 0\leq z_1\}$ at time $t_\delta$ must have come from the set $B_{C\delta}(0)\cap \{\sin(2\theta_1(0))z_1\leq z_2\leq 2z_1\}$. By assumption, the total area of this arbitrarily small relative to the size of $|B_\delta(0)|$ as $\delta\rightarrow 0$. Now let's estimate the measure of $\Omega(t)\cap S_{-\theta_1(0),\theta_1(0)}\cap B_\delta(0)$ at time $t_\delta$. According to the assumption, the measure of this set should be approximately $2\theta_1(0)|B_\delta(0)|$. We will show that it is bounded by $(\theta_1(0)+o(1))|B_\delta(0)|$ as $\delta\rightarrow 0$, which is a contradiction.  Indeed, \[\Big|\Omega(t_\delta)\cap S_{-\theta_1(0),\theta_1(0)}\cap B_\delta(0)\Big|\leq |B_\delta(0)|\theta_1(0)+|\Omega(t_\delta)\cap S_{-\theta_1(0),0}\cap B_\delta(0)|\] \[ \leq |B_\delta(0)|\theta_1(0)+|\Omega(0)\cap B_{C\delta}(0)\cap S_{-\frac{3\pi}{8},-\theta_1(0)}|\leq |B_\delta(0)|\theta_1(0)+\frac{C}{M}|B_\delta(0)|2\theta_1(0). \]  This finishes the proof. 
\end{proof}

\subsection{Ill-posedness for vortex patches with corners of size $\pi/2$}

The purpose of this subsection is to establish the following proposition.

\begin{proposition}\label{prop:90}
	For generic initial vortex patches $\Omega_0$ with $\Omega_0\cap B_1(0)=S_{-\pi/4,\pi/4}\cup S_{{3\pi/4},{5\pi/4}}$ and $\Omega_0$ smooth and compactly supported outside of $B_{\frac{1}{2}}(0)$ we have that the associated unique vortex patch solution $\Omega(t)$ does not keep a pair of regular corners of size $\pi/2$ for positive time. That is, $\partial \Omega(t)\cap B_{\frac{1}{2}}(0)$ cannot be written as the intersection of two $C^{1,\alpha}$ curves lying in $C([0,\delta]; C^{1,\alpha})$ for any $\delta>0$ and $\alpha>0$.
\end{proposition}
\begin{remark}
	In the above, "generic" means that given any vortex patch $\Omega_0$ with a 90-degree corner at the origin, it can be perturbed very slightly by a smooth perturbation very far away from the origin so that the corner does not remain regular for positive time. 
\end{remark}

\begin{remark}
	It will be apparent from the proof that we actually only need the two curves to be uniformly $C^1$ continuously in time.  
\end{remark}
\begin{proof}
	 { 
		Assume that for some $\delta, \alpha > 0$, the boundary $\partial\Omega(t) \cap B_{\frac{1}{2}(0)}$ is written as the intersection of two $C^{1,\alpha}$ curves. For such curves to exist (for any vortex patch with a right angle at time zero), a necessary condition is that the curves intersect at a right angle for all $t\in [0,\delta]$. Let $\alpha(t)$ and $\beta(t)$ denote the tangent vectors to the two curves. This is an immediate consequence of Theorem \ref{thm:illposed-acute} proved in the above. We claim that for $0\le t \le\delta$ the following holds: \[\alpha'(t)=\lim_{\lambda\rightarrow 0^+}\frac{u(t,\lambda\alpha(t))\cdot \alpha(t)^\perp}{\lambda}\alpha(t)^\perp,\] and \[\beta'(t)=\lim_{\lambda\rightarrow 0^+}\frac{u(t,\lambda\beta(t))\cdot \beta(t)^\perp}{\lambda}\beta(t)^\perp.\] To establish the claim, we just need to show that so long as $\Omega$ can be locally written as an intersection of two $C^{1,\alpha}$ curves for some $\alpha>0$, then the limits above exist. The definition of tangent vector then implies that the tangent vector actually evolves as claimed. The existence of the limit follows since $\frac{1}{|x|}u(t,x)\cdot x^\perp$ can be written as a smooth function of $\frac{x}{|x|}$ plus a logarithmic term which vanishes when $x=\alpha(t)$ and $x=\beta(t)$ plus a term which vanishes as $|x|\rightarrow 0$. To see that the log term vanishes in this case, observe that (using \eqref{eq:velgrad_sectors}) \[u(x)=\log|x| A x +l.o.t.\] where "l.o.t." is the part which is smooth in $\frac{x}{|x|}$ plus the vanishing term, where $A$ is a constant matrix whose eigenvectors are $\alpha(t)$ and $\beta(t)$. Consequently, we must have that $$A\alpha(t) \cdot \alpha(t)^\perp= A\beta(t) \cdot \beta(t)^\perp=0.$$
		 
		In fact, the assumption implies that the limits are continuous in time. As a consequence, if we have that \[\lim_{\lambda\rightarrow 0^+}\frac{u_0(\lambda\alpha(0))\cdot \alpha(0)^\perp}{\lambda}\not=\lim_{\lambda\rightarrow 0^+}\frac{u_0(\lambda\beta(0))\cdot \beta(0)^\perp}{\lambda},\] then we will have a contradiction. If equality holds for $\Omega_0$, we can just choose a small perturbation (keeping the 2-fold symmetry) of the patch away from the origin so that equality does not hold at time zero. To be more precise, for a two-fold symmetric perturbation $\tilde{\omega}$ supported away from the origin, if we denote the corresponding velocity by $\tilde{u}$, we may arrange that locally near $x = 0$ \begin{equation*}
		\begin{split}
		\tilde{u}(x) = \begin{pmatrix}
		a & 0 \\
		0 & -a
		\end{pmatrix} \begin{pmatrix}
		x_1 \\
		x_2
		\end{pmatrix} + O(|x|^2) 
		\end{split} 
		\end{equation*}  with $a \ne 0$ since $a$ is a constant multiple of \begin{equation*}
		\begin{split}
		\int_{\mathbb{R}^2} \frac{y_1y_2}{|y|^4} \tilde{\omega}(y) dy 
		\end{split}
		\end{equation*} which can be made nonzero. } 
\end{proof}

\subsection{Loss of Lipschitz continuity for locally symmetric patches}
In this subsection we prove the following result:
\begin{theorem}\label{thm:local-symmetry}
	There is a vortex patch $\Omega_0$ with corners which is locally four-fold symmetric about the origin for which the velocity field satisfies:
	\[\sup_{0<t<\delta}\V\nabla u(t)\V_{L^\infty}=+\infty\] for all $\delta>0$ despite the initial velocity field being Lipschitz continuous. 
\end{theorem}  This result shows that the global symmetry assumption in Theorem \ref{mainthm:wellposedness} \emph{cannot} be replaced by a local one. This is done using the framework introduced in \cite{EM1}. 
\begin{proof}
Let $\tilde\Omega_0$ be a compactly supported vortex patch which near the origin is equal to $ \theta\in \bigcup [-\pi/8+k\pi/2,\pi/8+k\pi/2], k=0,1,2,3\}$. In particular, we assume that $\tilde\Omega_0$ satisfies the conditions of Theorem \ref{mainthm:wellposedness}. Now we define $\Omega_0$ by $\Omega_0={\tilde\Omega_0} \cup {\Omega_p},$ where $0\not\in\overline{\Omega_p}$  and $\Omega_p$ is compactly supported and 2-fold symmetric. The index $p$ in $\Omega_p$ stands for perturbation. Note that $\Omega_0$ is 2-fold symmetric but need not be 4-fold symmetric. We further require that $\Omega_0$ be infinitely smooth away from $0$. 
We now claim that for some special choices of $\Omega_p$ we have that for all $\delta>0$ \[\sup_{0<t<\delta}\V\nabla u(t)\V_{L^\infty}=+\infty.\] Following Section 8 of \cite{EM1}, to prove this, all we need to show is that the associated velocity field $u_0$ satisfies the following estimates: 
\[\nabla u_0\in L^\infty,\] 
\[\V D^2p_0\V_{L^q}\geq C q,\,\,\,\]
for some fixed constant $C>0$ and all $q\geq 2$, where $p_0$ is the initial Eulerian pressure given by:
\[\Delta p_0=2\det(\nabla u_0).\]
Now we write $u_0=\tilde u_0+u_{0,p}$ where $\tilde u_0=\nabla^\perp\Delta^{-1}\chi_{\tilde\Omega_0}$ while $u_{0,p}=\nabla^\perp\Delta^{-1}\chi_{\Omega_p}.$ Note that $u_{0,p}$ is infinitely smooth in some neighborhood of $0$. In particular, \[u_{0,p}(x)=\nabla u_{0,p}(0) x+O(|x|^2)\] as $|x|\rightarrow 0$. On the other hand, 
\[\tilde u_0(x)=G(\theta) x+2G'(\theta)x^\perp +O(|x|^2)\] as $|x|\rightarrow 0$, with $G$ the unique $\frac{\pi}{2}$-periodic solution to \[4G+G'' = \chi_{[-\frac{\pi}{8},\frac{\pi}{8}]}.\] Note that $G$ is even with respect to $\theta$. 
It is then easy to see that the $O(|x|^2)$ terms in the expansions of $u_{0,p}$ and $\tilde u_0$ are negligible and that, in a neighborhood of $0$, \[\Delta p_0=2\det(\nabla u_{0,p}(0)+\nabla (G(\theta)x+2G'(\theta) x^\perp))+f,\] with $f\in C^\alpha$ for all $\alpha<1$. 
Now let's choose $u_{0,p}$ so that $$\nabla u_{0,p}(0)= \left[ {\begin{array}{cc}0 & K \\0 & 0 \\\end{array} } \right]$$ for some large constant $K$. Then we see that \[\Delta p_0=2K\partial_1(G(\theta)x_2-2G'(\theta)x_1)+F\] with $\V F\V_{L^\infty}\leq C$ for some fixed constant $C$ independent of $K$. In particular, \[ \V D^2\Delta^{-1} (F\chi)\V_{L^q}\leq Cq\] for all $q\in [1,\infty)$ where $\chi$ is a smooth cut-off function which is identically $1$ near zero and $C$ again is a fixed constant independent of $K$. Now let us consider $|D^2\Delta^{-1}(\partial_1(G(\theta)x_2-2G'(\theta)x_1)|_{L^q}.$ It is easy to see that it suffices to show that there exists a small constant $c>0$ so that \[|\nabla\Delta^{-1}(\partial_1(G(\theta)x_2-2G'(\theta)x_1)|\geq c|x|\log\frac{1}{|x|}.\] As we have shown in our expansion of the velocity field in Lemma \ref{lem:key_radial}, it suffices to show that the quantity $I$ below is non-zero: \begin{equation*}
\begin{split}
I & :=\int_0^{2\pi}\int_{|x|}^1\frac{\sin(2\theta)}{r^2}(\partial_1(G(\theta)x_2-2G'(\theta)x_1)rdrd\theta \\
& =\log\frac{1}{|x|}\int_0^{2\pi}\sin(2\theta)(-G'(\theta)\sin^2(\theta)+2G''(\theta)\sin(\theta)\cos(\theta)-2G'(\theta))d\theta. 
\end{split}
\end{equation*}  Noting that $G$ is $\pi/2$ periodic, we see that $\int_0^{2\pi} G'(\theta)\sin(2\theta)=0$. Moreover, we note that \[\sin(2\theta)\sin^2(\theta)=\frac{1}{2}\sin(2\theta)-\frac{1}{4}\sin(4\theta)\] and \[\sin(2\theta)\sin(\theta)\cos(\theta)=\frac{1}{2}\sin^{2}(2\theta)=\frac{1}{4}(1-\cos(4\theta)).\] 
Now we see: \[\int_0^{2\pi}G'(\theta)\sin(2\theta)=0,\] \[-\int_0^{2\pi} G'(\theta)\sin(2\theta)\sin^2(\theta)=\frac{1}{4}\int_0^{2\pi}G'(\theta)\sin(4\theta),\]
\[2\int_0^{2\pi}G''(\theta)\sin(\theta)\cos(\theta)\sin(2\theta)=-\frac{1}{2}\int_0^{2\pi}G''(\theta)\cos(4\theta).\]
Thus, \[I=7\log\frac{1}{|x|}\int_0^{2\pi}G(\theta)\cos(4\theta)d\theta\not=0\] as can be easily seen from the relation \[G''(\theta)+4G(\theta)=\chi_{[-\frac{\pi}{8},\frac{\pi}{8}]}\] by multiplying by $\cos(4\theta)$. Thus by choosing $K$ sufficiently large, we can ensure that \[\V\nabla u_0\V_{L^\infty}<\infty\] while \[\V D^2 p_0\V_{L^q}\geq q\] for all $q\geq 2$. Now we may apply the arguments of Section 7 of \cite{EM1} to conclude. 
\end{proof}

\section{Effective system for the boundary evolution near the corner}\label{effective}

In the previous section we established that a non-right angle in a vortex patch does not propagate continuously in time. Then the following question may be raised: what exactly happens to the vortex patch near the corner? 
The purpose of this section is to propose some asymptotic models which we believe describe the behavior of a vortex patch. The key tool is the (rigorous) expansion of the velocity field in the previous section.

\subsection{The formal evolution equation near the corner}

From now on, we shall assume that the vorticity is two-fold symmetric around the origin. Here, this symmetry assumption is just for simplicity and does not seem to alter the qualitative dynamics near the origin, except for the translation of the patch which can be fixed using the Galilean invariance. Moreover, assume that the vorticity is supported in a small ball of radius less than 1, which is guaranteed for some nonempty time interval if the initial vorticity has this property. Then, we have from Lemma \ref{lem:key_radial} that \begin{equation}\label{eq:key_radial_general}
\begin{split}
u(r,\theta) =   \frac{1}{2\pi} \begin{pmatrix}
\cos(\theta) \\
-\sin(\theta) 
\end{pmatrix} rI^s(r) - \frac{1}{2\pi} \begin{pmatrix}
\sin(\theta) \\
\cos(\theta)
\end{pmatrix} rI^c(r)  + O(r)
\end{split}
\end{equation}  where \begin{equation}\label{eq:key_integrals}
\begin{split}
I^s(r) = \int_r^1 \int_0^{2\pi} \sin(2\theta') \frac{\omega(s,\theta')}{s} d\theta' ds, \quad I^c(r) = \int_r^1 \int_0^{2\pi} \cos(2\theta') \frac{\omega(s,\theta')}{s} d\theta' ds.
\end{split}
\end{equation} From \eqref{eq:key_integrals}, note that in the limit $r \rightarrow 0^+$, if for \textit{majority} of $s \in [r,1]$ we have a lower bound on the integral of $\omega(s,\cdot)$ against either $\sin(2\theta)$ or $\cos(2\theta)$ on the circle, then either $|I^s(r)| \gg 1$ or $|I^c(r)| \gg 1$, and it is reasonable to believe that the behavior of the vorticity at the origin is determined only by the first two terms in the right hand side of \eqref{eq:key_radial_general}. Moreover, the term of order $r$ cannot account for a sudden change of angle, such as  instantaneous cusp or spiral formation.  Then, we may formally consider the following modified Euler equation: \begin{equation}\label{eq:formal_system_prelim}
\begin{split}
\pr_t\omega + \frac{1}{2\pi} \left[ \begin{pmatrix}
\cos(\theta) \\
-\sin(\theta) 
\end{pmatrix} I^s(r) - \begin{pmatrix}
\sin(\theta) \\
\cos(\theta)
\end{pmatrix}  I^c(r) \right] r\cdot\nabla \omega = 0. 
\end{split}
\end{equation} The measure $\frac{ds}{s}$ in the expression \eqref{eq:key_integrals} suggests that we write the vorticity as \begin{equation}\label{eq:formal_ansatz}
\begin{split}
\omega(t,r,\theta) = \mathsf{g}(t\ln\frac{1}{r},\theta) + \mbox{remainder},
\end{split}
\end{equation} where we formally assume that the remainder term is negligible for some time interval $[0,t^*]$ in the limit $r \rightarrow 0^+$ compared to the $\mathsf{g}$-term. Using \eqref{eq:formal_ansatz} and neglecting the terms involving the remainder, we obtain from \begin{equation*}
\begin{split}
 I^s(r) = \frac{1}{t} \int_0^{t\ln\frac{1}{r}} \int_0^{2\pi} \sin(2\theta') \mathsf{g}(t\ln\frac{1}{s},\theta')d\theta' d(t\ln\frac{1}{s}) ,
\end{split}
\end{equation*} \begin{equation*}
\begin{split}
  I^c(r) = \frac{1}{t} \int_0^{t\ln\frac{1}{r}} \int_0^{2\pi} \cos(2\theta') \mathsf{g}(t\ln\frac{1}{s},\theta')d\theta' d(t\ln\frac{1}{s}) 
\end{split}
\end{equation*} that after introducing the variable $\tau = t\ln\frac{1}{r}$, \begin{equation*}
\begin{split}
&\ln\frac{1}{r} \pr_{\tau}\mathsf{g}  + \frac{1}{2\pi t} \left[  \begin{pmatrix}
\cos(\theta) \\
-\sin(\theta) 
\end{pmatrix} \int_0^\tau \int_0^{2\pi} \sin(2\theta') \mathsf{g}(\tau',\theta')d\theta' d\tau' - \begin{pmatrix}
\sin(\theta) \\
\cos(\theta)
\end{pmatrix} \int_0^\tau \int_0^{2\pi} \cos(2\theta') \mathsf{g}(\tau',\theta')d\theta' d\tau'  \right] \\
&\qquad \cdot \left[ \frac{1}{r} \pr_{\theta}\mathsf{g} \begin{pmatrix}
-\sin(\theta) \\
\cos(\theta)
\end{pmatrix} - \frac{t}{r} \pr_{\tau}\mathsf{g} \begin{pmatrix}
\cos(\theta)\\
\sin(\theta)
\end{pmatrix} \right] = 0.
\end{split}
\end{equation*} Dividing by $\ln\frac{1}{r}$, \begin{equation*}
\begin{split}
&\pr_{\tau}\mathsf{g}-\frac{1}{2\pi\tau} \left[  \sin(2\theta)\left(\int_0^\tau \int_0^{2\pi} \sin(2\theta') \mathsf{g}(\tau',\theta')d\theta' d\tau'\right) +\cos(2\theta) \left( \int_0^\tau \int_0^{2\pi} \cos(2\theta') \mathsf{g}(\tau',\theta')d\theta' d\tau'  \right)\right]\pr_{\theta} \mathsf{g}  \\
&\qquad = \frac{t}{\tau} \left[ \cos(2\theta)\left(\int_0^\tau \int_0^{2\pi} \sin(2\theta') \mathsf{g}(\tau',\theta')d\theta' d\tau'\right) +\sin(2\theta) \left( \int_0^\tau \int_0^{2\pi} \cos(2\theta') \mathsf{g}(\tau',\theta')d\theta' d\tau'  \right)    \right]\pr_{\tau}\mathsf{g}.
\end{split}
\end{equation*} For $0 \le t \ll 1$, again formally we may drop the entire right hand side, which results in the following transport system for $\mathsf{g}$: \begin{equation}\label{eq:formal_model}
\begin{split}
\pr_{\tau}\mathsf{g}-\frac{1}{2\pi\tau} \left[  \sin(2\theta)\left(\int_0^\tau \int_0^{2\pi} \sin(2\theta') \mathsf{g}(\tau',\theta')d\theta' d\tau'\right) +\cos(2\theta) \left( \int_0^\tau \int_0^{2\pi} \cos(2\theta') \mathsf{g}(\tau',\theta')d\theta' d\tau'  \right)\right]\pr_{\theta} \mathsf{g}  = 0. 
\end{split}
\end{equation} We now investigate the system \eqref{eq:formal_model} in a number of concrete situations. 

\subsection{Evolution of a corner under the odd symmetry}

We now consider the simpler case of vorticity which is odd in the $x_1$-axis. Together with the two-fold symmetry assumption, we have that the vorticity is odd also in the $x_2$-axis. Then, in \eqref{eq:formal_model}, the term involving integration against $\cos(2\theta')$ vanishes, and we are left with \begin{equation}\label{eq:formal_model_odd}
\begin{split}
0 = \pr_{\tau} \mathsf{g} - \left(\frac{1}{2\pi\tau}\int_0^\tau \int_0^{2\pi}\sin(2\theta')\mathsf{g}(\tau',\theta') d\theta'd\tau'\right) \sin(2\theta)\pr_{\theta}\mathsf{g}.
\end{split}
\end{equation} 
Using the Fourier expansion
\begin{equation}\label{eq:main_exp_odd_case}
\begin{split}
\mathsf{g}(\tau,\theta) = \sum_{ k \ge 1} g_k(\tau)\sin(2k\theta),
\end{split}
\end{equation} we may rewrite \eqref{eq:formal_model_odd} in the following equivalent form: \begin{equation}\label{eq:main_eq_odd_fourier}
\begin{split}
\dot{g}_k(\tau) = \left(  \frac{1}{2\pi\tau}\int_0^\tau g_1(\tau')d\tau' \right) \left((k-1)g_{k-1}(\tau) - (k+1)g_{k+1}(\tau)\right),
\end{split}
\end{equation} where we have used the convention that $g_0 \equiv 0$. From \eqref{eq:main_eq_odd_fourier}, it is straightforward to see that the Bahouri-Chemin solution is characterized as the unique stationary solution to \eqref{eq:formal_model_odd}. 

We now consider the case where the initial data is locally a union of corners attached on the $x_1$-axis: that is, $$\mathsf{g}(0,\theta) = \pm \left( \mathbf{1}_{[0,A_0] \cup [\pi,A_0+\pi]} - \mathbf{1}_{[-A_0,0] \cup [-A_0 + \pi,\pi]}\right)$$ for some $0 \le A_0 < \pi/2$. Then, we have that $\mathsf{g}(\tau,\theta) = \pm(\mathbf{1}_{[0,A(\tau)] \cup [\pi,A(\tau)+\pi]} - \mathbf{1}_{[-A(\tau),0]\cup [-A(\tau)+\pi,\pi]})$ with $A(0) = A_0$ from the transport nature of the system \eqref{eq:formal_model_odd}. In the case of the negative sign (i.e. when the vorticity is negative on the positive quadrant), we obtain from \eqref{eq:formal_model_odd} that $A(\cdot)$ satisfies \begin{equation}\label{eq:char_odd_case_negative}
\begin{split}
\dot{A}(\tau) &=  \frac{4\sin(2A(\tau))}{2\pi\tau}\int_0^{\tau} \int_0^{A(\tau')}\sin(2\theta')d\theta' d\tau' \\
&= \frac{\sin(2A(\tau))}{\pi\tau}\int_0^{\tau} 1 - \cos(2A(\tau'))d\tau',
\end{split}
\end{equation} and in the opposite sign case,  \begin{equation}\label{eq:characteristic_odd_case}
\begin{split}
\dot{A}(\tau) = -\frac{\sin(2A(\tau))}{\pi\tau}\int_0^{\tau} 1 - \cos(2A(\tau'))d\tau'.
\end{split}
\end{equation} Note that \eqref{eq:char_odd_case_negative}, \eqref{eq:characteristic_odd_case} can be rewritten into the form of a second order ordinary differential equation; using $a := 2A$, we have \begin{equation}\label{eq:char_odd_case_ODE}
\begin{split}
\ddot{a}(\tau) = \frac{\cos(a(\tau))}{\sin(a(\tau))} (\dot{a}(\tau))^2 - \frac{1}{\tau} \left( \dot{a}(\tau) \mp \frac{2}{\pi}  \sin(a(\tau)) (1 - \cos(a(\tau)) )  \right) 
\end{split}
\end{equation} with initial data \begin{equation}\label{eq:char_odd_case_ODE_initial}
\begin{split}
a(0) = 2A_0, \quad \dot{a}(0) = \pm  \frac{2}{\pi}  \sin(2A_0) (1 - \cos(2A_0) ) .
\end{split}
\end{equation} Using \eqref{eq:char_odd_case_negative}, it is direct to see that for all $0 \le \tau$, $A_0 \le A(\tau) < \frac{\pi}{2}$. Since $\cos(2\cdot)$ is a decreasing function on $[0,\pi/2]$, we have that \begin{equation*}
\begin{split}
\dot{A}(\tau) \ge \frac{\sin(2A(\tau))}{\pi} (1 - \cos(2A)),
\end{split}
\end{equation*} and this guarantees that $A(\tau) \rightarrow \frac{\pi}{2}$ at exponential speed as $ \tau \rightarrow +\infty$. The dynamics is more delicate in the case of \eqref{eq:characteristic_odd_case};  while it is not difficult to show that the solution decays to 0 as $\tau$ goes to infinity with bounds \begin{equation}\label{eq:odd_decay_bounds}
\begin{split}
\frac{c}{1 + \tau} \le A(\tau) \le \frac{C}{1 + \tau^{1/2}},\quad \tau \ge 0 ,
\end{split}
\end{equation} obtaining the rate of decay is an interesting problem. Direct numerical simulations show that $A$ decays as $\tau^{-1}$. Let us show that the integral \begin{equation}\label{eq:integral}
\begin{split}
\int_0^\infty 1 - \cos(2A(\tau))d\tau 
\end{split}
\end{equation} is bounded by a constant depending only on $A_0$. Otherwise, for any large $M > 0$, one can find $\tau^* = \tau^*(M) > 0$ such that \begin{equation*}
\begin{split}
\forall \tau \ge \tau^*,\quad \int_0^\tau 1 - \cos(2A(\tau'))d\tau' \ge M. 
\end{split}
\end{equation*} Then \begin{equation*}
\begin{split}
\forall \tau \ge \max(\tau^*,\tau') ,\quad \dot{A}(\tau) \le -\frac{M}{10\pi\tau} {A}(\tau)
\end{split}
\end{equation*} where $\tau' = \tau'(A_0) > 0$ is chosen that $\sin(2A(\tau')) \ge \frac{A(\tau')}{10}$,  and \begin{equation*}
\begin{split}
\forall \tau \ge  \max(\tau^*,\tau'),\quad\frac{c}{1+\tau} \le A(\tau) \le C(M) \tau^{-\frac{M}{10}},
\end{split}
\end{equation*} where $C(M) > 0$ is a constant depending only on $M$. Taking $\tau \rightarrow +\infty$, we obtain a contradiction. If the ansatz \eqref{eq:formal_ansatz} were correct, the bound on the integral in \eqref{eq:integral} implies that for small $t > 0$, we have \begin{equation*}
\begin{split}
\lim_{ r\rightarrow 0^+}\frac{|u(t,r,\theta)|}{r} \le \frac{C}{t},
\end{split}
\end{equation*} where $u(t,\cdot)$ is the solution associated with initial vorticity given by $\omega_0(r,\theta) = \mathsf{g}(0,\theta) \mathbf{1}_{ \{ r \le \frac{1}{10} \} }$. 

\subsection{Evolution of a single corner}
 
We now consider initial data of the form \begin{equation*}
\begin{split}
\mathsf{g}(0,\theta) = \mathbf{1}_{[A_0-B_0,A_0+B_0]} + \mathbf{1}_{[A_0-B_0+\pi,A_0+B_0+\pi]}.
\end{split}
\end{equation*} The initial corner is centered at $A_0$ and has angle $2B_0$, with two-fold symmetry.  Without loss of generality, we may set $A_0 = 0$. From \begin{equation*}
\begin{split}
\mathsf{g}(\tau,\theta) = \mathbf{1}_{[A-B,A+B]} + \mathbf{1}_{[A-B+\pi,A+B+\pi]},
\end{split}
\end{equation*} we obtain, by evaluating \eqref{eq:formal_model} at $A + B$ and $A-B$, that \begin{equation*}
\begin{split}
(A \pm B)'(\tau)   = -\frac{1}{2\pi\tau} &\left[ \sin(2A(\tau) \pm 2B(\tau)) \int_0^\tau \cos(2A(\tau')-2B(\tau')) - \cos(2A(\tau') + 2B(\tau')) d\tau' \right. \\
&\qquad \left. +\cos(2A(\tau) \pm 2B(\tau)) \int_0^\tau \sin(2A(\tau') + 2B(\tau')) - \sin(2A(\tau') - 2B(\tau')) d\tau'   \right].
\end{split}
\end{equation*} It follows that \begin{equation}\label{eq:G'}
\begin{split}
\tau B'(\tau) = -\frac{1}{\pi}&\left( \sin(2B)\cos(2A) \int_{ 0}^{\tau} \sin(2B(\tau'))\sin(2A(\tau')) d\tau' \right. \\
&\quad\left. - \sin(2B)\sin(2A) \int_0^\tau \sin(2B(\tau'))\cos(2A(\tau'))d\tau'\right)
\end{split}
\end{equation} and \begin{equation}\label{eq:F'}
\begin{split}
\tau A'(\tau) = -\frac{1}{\pi}&\left( \cos(2B)\sin(2A) \int_{ 0}^{\tau} \sin(2B(\tau'))\sin(2A(\tau')) d\tau' \right. \\
&\quad \left.  + \cos(2B)\cos(2A) \int_0^\tau \sin(2B(\tau'))\cos(2A(\tau'))d\tau'\right).
\end{split}
\end{equation} We may rewrite \eqref{eq:G'}--\eqref{eq:F'} in the form \begin{equation*}
\begin{split}
B'(\tau) = -\frac{\sin(2B(\tau))}{\pi\tau} \int_0^\tau \sin(2B(\tau')) \sin(2A(\tau') - 2A(\tau))  d\tau'
\end{split}
\end{equation*} and \begin{equation*}
\begin{split}
A'(\tau) = -\frac{\cos(2B(\tau))}{\pi\tau } \int_0^\tau \sin(2B(\tau')) \cos(2A(\tau') - 2A(\tau))  d\tau'.
\end{split}
\end{equation*} This implies that if $-\pi/4 \le A \le 0$ in an interval $[0,\tau']$, then $A'(\tau) \le 0$, which in turn implies $B'(\tau) \ge 0$. 

Similarly as in the case of odd symmetry, it is possible to turn the above system into a system of second order ordinary differential equations.   Differentiating both sides of \eqref{eq:G'}--\eqref{eq:F'} gives \begin{equation*}
\tau\begin{pmatrix}
B''  \\  A''
\end{pmatrix} = - \begin{pmatrix}
B' \\  A'
\end{pmatrix} - \frac{1}{\pi}\begin{pmatrix}
0 \\ \cos(2B)\sin(2B)
\end{pmatrix} + M' M^{-1} \tau\begin{pmatrix}
B' \\  A' 
\end{pmatrix}
\end{equation*} with \begin{equation*}
\begin{split}
M = \begin{pmatrix}
\sin(2B)\cos(2 A) & -\sin(2B)\sin(2 A) \\
\cos(2B)\sin(2 A) & \cos(2B)\cos(2 A)
\end{pmatrix}.
\end{split}
\end{equation*}
Then
\begin{equation*}
\begin{split}
M' = 2B' \begin{pmatrix}
\cos(2B)\cos(2 A) & -\cos(2B)\sin(2 A) \\
-\sin(2B)\sin(2 A) & -\sin(2B)\cos(2 A)
\end{pmatrix} + 2A' \begin{pmatrix}
-\sin(2B)\sin(2 A) & -\sin(2B)\cos(2 A) \\
\cos(2B)\cos(2 A) & -\cos(2B)\sin(2 A)
\end{pmatrix}
\end{split}
\end{equation*} and \begin{equation*}
\begin{split}
M^{-1} = \frac{1}{\sin(2B)\cos(2B)} \begin{pmatrix}
\cos(2B)\cos(2 A) & \sin(2B)\sin(2 A) \\
-\cos(2B)\sin(2 A) & \sin(2B)\cos(2 A)
\end{pmatrix}.
\end{split}
\end{equation*} Then finally, we arrive at the system \begin{equation}\label{eq:ode_single_corner}
\tau\begin{pmatrix}
B''  \\  A''
\end{pmatrix} = - \begin{pmatrix}
B' \\  A'
\end{pmatrix} - \frac{1}{\pi}\begin{pmatrix}
0 \\ \cos(2B)\sin(2B)
\end{pmatrix} + \frac{2\tau}{\sin(2B)\cos(2B)} \begin{pmatrix}
\cos^2(2B)(B')^2 - \sin^2(2B)( A')^2 \\
-\sin^2(2B)B' A' + \cos^2(2B) B' A' 
\end{pmatrix}.
\end{equation} The initial condition is given by \begin{equation}\label{eq:ode_single_corner_id}
\begin{split}
B(0) =B_0, B'(0) = 0, A(0) = 0, A'(0) = -\frac{1}{\pi}\cos(2B_0)\sin(2B_0). 
\end{split}
\end{equation} Numerical simulations suggest that $A \rightarrow A_\infty$ and $B \rightarrow 0$ for some constant $A_\infty $ depending on $B_0$, as $\tau \rightarrow +\infty$. This suggests instantaneous cusping without spiral formation, which is comparable with the direct numerical study on vortex patches  {in \cite{CS,CD}.}  

\appendix

\section*{Appendix}

\section{Local well-posedness for symmetric patches}\label{sec:LWP}
 
The local well-posedness results for smooth vortex patches is usually obtained via an iteration scheme, using the contour dynamics equation (see for instance \cite{B}, \cite{MB}). An alternative approach which works directly with the flow maps restricted to the patch was described in an illuminating work of Huang \cite{Hu}.\footnote{The main result of this work is that $C^{1,\alpha}$-patches in 3D is locally well-posed under the Euler equations (see also an earlier work of Serfati \cite{Serfati3D}). In the three-dimensional case, the vorticity does not remain a constant inside the patch even if initially so, and therefore the contour dynamics approach is not available.} This method originates from a previous work of Friedman and Huang \cite{FH}, and it seems to be applicable for a wide variety of situations. We shall adopt this approach to show local well-posedness (as well as continuation criteria) in the setting of Section \ref{sec:intermediate}, i.e., patches admitting a level set function $\phi$ with $\nabla^\perp\phi \in \mathring{C}^\alpha$. 

The starting point of this method is to write the 2D Euler equation purely in terms of the flow maps: \begin{equation}\label{eq:flow_formulation}
\begin{split}
\Phi(t,x) &= x + \int_0^t \int_{\mathbb{R}^2} K(\Phi(x,s)-y)\omega_0(\Phi^{-1}_t(y)) dy ds,  \\
&= x + \int_0^t \int_{\Omega_0} K(\Phi(s,x) - \Phi(s,z))  dzds. 
\end{split}
\end{equation} At this point, note that we only need to know $\Phi(t,\cdot)$ on $\Omega_0$ to determine the velocity of the Euler equation everywhere in $\mathbb{R}^2$. It is easy to show that the above formulation is equivalent to the (usual) weak formulation of the 2D Euler equations under $\omega \in L^\infty \cap L^1$, and the Yudovich theorem gives that there is a unique solution $\Phi$ satisfying \eqref{eq:flow_formulation}.

The formulation \ref{eq:flow_formulation} suggests one to build an iteration scheme; all that is necessary to appropriately define the space of functions. Following \cite{Hu}, we consider \begin{equation}\label{eq:LWP_space}
\begin{split}
B(M,T) &= \left\{ \Phi(t,x) \in X : \Phi(0,x) = x, \Phi(t,0) = 0, \Vert \Phi \Vert_{X} \le M, \sup_{\overline{\Omega}_0 \times [0,T]} |\nabla\Phi(t,x) - I | \le 1/2 \right\}
\end{split}
\end{equation} where the space $X$ is defined for functions $\Phi : \overline{\Omega}_0 \times [0,T] \rightarrow \mathbb{R}^2$ with $\det(\nabla_x\Phi) \equiv 1$ by the norm \begin{equation*}
\begin{split}
\V \Phi\V_X = \sup_{t \in [0,T]} \left( \V \nabla_x \Phi \V_{\mathring{C}^\alpha(\overline{\Omega}_0)} + \V \partial_t \Phi\V_{L^\infty(\overline{\Omega}_0)} \right).
\end{split}
\end{equation*} That is, we have simply replaced the assumption in \cite{Hu} that $\nabla\Phi(t,\cdot)$ is uniformly $C^{\alpha}$ (up to the boundary of $\overline{\Omega}_0$) by $\mathring{C}^\alpha$. The extra assumption that $\Phi(t,0) = 0$ holds will be guaranteed by symmetry. Under the assumption $|\nabla\Phi(t,x) - I| \le 1/2$, it follows that the inverse map $\Phi_t^{-1} : \overline{\Omega}_t \rightarrow \overline{\Omega}_0$ is Lipschitz with $|\nabla \Phi_t^{-1}| \le 2$. Moreover, it is elementary to verify that for $\Phi \in B(M,T)$, $\nabla \Phi_t^{-1}$ belongs to $\mathring{C}^\alpha(\overline{\Omega}_t)$ with norm depending only on $M$ (see below Lemma \ref{lem:calculus}). 

Then, we define a mapping $F$, \begin{equation*}
\begin{split}
F(\Phi)(t,x) := x + \int_0^t \int_{\Omega_0} K(\Phi(s,x) - \Phi(s,z))  dzds,
\end{split}
\end{equation*} so that a fixed point of $F$ provides a solution to the 2D Euler equation on $[0,T]$ with initial data $\omega_0 = \chi_{\Omega_0}$. 

We need to propagate the regularity of $\phi$ in time, where the level set function $\phi_0$ is given together with the initial data $\Omega_0$. We observe that, as long as $\Phi \in B(M,T)$, by defining \begin{equation*}
\begin{split}
\phi(t,x) := \phi_0 (\Phi^{-1}_t(x)),\qquad x \in \overline{\Omega}_t,
\end{split}
\end{equation*} we have \begin{equation}\label{eq:level_set_bound_via_flow}
\begin{split}
\sup_{t \in [0,T]} \mathring{\Gamma}_t :=   \sup_{t \in [0,T]}\left(\frac{\V \nabla^\perp\phi_t(\cdot)\V_{\mathring{C}^\alpha(\Omega_t)}}{\V \nabla^\perp\phi_t(\cdot)\V_{\inf(\partial\Omega_t)}} \right)^{1/\alpha} \le  \mathfrak{C}(M), 
\end{split}
\end{equation} where, here and in the following, we use the notation $\mathfrak{C}(M)$ to denote a positive and increasing function of $M > 0$ depending on $\mathring{\Gamma}_0$. This function may change from a line to another. 

We are now in a position to formally state the local well-posedness results:
\begin{proposition}\label{prop:LWP}
	Assume that $\Omega_0$ is $m$-fold symmetric for some $m \ge 3$ admitting a level set $\phi_0$ satisfying Definition \ref{def:level_set_circle}. Then there exists some $T > 0$, depending only on $\mathring{\Gamma}_0$, such that there is a unique local solution $\Phi \in X$ of \eqref{eq:flow_formulation}. In particular, we can extend the solution beyond some $T^*$ as long as $\sup_{t \in [0,T^*)}\Gamma_t < + \infty$. 
\end{proposition}
In the case of $C^{1,\alpha}$-patch with symmetric corners, we have: 
\begin{proposition}\label{prop:LWP2}
	Assume that $\Omega_0$ is a $C^{1,\alpha}$-patch with symmetric corners satisfying Definition \ref{def:C_1alpha_corner}. Then there exists some $T > 0$, depending only on its initial $C^{1,\alpha}$-characteristic  $\Gamma_0$, such that $\sup_{t \in [0,T]} \Gamma_t < + \infty$: that is,
	the associated flow map $\Phi \in X$ on the time interval $[0,T]$ provided by Proposition \ref{prop:LWP} satisfies $\Phi_t(x,f_0(x)), \Phi_t(x,g_0(x)) \in C^{1,\alpha}_x[0,\delta_0]$ uniformly in $t\in [0,T]$. In particular, we can extend the solution beyond some $T^* > 0$ as long as $\sup_{t \in [0,T^*)}\Gamma_t < + \infty$. 
\end{proposition}

The proof is a direct consequence of the following estimates: 

\begin{lemma}\label{lem:iteration_bound}
	For any initial data $\Omega_0$ satisfying Definition \ref{def:level_set_circle}, there exists some $M, T > 0$ depending only on $\mathring{\Gamma}_0$ so that $F$ maps the space $B(M,T)$ to itself. 
\end{lemma}

\begin{lemma}\label{lem:iteration_lipschitz}
	Assume that we are in the situation where Lemma \ref{lem:iteration_bound} holds. Then, there exists some $0 < T_1 \le T$, depending only on $M$ and $\mathring{\Gamma}_0$, so that for any $\Phi, \tilde{\Phi} \in B(M,T)$, \begin{equation*}
	\begin{split}
	&\V F(\Phi)(t) - F(\tilde{\Phi})(t) \V_{L^\infty(\overline{\Omega}_0)} \\
	&\qquad\le \mathfrak{C}(M) \int_0^t  \V \Phi_s - \tilde{\Phi}_s\V_{L^\infty(\overline{\Omega}_0)} \left( 1 + \log\left( 1 + \V \Phi_s - \tilde{\Phi}_s\V_{L^\infty(\overline{\Omega}_0)} \right) \right)ds
	\end{split}
	\end{equation*} and \begin{equation*}
	\begin{split}
	&\V \nabla F(\Phi)(t) - \nabla F(\tilde{\Phi})(t) \V_{L^\infty(\overline{\Omega}_0)} \\
	&\qquad\le \mathfrak{C}(M) \int_0^t  \V \nabla\Phi_s - \nabla\tilde{\Phi}_s \V_{L^\infty(\overline{\Omega}_0)} \left( 1 + \log\left( 1 + \V \nabla\Phi_s - \nabla\tilde{\Phi}_s \V_{L^\infty(\overline{\Omega}_0)}\right)  \right)  ds 
	\end{split}
	\end{equation*} hold for any $t \in [0,T_1]$. 
\end{lemma}

Assuming the statements of Lemmas \ref{lem:iteration_bound} and \ref{lem:iteration_lipschitz}, let us just provide a sketch of the proof, as the argument is parallel to \cite[Proof of Theorem 4.1]{Hu}. 

\begin{proof}[Proof of Proposition \ref{prop:LWP}]
	Take $M$ and $T_1$ such that the map $F $ sends $B(M,T_1)$ to itself, and moreover, for any $\Phi, \tilde{\Phi} \in B(M,T_1)$, \begin{equation*}
	\begin{split}
	&\V F(\Phi)(t) - F(\tilde{\Phi})(t) \V_{W^{1,\infty}(\overline{\Omega}_0)} \\
	&\qquad \le C(M) \int_0^t \V \Phi(s) - \tilde{\Phi}(s)\V_{W^{1,\infty}(\overline{\Omega}_0)} \left(  1 + \log\left(1 + \V \Phi(s) - \tilde{\Phi}(s)\V_{W^{1,\infty}(\overline{\Omega}_0)}\right)\right)ds
	\end{split}
	\end{equation*} for $t \in [0,T_1]$. Here, $M$ and $T_1$ depends only on $\mathring{\Gamma}_0$. Define a sequence $\{ \Phi_n \}_{n \ge 0}$ in $B(M,T_1)$ by \begin{equation*}
	\begin{split}
	\Phi_0(t,x) = x, \qquad \Phi_{n+1}(t,x) = F(\Phi_n)(t,x), \quad n \ge 0. 
	\end{split}
	\end{equation*} It is straightforward to see that at each step of the iteration, the flow is $m$-fold symmetric around the origin and therefore $\Phi_n(t,0) = 0$ for all $n \ge 0$. Setting \begin{equation*}
	\begin{split}
	\rho_n(t) := \V \Phi_{n+1}(t) - \Phi_n(t) \V_{W^{1,\infty}(\overline{\Omega}_0)},
	\end{split}
	\end{equation*} we have \begin{equation*}
	\begin{split}
	\rho_n(t) \le \mathfrak{C}(M) \int_0^t \rho_{n-1}(s)\left( 1  + \log\left(1 + \rho_{n-1}(s)\right)\right) ds. 
	\end{split}
	\end{equation*} This is sufficient to deduce that, taking a smaller value of $T_1$ depending only on $M$ if necessary (see \cite[Chapter 2]{MP} for instance), there exists a function $\Phi:[0,T_1]\times \Omega_0  \rightarrow \mathbb{R}^2$ such that \begin{equation*}
	\begin{split}
	\V \Phi_n - \Phi \V_{L^\infty([0,T_1];W^{1,\infty}(\overline{\Omega}_0))} \rightarrow 0. 
	\end{split}
	\end{equation*} At this point, it is easy to see that $\Phi$ actually belongs to $B(M,T_1)$ and $F(\Phi) = \Phi$. Therefore, we have constructed a solution, which belongs to the desired class, to the 2D Euler equation with initial data $\Omega_0$ on the time interval $[0,T_1]$. 
	
	We briefly comment on the issue of continuing the solution past $T_1$. (All the details can be found in \cite{Hu}.) Take $\Omega_{T_1}$ as the new initial data, which has associated level set function $\phi_{T_1}$ with its characteristic $\mathring{\Gamma}_{T_1}$. Going through the exact same iteration scheme again with this new data, one obtains a unique solution on some time interval $[0,T_2]$, with $T_2 = T_2(\mathring{\Gamma}_{T_1}) > 0$. 
	Then, by putting this solution together with the previous one, we obtain a patch solution, admitting a $\mathring{C}^{1,\alpha}$-level set, to the 2D Euler equation on the time interval $[0, T_1 + T_2]$ with initial data $\Omega_0$. This procedure can go on as long as we have a bound on $\mathring{\Gamma}_t$. This finishes the proof. 
\end{proof}

\begin{proof}[Proof of Proposition \ref{prop:LWP2}]
	The assumptions given in Definition \ref{def:C_1alpha_corner} are strictly stronger than the ones in Definition \ref{def:level_set_circle}, so we may work inside the time interval within which we have available the iterates $\Phi_n$ and the limit $\Phi$ belonging to the class $X$, defined in the above proof of Proposition \ref{prop:LWP}. It suffices to carry the information that, by shrinking $T$ if necessary in a way only depending on $\Gamma_0$, for some time interval $[0,T]$, each of $\Phi_n$ satisfies following the H\"older estimate uniformly in $n$:  \begin{equation*}
	\begin{split}
	\V \Phi_n(t,(x,f_0(x))) \V_{C^{1,\alpha}[0,\delta_0]} + \V \Phi_n(t,(x,g_0(x))) \V_{C^{1,\alpha}[0,\delta_0]} \le C(\Gamma_0) < \infty.
	\end{split}
	\end{equation*} This follows directly from the a priori estimates given in the proof of Theorem \ref{thm:symmetric_corner}. It is not difficult to see that $\Phi$ inherits the same H\"older estimate. 
\end{proof}

The lemmas \ref{lem:iteration_bound}, \ref{lem:iteration_lipschitz}, and the bound \eqref{eq:level_set_bound_via_flow} are direct consequences of the following simple lemmas. The first one provides substitutes for the usual calculus inequalities on $C^\alpha$-spaces. 

\begin{lemma}\label{lem:calculus}
	Let $f$ and $g$ be $\mathring{C}^\alpha$ functions on some domain $\Omega \subset \mathbb{R}^2$. Then we have \begin{equation}\label{eq:product}
	\begin{split}
	\V f g \V_{\mathring{C}^\alpha} \le C \left( \V f \V_{\mathring{C}^\alpha} \cdot \V g \V_{L^\infty} +  \V f \V_{L^\infty} \cdot \V g \V_{\mathring{C}^\alpha} \right)
	\end{split}
	\end{equation} and if we assume further that $|f| > 0$ on $\Omega$, \begin{equation}\label{eq:division}
	\begin{split}
	\V 1/f \V_{\mathring{C}^\alpha} \le C(\V f \V_{\inf(\Omega)}) \V f \V_{\mathring{C}^\alpha}.
	\end{split}
	\end{equation} Moreover, if $\Psi$ is a Lipschitz diffeomorphism of $\mathbb{R}^2$ with $\Psi(0) = 0$, then \begin{equation}\label{eq:composition}
	\begin{split}
	\V f \circ \Psi\V_{\mathring{C}^\alpha} \le C\left( \V \nabla\Psi\V_{L^\infty}, \V \nabla \Psi \V_{\inf} \right) \V f \V_{\mathring{C}^\alpha}.
	\end{split}
	\end{equation}
\end{lemma}

\begin{proof}
	Let us note first that for two points at comparable distance, i.e. if $x \ne x'$ satisfy $c_1|x'| \le |x| \le c_2|x'|$, \begin{equation*}
	\begin{split}
	\frac{|f(x) - f(x')|}{|x-x'|^\alpha} \le C \frac{\V f \V_{\mathring{C}^\alpha}}{|x|^\alpha}
	\end{split}
	\end{equation*} with $C$ depending on $c_1, c_2$. 
	
	We begin with \eqref{eq:product}. First, we have an $L^\infty$-bound $\V fg\V_{L^\infty} \le \V f \V_{L^\infty} \cdot \V g \V_{L^\infty}$. Now take two points $x \ne x' \in \Omega$ and assume without loss of generality that $|x| \ge |x'|$. Consider two cases, (i) $|x-x'| \le |x|/2$ and (ii) $|x - x'| > |x|/2$. In the latter case, \begin{equation*}
	\begin{split}
	\frac{|x|^\alpha f(x)g(x) - |x'|^\alpha f(x')g(x')}{|x-x'|^\alpha} &= \frac{|x|^\alpha \left(f(x)g(x) - f(x')g(x') \right) + \left( |x|^\alpha - |x'|^\alpha \right) f(x')g(x')}{|x-x'|^\alpha} \\
	&\le C \V f\V_{L^\infty} \cdot \V g \V_{L^\infty}. 
	\end{split}
	\end{equation*} Next, when (i) holds, we rewrite \begin{equation*}
	\begin{split}
	&\frac{|x|^\alpha \left(f(x)g(x) - f(x')g(x') \right) + \left( |x|^\alpha - |x'|^\alpha \right) f(x')g(x')}{|x-x'|^\alpha} \\
	&\qquad= \frac{|x|^\alpha (f(x) - f(x'))g(x)}{|x-x'|^\alpha} + \frac{|x|^\alpha f(x')(g(x) - g(x'))}{|x-x'|^\alpha} + \frac{|x|^\alpha - |x'|^\alpha}{|x-x'|^\alpha}f(x')g(x'),
	\end{split}
	\end{equation*} which is bounded in absolute value by the right hand side of \eqref{eq:product}, noting that \begin{equation*}
	\begin{split}
	|f(x) - f(x')| \le C \V f \V_{\mathring{C}^\alpha} \frac{|x-x'|^\alpha}{|x|^\alpha}
	\end{split}
	\end{equation*} whenever $|x  -x'| \le |x|/2$. The proof of \eqref{eq:division} is strictly analogous, so let us omit it. 
	
	To show the last statement \eqref{eq:composition}, it suffices to treat the case when $|x'| \le |x|$ and $|x - x'| \le |x|/2$. Moreover, it suffices to bound the quantity \begin{equation*}
	\begin{split}
	|x|^\alpha \frac{|f(\Psi(x)) - f (\Psi(x'))|}{|x-x'|^\alpha} &= |x|^\alpha \frac{|f(\Psi(x)) - f (\Psi(x'))|}{|\Psi(x)-\Psi(x')|^\alpha} \cdot \frac{|\Psi(x)-\Psi(x')|^\alpha}{|x-x'|^\alpha}.
	\end{split}
	\end{equation*}  Note that since $\Psi(0) = 0$, \begin{equation*}
	\begin{split}
	\V \nabla\Psi\V_{\inf} \le \frac{|\Psi(z)|}{|z|} \le \V \nabla\Psi\V_{L^\infty}
	\end{split}
	\end{equation*} for any $z$, and since we have $|x'| \le |x| \le 2|x'|$, there exists some constants $c_1, c_2 > 0$ so that \begin{equation*}
	\begin{split}
	c_1 |\Psi(x')| \le |\Psi(x)| \le c_2 |\Psi(x')|. 
	\end{split}
	\end{equation*} This allows us to bound \begin{equation*}
	\begin{split}
	|x|^\alpha \frac{|f(\Psi(x)) - f (\Psi(x'))|}{|\Psi(x)-\Psi(x')|^\alpha} \cdot \frac{|\Psi(x)-\Psi(x')|^\alpha}{|x-x'|^\alpha} &\le C\V f \V_{\mathring{C}^\alpha} \cdot \frac{|x|^\alpha}{|\Psi(x)|^\alpha} \cdot \left( \V \nabla \Psi\V_{L^\infty} \right)^\alpha.
	\end{split}
	\end{equation*} This finishes the proof. 
\end{proof}

Next, we shall need the piece of information that in the setting of Proposition \ref{prop:LWP}, for each fixed time $t$, the velocity gradient $ \nabla u_t$ actually belongs to $\mathring{C}^\alpha(\overline{\Omega}_t)$. In the case of $C^{1,\alpha}$-patches, this is a direct consequence of velocity being $C^{1,\alpha}$ on the boundary, since then $\Delta u_t = 0$ in $\Omega$ and hence an elliptic regularity statement applies. It is likely that such an argument could be used here, but let us adopt the approach of Serfati \cite{Ser1} (see also recent papers by Bae and Kelliher \cite{BK1}, \cite{BK2}):

\begin{lemma}\label{lem:velocity_circle}
	Let $W $ be a vector field on a domain $\Omega$ with components in $\mathring{C}^\alpha(\overline{\Omega})$. Assume further that $|W| \ge c_0 > 0$ on $\Omega$. Then, for $\omega = \chi_{\Omega}$, the associated velocity satisfies \begin{equation*}
	\begin{split}
	\V \nabla u\V_{\mathring{C}^\alpha(\overline\Omega)} \le C(c_0) \V W\cdot\nabla u\V_{\mathring{C}^\alpha(\overline\Omega)}. 
	\end{split}
	\end{equation*}
\end{lemma}

\begin{proof}
	With $W = (W_1,W_2)$ and $u = (u_1,u_2)$, one computes that \begin{equation*}
	\begin{split}
	\begin{pmatrix}
	\partial_1 u_1 \\
	\partial_2 u_1
	\end{pmatrix} = \frac{1}{|W|^2} \begin{pmatrix}
	W_1 & -W_2 \\
	W_2 & W_1
	\end{pmatrix} \begin{pmatrix}
	W_1 \pr_1 u_1 + W_2 \pr_2 u_1 \\
	W_1 \pr_2 u_1 - W_2 \pr_1 u_1 
	\end{pmatrix}
	\end{split}
	\end{equation*} and note that using $\pr_1 u_1 + \pr_2 u_2 = 0$ as well as $\pr_1u_2-\pr_2 u_1 = \omega \equiv \mathrm{constant}$, \begin{equation*}
	\begin{split}
	W_1 \pr_2 u_1 - W_2 \pr_1 u_1  = W\cdot\nabla u_2 -  W_1 \omega,
	\end{split}
	\end{equation*} so that using \eqref{eq:product} and \eqref{eq:division}, we conclude that $\nabla u_1 \in \mathring{C}^\alpha$. It follows that $\nabla u_2 \in \mathring{C}^\alpha$ as well. 
\end{proof}

\begin{remark}
	 To apply the above lemma to the setting of Proposition \ref{prop:LWP}, taking $W_0 := \nabla^\perp\phi_0$ is strictly speaking not allowed since it may vanish at some points in the interior of the initial patch $\Omega_0$. This can be simply fixed as follows (see \cite[Section 10]{BK1}). First, we know that for points $x \in \Omega$ with $d(x,\partial\Omega_0) < \delta|x|$, $|\nabla^\perp\phi_0|$ is bounded from below with a constant uniform in $|x|$, where $\delta$ can be taken as $1/(10 \mathring{\Gamma}_0)$, for instance. Then it suffices to take a vector field $\tilde{W}_0$ which does not vanish for points $x \in \Omega_0$ with $d(x,\partial\Omega_0) \ge  \delta|x|$. It is easy to require in addition that $\tilde{W}_0$ vanishes on $\partial\Omega_0$ and $\nabla \cdot \tilde{W}_0 \in \mathring{C}^\alpha(\mathbb{R}^2)$.\footnote{To construct such a vector field, one first considers the family of annuli $A_n = \{ x \in \mathbb{R}^2 :  2^{-n-1} < |x| < 2^{-n+1} \}$. By rescaling the region $A_n \cap \Omega$ to a domain of size $O(1)$, we obtain a region with boundary in $C^{1,\alpha}$. Then in this rescaled subset of the annulus one constructs easily a vector field in $C^\alpha$ with desired properties. Rescaling it back, and patching all the vector fields together finishes the construction of $\tilde{W}_0$.} Then, we evolve the vector field by \begin{equation*}
	 \begin{split}
	 \tilde{W}(t,\Phi(t,x)) := (\tilde{W}_0(x) \cdot \nabla) \Phi(t,x),
	 \end{split}
	 \end{equation*} which is consistent with the evolution of vector fields having the form $\nabla^\perp\phi$ for some scalar function $\phi$ advected by the flow. 
\end{remark}

\begin{proof}[Proof of Lemma \ref{lem:iteration_bound}]
	Given an initial vortex patch $\Omega_0$ satisfying conditions of Proposition \ref{prop:LWP}, we fix a vector field $\tilde{W}_0$ described in the remark following Lemma \ref{lem:velocity_circle}, as well as the level set $\phi_0$. Then, one may fix a vector field $W_0$ which coincides with $\nabla^\perp\phi_0$ near $\partial\Omega_0$ and with $\tilde{W}_0$ in a region where $\nabla^\perp\phi_0$ vanishes. We have $\nabla\cdot W_0 \in \mathring{C}^\alpha$. 
	
	We have \begin{equation*}
	\begin{split}
	F(\Phi)(t,x) = x + \int_0^t u(s,\Phi(s,x)) ds
	\end{split}
	\end{equation*} as well as \begin{equation*}
	\begin{split}
	(\nabla F(\Phi))(t,x) = I + \int_0^t \nabla u(s,\Phi(s,x)) \nabla\Phi(s,x) ds. 
	\end{split}
	\end{equation*} We claim that the push-forward of the vector field $W_0$ (recall that ${W}(t,\Phi(t,x)) := ({W}_0(x) \cdot \nabla) \Phi(t,x)$) satisfies $\sup_{t \in [0,T]} \V W_t\V_{\mathring{C}^\alpha(\overline{\Omega}_t)} \le \mathfrak{C}(M)$ as well as $\inf_{t \in [0,T]} \V W_t\V_{\inf(\overline{\Omega}_t)} \ge (\mathfrak{C}(M))^{-1} > 0$ (see \cite{BK1}, \cite{BK2} for complete details of this proof in the context of $C^\alpha$ vector fields -- the proof can be adapted to our setting with straightforward modifications). It then follows from Lemma \ref{lem:velocity_circle} that \begin{equation*}
	\begin{split}
	\V \nabla u\V_{\mathring{C}^\alpha(\overline{\Omega}_t)} \le \mathfrak{C}(M). 
	\end{split}
	\end{equation*} Then, using the inequalities from \ref{lem:calculus} we immediately obtain \begin{equation*}
	\begin{split}
	\V \nabla F(\Phi) \V_{\mathring{C}^\alpha}  \le \mathfrak{C}(M)T
	\end{split}
	\end{equation*} and also \begin{equation*}
	\begin{split}
	\sup_{\overline{\Omega}_0 \times [0,T]}| \nabla F(\Phi) - I| \le \mathfrak{C}(M)T. 
	\end{split}
	\end{equation*} Taking $T$ sufficiently small, we see that $F(\Phi) \in B(M,T)$. 
\end{proof}

Finally, we give a sketch of the proof of Lemma \ref{lem:iteration_lipschitz}.

\begin{proof}[Proof of Lemma \ref{lem:iteration_lipschitz}]
	Fix some $x \in \Omega_0$ and $t \in [0, T_1]$, and let us first obtain a bound on $| F(\Phi)(t,x) - F(\tilde{\Phi})(t,x)|$. We need to estimate \begin{equation}\label{eq:Quantity}
	\begin{split}
	\int_{\Omega_0} \left| K(\Phi(s,x) - \Phi(s,z)) - K(\tilde{\Phi}(s,x) - \tilde{\Phi}(s,z)) \right| dz 
	\end{split}
	\end{equation} for each $s \in [0,t]$. We split the integral: when $|z-x| > \epsilon$, we have \begin{equation*}
	\begin{split}
	&\int_{\Omega_0 \backslash B_\epsilon(x)} \left| K(\Phi(s,x) - \Phi(s,z)) - K(\tilde{\Phi}(s,x) - \tilde{\Phi}(s,z)) \right| dz  \\ &\qquad\le \mathfrak{C}(M) \int_{\Omega_0 \backslash B_\epsilon(x)} \V \Phi(s) - \tilde{\Phi}(s)\V_{L^\infty}  \cdot \frac{1}{|x-z|^2} dz  \le \mathfrak{C}(M) \V \Phi(s) - \tilde{\Phi}(s)\V_{L^\infty}  \left(1 + |\log(\epsilon)| \right),
	\end{split}
	\end{equation*} whereas \begin{equation*}
	\begin{split}
	\int_{\Omega_0 \cap B_\epsilon(x)} \left| K(\Phi(s,x) - \Phi(s,z)) - K(\tilde{\Phi}(s,x) - \tilde{\Phi}(s,z)) \right| dz \le \mathfrak{C}(M) \int_{\Omega_0 \cap B_\epsilon(x)} \frac{1}{|x-z|}dz \le \mathfrak{C}(M)\epsilon. 
	\end{split}
	\end{equation*} We have used the following elementary inequality: \begin{equation*}
	\begin{split}
	|K(a) - K(b)| \le C|a-b| \left( \frac{1}{|a|^2} + \frac{1}{|b|^2} \right).
	\end{split}
	\end{equation*} Choosing $\epsilon =  \V \Phi(s) - \tilde{\Phi}(s)\V_{L^\infty} $ establishes the desired inequality (assuming that the latter quantity is non-zero -- otherwise the result is trivial). 
	
	Turning to the next inequality, one sees that the key is to obtain a bound on the following integral: \begin{equation*}
	\begin{split}
	\int_{\Omega_0} \left| \nabla K(\Phi(s,x) - \Phi(s,z)) - \nabla K(\tilde{\Phi}(s,x) - \tilde{\Phi}(s,z)) \right| dz,
	\end{split}
	\end{equation*} modulo the terms which are trivially bounded by $\mathfrak{C}(M) \V \nabla\Phi - \nabla\tilde{\Phi}\V_{L^\infty}$. 
	
	To begin with, take some constant $\epsilon_0 > 0$ (depending only on $\Omega_0$) with the property that, for any $x \in \partial\Omega_0$, there is an open ball of radius $4\epsilon_0|x|$ contained in $\Omega_0$ and whose boundary contains $x$. Now let us take some $\epsilon < \epsilon_0$, whose value will be determined later. We shall consider two cases: (i) $d(x,\partial\Omega_0) > 2\epsilon|x|$, (ii) $d(x,\partial\Omega_0) \le 2\epsilon|x|$. 
	
	When (i) holds, let us split the integral as \begin{equation*}
	\begin{split}
	\int_{\Omega_0 \backslash B_{\epsilon|x|}(x)} + \int_{\Omega_0 \cap B_{\epsilon|x|}(x)} \left| \nabla K(\Phi(s,x) - \Phi(s,z)) - \nabla K(\tilde{\Phi}(s,x) - \tilde{\Phi}(s,z)) \right| dz,
	\end{split}
	\end{equation*} and in the former region, we further decompose into regions where $\epsilon|x| < |z-x| \le 10|x|$ and $10|x| < |z-x|$. Then, in the case $\epsilon|x| < |z-x| \le 10|x|$, using the mean value theorem with the decay of $\nabla\nabla K$ gives a bound \begin{equation*}
	\begin{split}
	\mathfrak{C}(M) \V \nabla\Phi - \nabla\tilde{\Phi}\V_{L^\infty}(1 + |\log(\epsilon)|). 
	\end{split}
	\end{equation*} Then, when $10|x| < |z-x|$ holds, one first symmetrizes the kernel to gain extra decay and then use the mean value theorem to obtain \begin{equation*}
	\begin{split}
	\mathfrak{C}(M) \V \nabla\Phi - \nabla\tilde{\Phi}\V_{L^\infty}.
	\end{split}
	\end{equation*} In the latter region, the integral is bounded by \begin{equation*}
	\begin{split}
	\int_{\Sigma} |\nabla K(z) dz| \le C \int_{\Sigma} |z|^{-2}dz,
	\end{split}
	\end{equation*} where \begin{equation*}
	\begin{split}
	\Sigma &= ( \Phi(s,B_{\epsilon|x|}(x)) - \Phi(s,x) ) \Delta ( \tilde{\Phi}(s,B_{\epsilon|x|}(x)) - \tilde{\Phi}(s,x) ) \\
	&:= \{ y - \Phi(s,x) : y \in \Phi(s,B_{\epsilon|x|}(x))  \} \Delta \{ y - \tilde{\Phi}(s,x) : y \in \tilde{\Phi}(s,B_{\epsilon|x|}(x)) \}. 
	\end{split}
	\end{equation*} For any unit vector $\omega$, define \begin{equation*}
	\begin{split}
	r_1(\omega) = \min\{ r > 0 : r\omega \in \Sigma \}, \qquad r_2(\omega) = \max\{ r > 0 : r\omega \in \Sigma\}.
	\end{split}
	\end{equation*} Then, the claim of Huang \cite[(4.20) on p. 531]{Hu} translates in our setting to give that (after the usual scaling argument in $|x|$) \begin{equation*}
	\begin{split}
	r_1(\omega) \ge (\mathfrak{C}(M))^{-1} \epsilon|x|, \qquad r_2(\omega) \le \mathfrak{C}(M) \epsilon|x| \left( \epsilon^\alpha + \V \nabla\Phi - \nabla\tilde{\Phi} \V_{L^\infty} \right).
	\end{split}
	\end{equation*}   Using these bounds, we integrate \begin{equation*}
	\begin{split}
	\int_{\Sigma} |\nabla K(z) dz| &\le \mathfrak{C}(M) \int_{\partial B_1(0)} \int_{r_1(\omega)}^{r_2(\omega)} \frac{1}{r} dr d\omega \le \mathfrak{C}(M) \int_{\partial B_1(0)} \log\left( 1 + \frac{r_2(\omega) - r_1(\omega)}{r_1(\omega)}  \right) d\omega \\
	&\le \mathfrak{C}(M) \int_{\partial B_1(0)} \frac{r_2(\omega) - r_1(\omega)}{r_1(\omega)}d\omega \le \mathfrak{C}(M)(\epsilon^\alpha +  \V \nabla\Phi - \nabla\tilde{\Phi} \V_{L^\infty} ).
	\end{split}
	\end{equation*} We have established the desired bound on \eqref{eq:Quantity}, and it follows immediately that \begin{equation*}
	\begin{split}
	| \nabla u(s,\Phi(s,x)) - \nabla\tilde{u}(s,\tilde{\Phi}(s,x)) | \le \mathfrak{C}(M)(\epsilon^\alpha +  \V \nabla\Phi - \nabla\tilde{\Phi} \V_{L^\infty} \left( 1 + |\log(\epsilon)| \right) )
	\end{split}
	\end{equation*} when (i) holds, and with $\epsilon < \epsilon_0$. 
	
	Now, when (ii) holds for $x \in \Omega_0$, we can select (by the assumption on $\epsilon_0$) a point $y \in \Omega_0$, such that $d(y,\partial\Omega_0) \ge 2\epsilon|x|$ and $|x-y| \le 2\epsilon|x|$. Then, \begin{equation*}
	\begin{split}
	|\nabla u(s,\Phi(s,x)) - \nabla\tilde{u}(s,\tilde{\Phi}(s,x))| &\le |\nabla u(s,\Phi(s,x)) - \nabla u(s,\Phi(s,y))| \\
	&\qquad + |\nabla u(s,\Phi(s,y)) - \nabla\tilde{u}(s,\tilde{\Phi}(s,y))| \\
	&\qquad + |\nabla\tilde{u}(s,\tilde{\Phi}(s,y)) - \nabla\tilde{u}(s,\tilde{\Phi}(s,x))|\\
	&\le \mathfrak{C}(M) (\epsilon^\alpha +  \V \nabla\Phi - \nabla\tilde{\Phi} \V_{L^\infty} \left( 1 + |\log(\epsilon)| \right) ) + \mathfrak{C}(M) \epsilon^\alpha,
	\end{split}
	\end{equation*} where we have used that $\nabla u, \nabla \tilde{u} \in \mathring{C}^\alpha$: \begin{equation*}
	\begin{split}
	|\nabla u(s,\Phi(s,x)) - \nabla u(s,\Phi(s,y))| \le \mathfrak{C}(M) \frac{|\Phi(s,x) - \Phi(s,y)|^\alpha}{|\Phi(s,x)|^\alpha} \le \mathfrak{C}(M) \V\nabla \Phi\V_{L^\infty}^\alpha \cdot \frac{|x-y|^\alpha}{|x|^\alpha} \le \mathfrak{C}(M) \epsilon^\alpha,
	\end{split}
	\end{equation*} and similarly for the other term. 
	
	At this point, observe that \begin{equation*}
	\begin{split}
	\frac{d}{dt} \left|  \nabla\Phi(t,x) - \nabla\tilde{\Phi}(t,x)  \right| \le \mathfrak{C}(M),
	\end{split}
	\end{equation*} so that \begin{equation*}
	\begin{split}
	\V \nabla\Phi(t,\cdot) - \nabla\tilde{\Phi}(t,\cdot) \V_{L^\infty} \le \mathfrak{C}(M)t,
	\end{split}
	\end{equation*} and therefore by taking $T_1$ sufficiently small, relative to $M$ and $\Omega_0$, it can be assumed that \begin{equation*}
	\begin{split}
	\sup_{t \in [0,T_1]} \V \nabla\Phi - \nabla\tilde{\Phi} \V_{L^\infty} \le \frac{1}{10} \epsilon_0^\alpha.
	\end{split}
	\end{equation*} Now we may take $\epsilon^\alpha = \V \nabla\Phi(\cdot,s) - \nabla\tilde{\Phi}(\cdot,s) \V_{L^\infty}$ for each $s  \in [0,T_1]$ (or just a sufficiently small constant when the latter is zero). This finishes the proof. 
\end{proof}

In the course of the above local well-posedness proof, we needed to prove that the flow maps having regularity $\nabla\Phi \in \mathring{C}^\alpha$ implies that the corresponding velocity gradient satisfies $\nabla u \in \mathring{C}^\alpha$. For completeness we show that the converse also holds. 

\begin{proposition}
	Let $u$ be a vector field with regularity $\nabla u \in L^\infty([0,T);\mathring{C}^{\alpha}(\mathbb{R}^2))$ for some $0 < \alpha \le 1$ and satisfy $u(t,0) = 0$ for all $t \in [0,T)$. Then the associated flow map $\Phi$ satisfies \begin{equation*}
	\begin{split}
	\V \nabla\Phi(t)\V_{\mathring{C}^\alpha(\mathbb{R}^2)} \le M = M(t, \sup_{s \in [0,t]}\V \nabla u(s)\V_{\mathring{C}^\alpha(\mathbb{R}^2)} ) . 
	\end{split}
	\end{equation*}
\end{proposition}

\begin{proof}
	Since the velocity is Lipschitz, there is a unique solution to \begin{equation*}
	\begin{split}
	\frac{d}{dt} \Phi(t,a) = u(t, \Phi(t,a)), \qquad \Phi(0,a) = a,
	\end{split}
	\end{equation*} which defines the flow map $\Phi(t,\cdot) : \mathbb{R}^2 \rightarrow \mathbb{R}^2 $ for each $t \in [0,T)$. Clearly $\Phi(t,0) = 0$.

Taking two points $a \ne b$, we obtain that \begin{equation*}
\begin{split}
\left| \frac{d}{dt}\frac{|\Phi(t,a) - \Phi(t,b)|}{|a-b|}  \right| \le \V \nabla u(t,\cdot) \V_{L^\infty(D)} \frac{|\Phi(t,a) - \Phi(t,b)|}{|a-b|} ,
\end{split}
\end{equation*} and therefore, by integrating in time, we obtain \begin{equation*}
\begin{split}
\exp\left( -\int_0^t  \V \nabla u(s,\cdot) \V_{L^\infty(D)}ds   \right) \le \frac{|\Phi(t,a) - \Phi(t,b)|}{|a-b|} \le \exp\left( \int_0^t  \V \nabla u(s,\cdot) \V_{L^\infty(D)}ds   \right).
\end{split}
\end{equation*} 

We now proceed to obtain $\mathring{C}^\alpha$-estimates for the gradient of the flow. We know that $\Phi(t,\cdot)$ is differentiable almost everywhere, and it is straightforward to show that the gradient (defined almost everywhere) satisfies \begin{equation*}
\begin{split}
\frac{d}{dt} \nabla\Phi(t,a) = \nabla u(t,\Phi(t,a)) \nabla\Phi(t,a)
\end{split}
\end{equation*} for almost every $a \in D$. For two points $a, b \in D$, we write \begin{equation*}
\begin{split}
\frac{d}{dt} \left( |a|^\alpha\nabla\Phi(t,a) - |b|^\alpha\nabla\Phi(t,b) \right) &= \left[ |a|^\alpha \nabla u(t,\Phi(t,a)) - |b|^\alpha \nabla u(t,\Phi(t,b)) \right] \nabla \Phi(t,a) \\
&\qquad - \nabla u(t,\Phi(t,b)) \left[ |a|^\alpha \nabla\Phi(t,a) - |b|^\alpha \nabla\Phi(t,b) \right] \\
&\qquad + \left( |a|^\alpha - |b|^\alpha \right) \nabla u(t,\Phi(t,b)) \nabla\Phi(t,a).
\end{split}
\end{equation*} We further write \begin{equation*}
\begin{split}
|a|^\alpha \nabla u(t,\Phi(t,a)) - |b|^\alpha \nabla u(t,\Phi(t,b)) &= \left( \frac{|a|}{|\Phi(t,a)|} \right)^\alpha \left( |\Phi(t,a)|^\alpha \nabla u(t,\Phi(t,a)) - |\Phi(t,b)|^\alpha \nabla u (t,\Phi(t,b)) \right) \\
&\qquad + \left[ |a|^\alpha - |b|^\alpha  + \frac{(|\Phi(t,b)|^\alpha - |\Phi(t,a)|^\alpha)|a|^\alpha}{ |\Phi(t,a)|^\alpha} \right] \nabla u(t,\Phi(t,b)).
\end{split}
\end{equation*} 
Hence, \begin{equation*}
\begin{split}
&\left| \frac{d}{dt} \frac{\left|  |a|^\alpha\nabla\Phi(t,a) - |b|^\alpha\nabla\Phi(t,b)\right| }{|a- b|^\alpha} \right| \\
&\qquad \le \left( \frac{|a|}{|\Phi(t,a)|} \right)^\alpha \frac{ |\Phi(t,a)|^\alpha \nabla u(t,\Phi(t,a)) - |\Phi(t,b)|^\alpha \nabla u (t,\Phi(t,b))| }{|\Phi(t,a) - \Phi(t,b)|^\alpha} \Vert \nabla \Phi(t,\cdot)\Vert_{L^\infty}^{1+\alpha} \\
& \qquad + \left( 1 + \V \nabla \Phi(t,\cdot)\V_{L^\infty} \left( \frac{|a|}{|\Phi(t,a)|} \right)^\alpha  \right) \V \nabla u(t,\cdot) \V_{L^\infty} \\
&\qquad + \V \nabla u(t,\cdot)\V_{L^\infty} \frac{\left|  |a|^\alpha\nabla\Phi(t,a) - |b|^\alpha\nabla\Phi(t,b)\right| }{|a- b|^\alpha} \\
&\qquad + \V \nabla u(t,\cdot)\V_{L^\infty}\V \nabla \Phi(t,\cdot)\V_{L^\infty}.
\end{split}
\end{equation*}
Note that \begin{equation*}
\begin{split}
\frac{|a|}{|\Phi(t,a)|} \le \exp\left( \int_0^t \V \nabla u(s,\cdot) \V_{L^\infty(D)}ds  \right),
\end{split}
\end{equation*} and therefore we have the bound \begin{equation*}
\begin{split}
&\left| \frac{d}{dt} \frac{\left|  |a|^\alpha\nabla\Phi(t,a) - |b|^\alpha\nabla\Phi(t,b)\right| }{|a- b|^\alpha} \right| \\
&\qquad\le \exp\left( C\int_0^t \V \nabla u(s,\cdot) \V_{L^\infty(D)}ds  \right) \left(1 + \V \nabla u(t,\cdot)\V_{\mathring{C}^\alpha} + \frac{\left|  |a|^\alpha\nabla\Phi(t,a) - |b|^\alpha\nabla\Phi(t,b)\right|}{|a-b|^\alpha}  \right).
\end{split}
\end{equation*} Integrating in time using the Gronwall inequality, the proof is complete. \qedhere  

\end{proof}

\bibliographystyle{plain}
\bibliography{thesis}

\end{document}